\documentclass[12pt]{amsart}
\usepackage{amsmath,amsfonts,euscript,amscd,amsthm,amssymb,upref,graphics}
\usepackage{stmaryrd}
\usepackage{mathabx}
\usepackage[all]{xy}
\usepackage{color}

\usepackage{mathrsfs}
\newcommand{\Iscr}{\mathscr{I}}
\newcommand{\Mscr}{\mathscr{M}}
\newcommand{\Oscr}{\mathscr{O}}
\newcommand{\Pscr}{\mathscr{P}}

\newcommand{\Tscr}{\mathscr{T}}
\newcommand{\Xscr}{\mathscr{X}}

\swapnumbers

\theoremstyle{plain}

\newtheorem{theorem}{Theorem}[section]
\newtheorem{proposition}[theorem]{Proposition}
\newtheorem{lemma}[theorem]{Lemma}
\newtheorem{corollary}[theorem]{Corollary}
\newtheorem{nothing}[theorem]{}

\theoremstyle{definition}

\newtheorem{definition}[theorem]{Definition}

\newtheorem{nothing*}[theorem]{}
\newtheorem{example}[theorem]{Example}

\newtheorem{notation}[theorem]{Notation}

\newtheorem{bigremark}[theorem]{Remark}

\newtheorem{subnothing*}[sub]{}

\theoremstyle{remark}

\newtheorem*{remark}{Remark}

\newcommand{\dom}{	\operatorname{\text{\rm dom}}}
\newcommand{\Spec}{	\operatorname{\text{\rm Spec}}}
\newcommand{\haut}{	\operatorname{\text{\rm ht}}}
\newcommand{\supp}{	\operatorname{\text{\rm supp}}}
\newcommand{\image}{	\operatorname{\text{\rm im}}}
\newcommand{\trdeg}{	\operatorname{\text{\rm trdeg}}}
\newcommand{\rank}{	\operatorname{\text{\rm rank}}}
\newcommand{\Frac}{	\operatorname{\text{\rm Frac}}}

\newcommand{\lnd}{	\operatorname{\text{\rm LND}}}
\newcommand{\hlnd}{	\operatorname{\text{\rm HLND}}}
\newcommand{\khlnd}{	\operatorname{\text{\rm KHLND}}}
\newcommand{\ML}{	\operatorname{\text{\rm ML}}}
\newcommand{\HML}{	\operatorname{\text{\rm HML}}}
\newcommand{\Gr}{	\operatorname{\text{\rm Gr}}}
\newcommand{\Proj}{	\operatorname{\text{\rm Proj}}}
\newcommand{\Div}{	\operatorname{\text{\rm Div}}}
\newcommand{\Sing}{	\operatorname{\text{\rm Sing}}}
\renewcommand{\div}{	\operatorname{{\rm div}}}
\newcommand{\lcm}{	\operatorname{\text{\rm lcm}}}
\newcommand{\Gal}{	\operatorname{\text{\rm Gal}}}
\newcommand{\cotype}{	\operatorname{\text{\rm cotype}}}

\newcommand{\EXT}{	\operatorname{\text{\rm EXT}}}
\newcommand{\torc}{	\operatorname{\text{\rm torc}}}

\newcommand{\YY}{Y^{(1)}}
\newcommand{\G}{\operatorname{\bbG}}
\newcommand{\bG}{\operatorname{\bar\bbG}}

\newcommand{\isomdot}{\overset{\text{\tiny$*$}}{\isom}}
\newcommand{\NR}{	\operatorname{\text{\rm NR}}}
\newcommand{\simQ}{\mathrel{\sim_{\Rat}}}

\newcommand{\setspec}[2]{\big\{\,#1\, \mid \,#2\, \big\}}

\newcommand{\ba}{{\mathbf{a}}}

\newcommand{\Integ}{\ensuremath{\mathbb{Z}}}
\newcommand{\Nat}{\ensuremath{\mathbb{N}}}
\newcommand{\Rat}{\ensuremath{\mathbb{Q}}}
\newcommand{\Comp}{\ensuremath{\mathbb{C}}}
\newcommand{\Reals}{\ensuremath{\mathbb{R}}}
\newcommand{\aff}{\ensuremath{\mathbb{A}}}
\newcommand{\proj}{\ensuremath{\mathbb{P}}}
\newcommand{\bk}{{\ensuremath{\rm \bf k}}}
\newcommand{\ck}{\bar{\bk}}
\newcommand{\kk}[1]{\bk^{[{#1}]}}
\newcommand{\bbV}{\ensuremath{\mathbb{V}}}
\newcommand{\bbM}{\ensuremath{\mathbb{M}}}
\newcommand{\bbD}{\ensuremath{\mathbb{D}}}
\newcommand{\bbG}{\ensuremath{\mathbb{G}}}
\newcommand{\bbT}{\ensuremath{\mathbb{T}}}

\newcommand{\pgoth}{{\ensuremath{\mathfrak{p}}}}
\newcommand{\qgoth}{{\ensuremath{\mathfrak{q}}}}
\newcommand{\Pgoth}{{\ensuremath{\mathfrak{P}}}}
\newcommand{\mgoth}{{\ensuremath{\mathfrak{m}}}}

\newcommand{\Beul}{\EuScript{B}}

\newcommand{\Feul}{\EuScript{F}}

\newcommand{\Heul}{\EuScript{H}}

\newcommand{\Oeul}{\EuScript{O}}

\newcommand{\Reul}{\EuScript{R}}

\newcommand{\isom}{\cong}
\renewcommand{\epsilon}{\varepsilon}
\renewcommand{\phi}{\varphi}
\renewcommand{\emptyset}{\varnothing}

\newenvironment{enumerata}%
{\begin{enumerate}

}{\end{enumerate}}

\newenvironment{enumerati}%
{\begin{enumerate}
}{\end{enumerate}}

\newcommand{\rien}[1]{}

\begin{document}
\renewcommand{\baselinestretch}{1.07}


\title[Rigidity of graded domains]{Rigidity of graded integral domains\\ and of their Veronese subrings}

\author{Daniel Daigle}

\address{Department of Mathematics and Statistics\\
University of Ottawa\\
Ottawa, Canada\ \ K1N 6N5}

\email{ddaigle@uottawa.ca}

{\renewcommand{\thefootnote}{}
}

\begin{abstract} 
Let $B = \bigoplus_{i \in G} B_i$ be a commutative integral domain of characteristic $0$ graded by an abelian group $G$.
We say that $B$ is {\it rigid\/} (resp.\ {\it graded-rigid}) if the only locally nilpotent derivation (resp.\ homogeneous locally nilpotent derivation)
of $B$ is the zero derivation.
Given a subgroup $H$ of $G$, define $B^{(H)} = \bigoplus_{i \in H} B_i$.
We give results that answer or partially answer the following questions:
Does non-rigidity of $B$ imply non-rigidity of $B^{(H)}$?
When can a derivation of $B^{(H)}$ be extended to one of $B$?
What are the properties of the set of subgroups $H$ of $G$ such that $B^{(H)}$ is not graded-rigid?
We define the subgroups $\bG(B) \subseteq \G(B)$ of $G$ and find that these are related to the locally nilpotent derivations of $B$ in several interesting ways.
(The definitions of $\bG(B)$ and $\G(B)$ do not involve derivations, and these two groups are usually easy to determine.)
One of our results states that if $B$ is a normal affine $G$-graded domain then $\trdeg(B : \ML(B)) \ge \rank( \G(B)/\bG(B) )$.
We also give a result relating the rigidity of $B_{(x)}$ to that of $B/xB$, where $B$ is an $\Nat$-graded normal affine domain and $x$ is a homogeneous prime element of $B$.
We give some applications to Pham-Brieskorn rings.
\end{abstract}

\maketitle
  
\vfuzz=2pt


\section{Introduction}

Let $B$ be a ring (by which we mean a commutative, associative, unital ring).
A derivation $D : B \to B$ is {\it locally nilpotent} if, for each $b \in B$, there exists $n>0$ such that $D^n(b)=0$.
We write $\lnd(B)$ for the set of locally nilpotent derivations $D : B \to B$.
The ring $B$ is said to be {\it rigid\/} if $\lnd(B) = \{0\}$.

Given a ring $B$ graded by an abelian group $G$,
let $\hlnd(B)$ denote the set of derivations $D \in \lnd(B)$ that are homogeneous with respect to the given grading (see Def.\ \ref{8hnvnBbslijvxbgxsab8AbN7p03msS726rfXC}).
We say that $B$ is {\it graded-rigid\/} if $\hlnd(B)=\{0\}$.
We say that $B$ is {\it rigid\/} if it is rigid as a non-graded ring, i.e., if $\lnd(B)=\{0\}$.
One has:
$$
\text{$B$ is rigid} \overset{\text{\rm def}}{\iff} \lnd(B) = \{0\}
\ \begin{array}{c} \hspace{-3.5mm}\Longleftarrow\hspace{-5mm}\raisebox{.5mm}{\tiny$/$} \\[-2.3mm] \implies \end{array} \ 
\hlnd(B) = \{0\}  \overset{\text{\rm def}}{\iff} \text{$B$ is graded-rigid.}
$$
Also, it is well known that if $\bk$ is a field of characteristic $0$ and $B$ is an affine $\bk$-domain graded by a torsion-free abelian group,
then $B$ is rigid if and only if it is graded-rigid (see Lemma \ref{876543wsdfgxhnm290cAo}).

To give context for the present work, let us recall how rigidity is related to cylinders.
Given a ring $B$, an open subset $U$ of $\Spec B$ is {\it basic\/} if $U = \bbD(f) := \setspec{ \pgoth \in \Spec B }{ f \notin \pgoth }$ for some $f \in B$.
Given an $\Nat$-graded ring $B$, an open subset $U$ of $\Proj B$ is {\it basic\/}
if $U = \bbD_+(f) := \setspec{ \pgoth \in \Proj B }{ f \notin \pgoth }$ for some homogeneous $f \in B$ of nonzero degree.
A {\it cylinder\/} in a variety $V$ is a nonempty open subset of $V$ that is isomorphic to $Z \times \aff^1$ for some variety $Z$.
If the variety $V$ is $\Spec B$ or $\Proj B$ then by a {\it basic cylinder\/} of $V$ we mean a cylinder of $V$ that is also a basic open set.
The following fact is very well known:

\begin{nothing} \label {jbhgfdxfewae33w6q920we}
If $\bk$ is a field of characteristic $0$ and $B$ is an affine $\bk$-domain, then:
$$
\text{$\Spec B$ has a basic cylinder} \iff \text{$B$ is non-rigid.}
$$
\end{nothing}

It is natural to ask if there is an analogous result for $\Proj(B)$.
An affirmative answer was first given in the seminal article \cite{KPZ2013};
specifically, Thm 0.6 and Cor.\ 3.2 of \cite{KPZ2013}, taken together, are analogous to \ref{jbhgfdxfewae33w6q920we}.
Those two results of \cite{KPZ2013} are generalized by \cite[Thm 1.2]{Chitayat-Daigle:cylindrical},
which we state below after giving some definitions.

Let $B = \bigoplus_{n \in \Integ} B_n$ be a $\Integ$-graded domain.
Given $d \in \Nat \setminus \{0\}$,
the ring $B^{(d)} = \bigoplus_{ n \in d\Integ } B_n$ is called the {\it $d$-th Veronese subring of $B$.}
The number $e(B) = \gcd\setspec{ n \in \Integ }{ B_n \neq 0 }$ is called the {\it saturation index\/} of $B$.
Let $Z$ temporarily denote the set of all height $1$ homogeneous prime ideals of $B$, 
and define $\bar e(B) = \lcm\setspec{ e(B/\pgoth) }{ \pgoth \in Z }$ if $Z \neq \emptyset$, and $\bar e(B) = e(B)$ if $Z = \emptyset$.
We call $\bar e(B)$ the {\it codimension $1$ saturation index\/} of $B$.
We have $e(B), \bar e(B) \in \Nat$ and $e(B) \mid \bar e(B)$.
We say that $B$ is {\it saturated in codimension $1$} if $\bar e(B) = e(B)$,
or equivalently,  if  $e( B/\pgoth ) = e( B )$ for all height $1$ homogeneous prime ideals $\pgoth$ of $B$.\footnote{The fact that 
$B$ is saturated in codimension $1$ if and only if $\bar e(B) = e(B)$ is a result (Lemma \ref{983bvkslwlia293tfg62zgbnm3kry}),
not a definition, but there is no harm in using it as a definition in this Introduction.}

The following is a modified version of \cite[Thm 1.2]{Chitayat-Daigle:cylindrical}:

\begin{theorem} \label {edh83yf6r79hvujhxu6wrefji9e} 
Let $\bk$ be a field of characteristic $0$ and $B = \bigoplus_{i \in \Nat}B_i$
an $\Nat$-graded affine $\bk$-domain such that $\trdeg(B:B_0) \ge 2$.
\begin{enumerata}

\item $\Proj B$ has a basic cylinder $\iff$ $\exists\ d \in \Nat \setminus \{0\}$ such that $B^{(d)}$ is non-rigid.

\item If $B$ is normal and saturated in codimension $1$ then the following are equivalent:
\begin{enumerata}

\item $\Proj B$ has a basic cylinder;
\item $B$ is non-rigid;
\item $B^{(d)}$ is non-rigid for some $d \in \Nat\setminus\{0\}$; 
\item $B^{(d)}$ is non-rigid for all $d \in \Nat\setminus\{0\}$.

\end{enumerata}
\end{enumerata}
\end{theorem}

The above statement differs from Thm 1.2 of \cite{Chitayat-Daigle:cylindrical} in two ways.
Firstly, this version is restricted to the case of $\Nat$-gradings, whereas the original version covers $\Integ$-gradings.
Secondly, part (b) of the original version only asserts the equivalence of (i--iii);
the fact that (iii) implies (iv) follows from Cor.\ \ref{1D0x2mv4yhhtfukCs1unc3t46bisgf9f2} of the present article.\footnote{Alternatively,
one can obtain that (i) implies (iv) as follows.
First, we note that if  $B$ is normal and saturated in codimension $1$ then so is $B^{(d)}$ for every $d$
(by Lemma \ref{98767gh11dmdffkll0ahcchvr0vfesh8263543794663cnir} and Cor.\ \ref{mMpPiq13096gxvbw5resip0tfs}).
Secondly, if $\Proj B$ has a basic cylinder then so does $\Proj B^{(d)}$ for every $d$ (because $\Proj B \isom \Proj B^{(d)}$).
So ``(i) implies (iv)'' follows by applying ``(i) implies (ii)'' to the ring $B^{(d)}$.}

Except for Rem.\ \ref{iytd43J45628c93j}, Ex.\ \ref{lBoibhgxfd7ii03A8vevcCxEvv82} and Rem.\ \ref{24799F47-1201-4FF5-A50D-93EB60154D1C},
we will not consider cylinders outside of this introduction.
The purpose of the above discussion is twofold: to help connect the results of this article to existing literature, and to suggest the idea that the set
$$
\NR(B) = \setspec{ d \in \Nat\setminus\{0\} }{ \text{$B^{(d)}$ is non-rigid} } 
$$
(which we define for any $\Integ$-graded domain $B$ of characteristic $0$) is an interesting object of study.
For instance, it is natural to ask if there is an easily computable number $d \in \Nat \setminus \{0\}$ with the property that
$\NR(B) \neq \emptyset$ $\Leftrightarrow$ $d \in \NR(B)$.
Or we can ask if $\NR(B)$ can be an arbitrary subset of $\Nat \setminus \{0\}$, or if it necessarily has some kind of structure.
These questions are the starting point of the present article,
and one of our objectives --- though not the only one --- is to describe the properties of $\NR(B)$.
As an example, consider the $\Nat$-graded normal domain 
$$
B = \Comp[X,Y,Z] / ( X^4 + Y^6 + Z^{10} ) .
$$
One can see (by Thm \ref{eion98h38dhb2d203ojfbdt7}) that $\NR(B) = \Iscr_6 \cup \Iscr_{10} \cup \Iscr_{15}$ where
we define $\Iscr_d = \{d,2d,3d,\dots\}$ for each $d \in \Nat \setminus \{0\}$.
This tells us in particular that $B^{(6)}$, $B^{(10)}$ and $B^{(15)}$ are not rigid, whereas $B$, $B^{(4)}$, $B^{(9)}$ and $B^{(14)}$ (for example) are rigid.

As we already mentioned just after Thm \ref{edh83yf6r79hvujhxu6wrefji9e}, one of our results (Cor.\ \ref{1D0x2mv4yhhtfukCs1unc3t46bisgf9f2})
states that if $B$ is normal and saturated in codimension $1$ then $\NR(B)$ is either $\emptyset$ or $\Nat\setminus\{0\}$.
The above example shows that $\NR(B)$ can be complicated when $B$ is not saturated in codimension $1$.
Cor.\ \ref{p98qy386e523brxhbfc6ghvc63x2345618q30} and Prop.\ \ref{finiteprimitiveset534823i93i} give a partial description of $\NR(B)$ under mild assumptions.

The results from Section \ref{sectionInteggradings} that we mentioned so far are corollaries of the theory developed in 
sections \ref{sectionFromderivationsofBtoderivationsofBH}--\ref{sectionThesetXscrB}, where we study rings graded by arbitrary abelian groups.
Given a domain $B = \bigoplus_{i \in G} B_i$ graded by an abelian group $(G,+)$,
we define $\G(B)$ to be the subgroup of $G$ generated by $\setspec{ i \in G }{ B_i \neq 0 }$.
Note that $\G(B)$ generalizes the number $e(B)$.
The generalization of $\bar e(B)$ is the subgroup $\bG(B)$ of $\G(B)$ defined in Section \ref{SectionTheGroupbGB}.
We say that $B$ is {\it saturated in codimension $1$} if $\bG(B) = \G(B)$.
Given a subgroup $H$ of $G$, we define the graded subring $B^{(H)} = \bigoplus_{i \in H} B_i$ of $B$, which generalizes the notion of Veronese subring.
We consider the set $\Xscr(B)$, which is the generalization of $\NR(B)$ in the context of $G$-graded rings.
As a first approximation, $\Xscr(B)$ is the set of subgroups $H$ of $G$ such that $B^{(H)}$ is non-rigid;
more precisely, we define
$$
\Xscr(B) = \setspec{ H \in \bbT(B) }{ \hlnd( B^{(H)} ) \neq \{0\} } ,
$$
where $\bbT(B)$ is the set of subgroups $H$ of $G$ such that $\G(B) / (H \cap \G(B))$ is a torsion group.
(Note that if $G$ is a finite group then $\bbT(B)$ is the set of all subgroups of $G$,
and that if $G=\Integ$ and the grading of $B$ is non-trivial then $\bbT(B)$ is the set of all subgroups $d\Integ$ of $\Integ$ with $d\ge1$.)
We also define $\Mscr(B)$ to be the set of maximal elements of the poset $(\Xscr(B),\subseteq)$.

The list below summarizes our most significant results. Each item in the list gives a {\it simplified version\/} of the result that is named at the beginning.
We need the following definition.
Let $\bk$ be a field and $B = \bigoplus_{i \in G} B_i$  a $\bk$-domain graded by an abelian group $G$.
If $\bk \subseteq B_0$, we say that the grading is ``over $\bk$''.
Note that if $G = \Integ$ then  the grading is necessarily over $\bk$.

\medskip

\noindent{\bf(1) }{\bf Thm \ref{8237e9d1983hdjyev93} (Descent Theorem).}
{\it Let $\bk$ be a field of characteristic $0$ and $B$ an affine $\bk$-domain graded over $\bk$ by an abelian group $G$.
If $\hlnd(B) \neq \{0\}$ then $\hlnd(B^{(H)}) \neq \{0\}$ for all $H \in \bbT(B)$.}

\smallskip

\noindent{\bf(2) }{\bf Thm \ref{3498huv8oe98y3876wgebf09}.}
{\it Let $B$ be a domain of characteristic $0$ graded by an abelian group~$G$.
\begin{enumerata}
\vspace{-1.5mm}\item $\bG(B) \subseteq \G( \ker D ) \subseteq \G(B)$ for all $D \in \hlnd(B)$.
\item If $B$ is saturated in codimension $1$ then $\G( \ker D ) = \G(B)$ for all $D \in \hlnd(B)$.
\end{enumerata}}
\vspace{-1.5mm}\noindent Assertion (b) is more general than several published results, for instance:
Cor.\ 2.2 of \cite{Dai:homog}, Thm 2.2 of \cite{Dai:KerHomog} and Corollaries 4.2--4.4 of \cite{DaiFreudMoser}.

\smallskip
\smallskip

\noindent{\bf(3) }{\bf Thm \ref{o823y48t2309fnb28} (Extension Theorem).}
{\it Let $B$ be a noetherian normal $\Rat$-domain graded by a finitely generated abelian group $G$.
If $H \in \bbT(B)$ then every derivation $\delta :  B^{(H)} \to B^{(H)}$ extends uniquely to a derivation $D: B^{( H + \bG(B) )} \to B^{( H + \bG(B) )}$.
Moreover, if $\delta$ is locally nilpotent (resp.\ homogeneous) then so is $D$.}

\vspace{-0.4mm}\noindent Note that, in the above statement, if  $B$ is saturated in codimension $1$ then $B^{( H + \bG(B) )} = B$.
The above Extension Theorem generalizes several published results,
for instance: \cite{AitzUmirbVeronese2023} gives the special case where $B$ is the polynomial ring $\bk[X,Y]$ equipped with the standard grading,
and \cite{Freud2024} gives the case where $B = \bk[X,Y,Z]/(XZ-Y^2-1) = \bk[x,y,z]$ is graded by the group $\Integ/2\Integ = \{0,1\}$, with $x,y,z \in B_1$.
(In these two cases, $B$ is saturated in codimension $1$.)

\smallskip
\smallskip

\noindent{\bf(4) }{\bf Thm \ref{987654esdfghjdj4edoodji2d0we9iu9w12twf5}.}
{\it Let $\bk$ be a field of characteristic $0$ and $B$ an affine $\bk$-domain graded over $\bk$ by a finitely generated abelian group $G$.
\begin{enumerata}
\vspace{-1.5mm}\item $\Xscr(B) = \bigcup_{ H \in \Mscr(B) } \bbT_H(B)$,\ \ where $\bbT_H(B) = \setspec{ H' \in \bbT(B) }{ H' \subseteq H }$.
\item If $B$ is normal then each element $H$ of $\Mscr(B)$ satisfies $H \supseteq \bG(B)$.
\end{enumerata}}
\vspace{-1.5mm}\noindent Note that part (a) of Thm \ref{987654esdfghjdj4edoodji2d0we9iu9w12twf5} reduces the problem of describing $\Xscr(B)$ to that of describing $\Mscr(B)$,
and that part (b) gives some information about $\Mscr(B)$.

\smallskip
\smallskip

\noindent{\bf(5) }{\bf Corollary \ref{87654edfgyudi30piwnf3bwvf2tyweuiklc9Bmo3vA34b6getw63ye}.}
{\it Let $\bk$ be a field of characteristic $0$ and $B$ a normal affine $\bk$-domain graded over $\bk$ by a finitely generated abelian group $G$.
If $B$ is saturated in codimension $1$ then $\Xscr(B) = \emptyset$ or $\Xscr(B) = \bbT(B)$.}

\smallskip

\noindent{\bf(6) }{\bf Section \ref{sectionInteggradings}} focuses on $\Integ$-gradings. 
Many of the results in this section are special cases of results from Sections \ref{sectionFromderivationsofBtoderivationsofBH}--\ref{sectionThesetXscrB},
so we omit their statements here.

\smallskip

\noindent{\bf(7) }{\bf Thm \ref{eion98h38dhb2d203ojfbdt7}} gives a complete description of $\NR(B)$ when $B$ is a Pham-Brieskorn ring satisfying a certain hypothesis
(this hypothesis is satisfied by all Pham-Brieskorn rings of dimensions 2 and 3, and is conjectured to hold in all dimensions).
\smallskip

\noindent{\bf(8) }{\bf Thm \ref{7654cvhgfsd2d34fds23f42dfdghjk8l}.}
{\it Let $\bk$ be a field of characteristic $0$ and
$B = \bigoplus_{n \in \Nat} B_n$ an $\Nat$-graded normal affine $\bk$-domain
such that the prime ideal $B_+ = \bigoplus_{n > 0} B_n$ has height at least $2$.
Let $x$ be a homogeneous prime element of $B$ of degree $d>0$.
If $B_{(x)}$ is non-rigid then so is $(B/xB)^{(d)}$.}

\vspace{-0.0mm}\noindent This generalizes Theorem 3.1 of \cite{Park_ComplementOfHypersurfaces_2022}.  See Thm \ref{7654cvhgfsd2d34fds23f42dfdghjk8l-geometric}, below.

\smallskip
\smallskip

\noindent{\bf(9) }{\bf Prop.\ \ref{u65vmklio909543qsdfh4vue} (Fibers of a polynomial).}
{\it Let $\bk$ be a field of characteristic $0$ and consider the polynomial ring $R = \bk[X_1, \dots, X_n]$ ($n\ge2$)
equipped with an $\Nat$-grading such that each $X_i$ is homogeneous and $X_1$ has positive degree.
Let $f$ be a homogeneous prime element of $R = \bigoplus_{i \in \Nat} R_i$ such that $f \notin R_0[X_1]$.
If $R/fR$ is rigid then so is $R/(f - c)R$ for every $c \in \bk$.}

\smallskip

\noindent{\bf(10) }{\bf Thm \ref{p0987uhmnbvccjiuseyt428}.}
{\it Let $\bk$ be a field of characteristic $0$ and $B$ a normal affine $\bk$-domain graded over $\bk$ by an abelian group $G$.
Define $r = \rank\big( \G(B) / \bG(B) \big)$.  There exists a field $K$ such that
$\ML(B) \subseteq K \subseteq \Frac(B)$ and $\Frac(B) = K^{(r)}$.
In particular, $\trdeg( B : \ML(B) ) \ge r$.}\footnote{The notation $\Frac(B) = K^{(r)}$ means that $\Frac(B)$ is a purely transcendental extension of $K$
of transcendence degree $r$. The symbol $\ML(B)$ denotes the Makar-Limanov invariant of $B$, i.e., the intersection of the kernels of all elements of $\lnd(B)$.}

\medskip

It is apparent from the above list that the description of $\Xscr(B)$ (or $\NR(B)$) is a central theme of this article, though not its sole objective.
Items (2), (8), (9) and (10) are not related to the description of $\Xscr(B)$.
Items (1) and (3) not only serve to prove the results describing $\Xscr(B)$ but also stand out
as independently interesting and significantly more general than previously published results.

\section*{Geometry}

Much of the literature in this area is framed in the language of polarized varieties,
whereas the present article is written in a purely algebraic language.
To facilitate connections between our results and the existing literature, we conclude this introduction
by recalling a few facts about polarized varieties and polar cylinders.
These remarks are not needed for understanding the article, and are not used in the body of the article.

\begin{nothing*}
Let $Y$ be a normal variety with function field $K$.
We write $\Div(Y)$ (resp.\ $\Div_\Rat(Y)$) for the free $\Integ$-module (resp.\ free $\Rat$-module) on the set of prime divisors of $Y$.
The elements of $\Div(Y)$ (resp.\ $\Div_\Rat(Y)$) are called {\it divisors\/} (resp.\ {\it $\Rat$-divisors\/}) of $Y$. 
Two $\Rat$-divisors $D,D' \in \Div_\Rat(Y)$ are {\it linearly equivalent\/} ($D \sim D'$) if $D-D' = \div_Y(f)$ for some $f \in K^*$;
we say that  $D,D'$ are {\it $\Rat$-linearly equivalent\/} ($D \simQ D'$) if $nD \sim nD'$ for some integer $n>0$.
If $D \in \Div(Y)$ then the sheaf $\Oscr_Y(D)$ on $Y$ is defined by
$$
\Gamma(U,\Oscr_Y(D)) = \{0\} \cup \setspec{ f \in K^* }{ \div_U(f) + D|_U \ge 0 } \qquad \text{($\emptyset \neq U \subseteq Y$ open)}.
$$
A divisor $D \in \Div(Y)$ is {\it very ample\/} if $\Oscr_Y(D)$ is a very ample invertible sheaf.
An {\it ample $\Rat$-divisor\/} of $Y$ is a $\Rat$-divisor $\Delta$ of $Y$ for which there exists an integer $n>0$ such that $n\Delta \in \Div(Y)$ is very ample.
Suppose that $\Delta$ is an ample $\Rat$-divisor of $Y$;
an open subset $U$ of $Y$ is said to be {\it $\Delta$-polar\/} if $U = Y \setminus \supp(D)$ for some effective $\Rat$-divisor $D$ such that $D \simQ \Delta$.
\end{nothing*}

\begin{nothing*}
Observe that if $B$ is an $\Nat$-graded domain satisfying $e(B)=1$ then there exists a homogeneous element $t$ of $\Frac(B)$ of degree $1$
(meaning that $t = f/g$ for some nonzero homogeneous elements $f,g \in B$ such that $\deg(f) - \deg(g) = 1$).
This is relevant in Thm \ref{jh782uwsdvvhjkqw547182}.
\end{nothing*}

The following is \cite[Thm 3.5]{Demazure_1988} with some extra pieces added to it (see \ref{7tyzvcxzcvbcvQcksbvcstdio976543dfsxnbx48} for the proof of the extra pieces).

\begin{theorem} \label {jh782uwsdvvhjkqw547182}
Let $\bk$ be a field.
Let $B$ be an $\Nat$-graded normal affine $\bk$-domain such that $e(B)=1$ and such that the number $\haut(B_+) = \trdeg(B:B_0)$ is at least $2$.
\begin{enumerata}

\item For each homogeneous element $t$ of $\Frac(B)$ of degree $1$,
there exists a unique $\Rat$-divisor $\Delta$ on the normal variety $Y = \Proj B$ such that $B = \bigoplus_{m\in\Nat} H^0(Y,\Oscr_Y(m\Delta))t^m$ is an equality of rings.
Moreover, $\Delta$ has the following properties:
\begin{enumerata}

\item $\Delta$ is an ample $\Rat$-divisor of $Y$.

\item A nonempty open subset of $\Proj(B)$ is $\Delta$-polar if and only if it is basic.\footnote{An open subset $U$ of $\Proj B$ is {\it basic\/}
if $U = \bbD_+(f) := \setspec{ \pgoth \in \Proj B }{ f \notin \pgoth }$ for some homogeneous $f \in B$ of nonzero degree.}

\item $\Delta$ is Cartier if and only if $e( B/\pgoth ) = e( B )$ for all $\pgoth \in \Proj B$.

\item $\Delta \in \Div(Y)$ if and only if $B$ is saturated in codimension $1$.

\end{enumerata}

\item If $t,t'$ are homogeneous elements of $\Frac(B)$ of degree $1$ then the corresponding ample $\Rat$-divisors $\Delta,\Delta'$, 
defined as in part {\rm(a)}, are linearly equivalent.

\end{enumerata}
\end{theorem}

It is well known that there is a converse to Thm \ref{jh782uwsdvvhjkqw547182} and that, consequently,
studying an $\Nat$-graded normal affine $\bk$-domain $B$ satisfying $e(B)=1$ and $\haut(B_+)\ge2$ is equivalent to studying 
the corresponding polarized variety, i.e., the pair $(Y,\Delta)$ where $Y=\Proj(B)$ and $\Delta$ is the ample $\Rat$-divisor on $Y$
given (up to linear equivalence) by Thm \ref{jh782uwsdvvhjkqw547182}(a).
This equivalence between $B$ and $(Y,\Delta)$ allows us to interpret some of our results in the context of polarized varieties, or vice versa.
Here, one should keep in mind that
part (a-iv) of Thm \ref{jh782uwsdvvhjkqw547182} gives a geometric interpretation of saturation in codimension $1$,
and that part (a-ii) shows that the terms ``$\Delta$-polar cylinder'' and ``basic cylinder'' are interchangeable.
For instance, this shows that Thm \ref{edh83yf6r79hvujhxu6wrefji9e} generalizes Thm 0.6 and Cor.\ 3.2 of \cite{KPZ2013}.
As another example, consider:

\begin{theorem}  \label {7654cvhgfsd2d34fds23f42dfdghjk8l-geometric}
Let $\bk$ be a field of characteristic $0$ and let $\proj = \Proj(B)$ where
$B$ is an $\Nat$-graded normal affine $\bk$-domain such that $\haut( B_+ ) \ge 2$.
Let $F$ be a homogeneous prime element of $B$ of degree $d>0$ and let $X = \bbV_+(F) \subset \proj$.
If the affine variety $\proj \setminus X$ has a non-trivial $G_a$-action then $(B/FB)^{(d)}$ is non-rigid.
\end{theorem}

This result is clearly equivalent to  Thm \ref{7654cvhgfsd2d34fds23f42dfdghjk8l}, which we prove in Section \ref{sectionRigidityofBx}.
We wrote the statement of Thm \ref{7654cvhgfsd2d34fds23f42dfdghjk8l-geometric} in a form that is easy to compare with Theorem 3.1 of \cite{Park_ComplementOfHypersurfaces_2022}.
Thm \ref{7654cvhgfsd2d34fds23f42dfdghjk8l-geometric} generalizes \cite[3.1]{Park_ComplementOfHypersurfaces_2022} in several ways:
\cite[3.1]{Park_ComplementOfHypersurfaces_2022} assumes
that $B$ is a polynomial ring, that $B/FB$ is normal, that $\Oscr_\proj(d)$ is very ample and that $\bk$ is algebraically closed.

Note that Prop.\ \ref{u654edcvbhyu89okmnbgr32qazdcftgyu934} and Ex.\ \ref{lBoibhgxfd7ii03A8vevcCxEvv82} are applications
of Thm \ref{7654cvhgfsd2d34fds23f42dfdghjk8l-geometric} (or Thm \ref{7654cvhgfsd2d34fds23f42dfdghjk8l}) to the case where $B$ is a Pham-Brieskorn ring, i.e.,
a case not covered by \cite{Park_ComplementOfHypersurfaces_2022}.

\section{Preliminaries}
\label {SEC:Preliminaries}

This section gathers the definitions and facts that we need in this article.
All results are known, but we provide proofs when we are unable to give a reference.

\begin{nothing*}
We use ``$\setminus$'' for set difference, ``$\subset$'' for strict inclusion and  ``$\subseteq$'' for general inclusion.
We follow the convention that $0 \in \Nat$.

If $S$ is a nonempty subset of $\Integ$ then $\lcm(S)$ is defined to be the nonnegative generator of
the ideal $\bigcap_{a \in S} a\Integ$ of $\Integ$.  In particular, if $S$ is an infinite subset of $\Integ$ then $\lcm(S)=0$.
Also, $\gcd(S)$ is the nonnegative generator of the ideal of $\Integ$ generated by $S$.

All rings and algebras are assumed to be associative, commutative and unital.
If $A$ is a ring then $A^*$ denotes the set of units of $A$.
If $B$ is an algebra over a ring $A$, the notation $B = A^{[n]}$ (where $n \in \Nat$) means that $B$ is isomorphic as an $A$-algebra to a polynomial
ring in $n$ variables over $A$. If $L/K$ is a field extension then $L = K^{(n)}$ means that $L$ is a purely transcendental extension of $K$
of transcendence degree $n$.

The word ``domain'' means ``integral domain''.
We write $\Frac A$ for the field of fractions of a domain $A$.
If $\bk$ is a field, then a {\it $\bk$-domain\/} is a domain that is also a $\bk$-algebra;
an {\it affine $\bk$-domain\/} (or {\it $\bk$-affine domain}) is a domain that is a finitely generated $\bk$-algebra.
If $A \subseteq B$ are domains then the transcendence degree of $B$ over $A$ is denoted $\trdeg_A(B)$ or $\trdeg(B:A)$.

If $A$ is a ring, we write $\Spec^1(A) = \setspec{ \pgoth \in \Spec A }{ \haut\pgoth=1 }$.

If $S$ is a subset of a group $G$, then $\langle S \rangle$ denotes the subgroup of $G$ generated by $S$.
\end{nothing*}

\begin{nothing*} \label {8736egd9fv092b5}
Let $G$ be an abelian group (with additive notation).
A {\it $G$-grading\/} of a ring $B$ is a family $\big( B_i \big)_{i \in G}$ of subgroups of $(B,+)$ satisfying $B = \bigoplus_{i \in G} B_i$
and $B_iB_j\subseteq B_{i+j}$ for all $i,j \in G$.
The phrase ``let  $B = \bigoplus_{i \in G} B_i$ be a $G$-graded ring'' means that we are considering the ring $B$ together with 
the $G$-grading $\big( B_i \big)_{i \in G}$.

Let  $B = \bigoplus_{i \in G} B_i$ be a $G$-graded ring.
\begin{enumerate}

\item An element of $B$ is {\it homogeneous\/} if it belongs to $\bigcup_{i \in G} B_i$.
If $x$ is a nonzero homogeneous element, the {\it degree\/} of $x$, $\deg(x)$, is the unique $i \in G$ such that $x \in B_i$.
The degree of a non-homogeneous element is not defined.

\item $B_0$ is a subring of $B$ and is called the {\it degree-$0$ subring\/} of $B$.
If $B = B_0$, we say that the grading is {\it trivial.}

\item Given a homogeneous element $f$ of $B$, $B_{(f)}$ denotes the degree-$0$ subring of the $G$-graded ring $B_f = S^{-1}B$ where $S= \{1,f,f^2,\dots\}$.
Given a homogeneous prime ideal $\pgoth$ of $B$, $B_{(\pgoth)}$ denotes the degree-$0$ subring of the $G$-graded ring
$S^{-1}B$ where $S$ is the set of homogeneous elements of $B \setminus \pgoth$.

\item If $H$ is a subgroup of $G$, define $B^{(H)} = \bigoplus_{i \in H} B_i$ and note that 
 the inclusion $B^{(H)} \hookrightarrow B$ is a degree-preserving homomorphism of graded rings.

\item If $G=\Integ$ and $B_i =0$ for all $i<0$, we say that $B$ is $\Nat$-graded. In this case we write $B = \bigoplus_{ i \in \Nat } B_i$
and define $B_+ = \bigoplus_{i>0} B_i$, which is an ideal of $B$.

\item The subgroup of $G$ generated by $\setspec{ i \in G }{ B_i \neq 0 }$ is denoted $\G(B)$.

\end{enumerate}
\end{nothing*}

\begin{theorem}[Thm 1.1 of \cite{GotoNoethGrRings83}] \label {983765432wevvajwklw92980}
Let $G$ be a finitely generated abelian group and $A$ a $G$-graded ring.
The following are equivalent:
\begin{enumerata}

\item $A$ is a noetherian ring;

\item the ring $A_0$ is noetherian and the $A_0$-algebra $A$ is finitely generated;

\item every homogeneous ideal of $A$ is finitely generated.

\end{enumerata}
\end{theorem}

\begin{lemma} \label {98767gh11dmdffkll0ahcchvr0vfesh8263543794663cnir}
Let $A$ be a ring graded by an abelian group $G$, and let $H$ be a subgroup of $G$.
\begin{enumerata}
\item If $I$ is an ideal of $A^{(H)}$ then $A^{(H)} \cap IA = I$.
\item If $A$ is noetherian then so is $A^{(H)}$.
\end{enumerata}
Moreover, if $A$ is a domain then the following hold.
\begin{enumerata}
\addtocounter{enumi}{2}

\item $A \cap \Frac( A^{(H)} ) = A^{(H)}$ 

\item If $A$ is normal then so is $A^{(H)}$.

\end{enumerata}
\end{lemma}

\begin{proof}
Write $A = \bigoplus_{i \in G} A_i$.
We first prove the case $H = 0$ of (a). Note that $A^{(H)} = A_0$.
Let $I$ be an ideal of $A_0$ and let $x \in A_0 \cap IA$.
Then $x = u_1 a_1 + \cdots + u_n a_n$ for some $u_1,\dots,u_n \in I$ and $a_1,\dots,a_n \in A$.
Moreover, we can arrange that $a_i$ is homogeneous and $u_i a_i \in A_0 \setminus \{0\}$ for all $i \in \{1,\dots,n\}$.
We have $a_i \in A_{d_i}$ for some $d_i \in G$, and $A_0 \setminus \{0\} \ni u_i a_i \in A_{0+d_i}$, so $d_i = 0$ and hence $a_i \in A_0$ for all $i \in \{1,\dots,n\}$.
Thus, $x \in I$, showing that (a) is true when $H=0$.
For the general case of (a), let $\pi : G \to \bar G = G/H$ be the canonical epimorphism and let $B = \bigoplus_{j \in \bar G}B_j$ be the
ring $A$ equipped with the $\bar G$-grading defined by $B_j = \bigoplus_{\pi(i)=j} A_i$ for each $j \in \bar G$.
Since $B_0 = A^{(H)}$ and $B=A$, the claim follows by applying the case ``$H=0$'' to $B$.
This proves (a).

(b) If $I_0 \subseteq I_1 \subseteq I_2 \subseteq \cdots$ is an infinite increasing sequence of ideals in $A^{(H)}$
then the sequence $I_0 A \subseteq I_1 A \subseteq I_2 A \subseteq \cdots$ stabilizes since $A$ is noetherian.
By (a), we have $I_n = A^{(H)} \cap I_n A$ for all $n$, so $(I_n)_{n \in \Nat}$ stabilizes, showing that $A^{(H)}$ is noetherian.

(c) Given $x = \sum_{i \in G} x_i \in A$ (where $x_i \in A_i$ for all $i \in G$), define $S(x) =\setspec{ i \in G }{ x_i \neq 0 }$. 
We claim that
\begin{equation}  \label {2Qqs1mlkwm1oq90uvxzaAa11w4}
\text{if $x \in A \cap \Frac( A^{(H)} )$ and $x \neq 0$ then $S(x) \cap H \neq \emptyset$.}
\end{equation}
Indeed, write $x = u/v$ where $u,v \in A^{(H)} \setminus \{0\}$.
Then $u\neq0$ implies $S(u) \neq \emptyset$ and $xv = u$ implies $S(u) \subseteq \setspec{ i+j }{ i \in S(x),\ j \in S(v) }$,
so there exist $i \in S(x)$ and $j \in S(v)$ such that $i+j \in S(u) \subseteq H$;
since $j \in H$, it follows that $i \in H$, so  $S(x) \cap H \neq \emptyset$.
This proves \eqref{2Qqs1mlkwm1oq90uvxzaAa11w4}.

Now consider $x \in A \cap \Frac( A^{(H)} )$. We can write $x = x_H + x'$ with $x_H,x' \in A$, $S(x_H) \subseteq H$ and $S(x') \cap H = \emptyset$.
Since $x$ and $x_H$ belong to $A \cap \Frac( A^{(H)} )$, we have $x' = x - x_H \in A \cap \Frac( A^{(H)} )$.
Since $S(x') \cap H = \emptyset$, \eqref{2Qqs1mlkwm1oq90uvxzaAa11w4} implies that $x'=0$, so $x = x_H \in A^{(H)}$.
Thus, $A \cap \Frac( A^{(H)} ) = A^{(H)}$.

(d) If $A$ is normal then $A^{(H)} = A \cap \Frac( A^{(H)} )$ is an intersection of two normal domains and hence is normal.
\end{proof}

\begin{lemma} \label {ze3xqrctwvyjmkiu0u9inj7b2q3q01iqu}
Let $B$ be a domain graded by an abelian group $G$, let $H$ be a subgroup of $G$,
and let $S$ be a multiplicative subset of $B^{(H)} \setminus \{0\}$ such that each element of $S$ is homogeneous. Then
$$
S^{-1} \big( B^{(H)} \big) = \big( S^{-1} B \big)^{(H)} \, .
$$
\end{lemma}

\begin{proof}
Note that $S^{-1} B$ is a $G$-graded domain.
Consider $\frac{b}{s} \in S^{-1} B$ where $b$ is a homogeneous element of $B \setminus \{0\}$ and $s \in S$.
Since $\deg\big(\frac{b}{s}\big) = \deg(b) - \deg(s)$ and $\deg(s) \in H$,
we have $\deg(b) \in H$ $\Leftrightarrow$ $\deg\big(\frac{b}{s}\big) \in H$, i.e.,
$\frac{b}{s} \in  S^{-1}\big( B^{(H)} \big)$ $\Leftrightarrow$ $\frac{b}{s} \in \big( S^{-1} B \big)^{(H)}$.
\end{proof}

The case $G=\Integ$ of the following result is well known.
We don't know a reference for the general case, so we include a proof.

\begin{lemma}  \label {0Edlv2y6ctwbeofyvrbqleuo245612nxr7m54}
Let $B$ be a ring graded by a torsion-free abelian group $G$.
\begin{enumerata}

\item  If $\pgoth \in \Spec B$ and $I$ is the ideal of $B$ generated by all homogeneous elements of $\pgoth$, then $I \in \Spec B$.

\item If $B$ is a domain, $\pgoth \in \Spec^1(B)$ and $\pgoth$ contains a nonzero homogeneous element of $B$, then $\pgoth$ is a homogeneous ideal.

\end{enumerata}
\end{lemma}

\begin{proof}
(a) Since $G$ is torsion-free, there exists an order relation $\le$ on $G$ such that $(G,\le)$ is a totally ordered abelian group
(cf.\ \cite[Prop.\ 1.1.7]{AndersonFeil}); we choose such an order relation.
Proceeding by contradiction, suppose that $x,y \in B \setminus I$ are such that $xy \in I$.
Write $x = \sum_{i \in G} x_i$ and $y = \sum_{i \in G} y_i$ with $x_i,y_i \in B_i$ for all $i$.
The sets $S_x =\setspec{ i \in G }{ x_i \notin I}$ and $S_y =\setspec{ i \in G }{ y_i \notin I}$ are nonempty and finite,
so we may define $m_x = \max S_x$ and $m_y = \max S_y$. Let $x' = \sum_{i \le m_x} x_i$ and $y' = \sum_{i \le m_y} y_i$;
then $x \equiv x'$ and $y \equiv y' \pmod{I}$, so $x'y' \in I$. Moreover, if we write $x'y' = \sum_{i \in G} z_i$ with $z_i \in B_i$ for all $i$,
then $x_{m_x} y_{m_y} = z_{ m_x + m_y } \in I$,
so $x_{m_x} y_{m_y} \in \pgoth$,
so $x_{m_x} \in \pgoth$ or $y_{m_y} \in \pgoth$,
so $x_{m_x} \in I$ or $y_{m_y} \in I$, a contradiction.
This proves (a).

To prove (b), consider $I$ as in part (a) and observe that, since $\haut\pgoth=1$, $\pgoth = I$.
\end{proof}

\begin{bigremark}  \label {iuytytsxc3welqmo9cVbvbAfvVGhue83983teij}
Of course, the above Lemma remains valid if we replace the assumption that $G$ is torsion-free
by the assumption that $\G(B)$ is torsion-free.
See \ref{8736egd9fv092b5} for the definition of $\G(B)$.
\end{bigremark}

\begin{lemma}  \label {i8765redfvbnki8765rfghytrew123456789iuhv}
Let $R = \bigoplus_{i \in G} R_i$ be a domain graded by an abelian group $G$,
let $S$ be the set of all nonzero homogeneous elements of $R$
and consider the $G$-graded domain $\Reul = S^{-1}R = \bigoplus_{ i \in G } \Reul_i$.
If $\G(R) \isom \Integ^r$ where $r \in \Nat$,
then $\Reul = \Reul_0[t_1^{\pm1}, \dots, t_r^{\pm1} ]$ where $\Reul_0$ is a field that contains $R_0$ and $t_1,\dots,t_r$ are 
nonzero homogeneous elements of $\Reul$ that are algebraically independent over $\Reul_0$.
\end{lemma}

\begin{proof}
We may assume that $r \neq 0$.
Observe that $\G(\Reul) = \G(R) \isom \Integ^r$.
Since every nonzero homogeneous element of $\Reul$ is a unit, we have $\G(\Reul) = \setspec{ i \in G }{ \Reul_i \neq 0 }$,
so we can choose nonzero homogeneous elements $t_1,\dots,t_r$ of $\Reul$ such that $( \deg(t_j) )_{j=1}^r$ is a basis of the free $\Integ$-module $\G(\Reul)$. 
The reader can check that $\Reul_0$ and $t_1,\dots,t_r$ have the desired properties.
\end{proof}

Let us now say a few words about graded algebras.

\begin{definition}  \label {8hVCixdCN95g2ZRB3iAcxzqe}
Let $\bk$ be a field and $B = \bigoplus_{i \in G} B_i$ a $\bk$-algebra graded by an abelian group $G$.
If $\bk \subseteq B_0$, we say that the grading is {\it over\/} $\bk$, or that $B$ is graded {\it over $\bk$\/} by $G$.
\end{definition}

\begin{bigremark}  \label {ijn3eoocnnqppieudbc}
Let $\bk$ be a field, $B$ a $\bk$-domain and $G$ an abelian group.
If $\bk = \Rat$ or $G$ is torsion-free then every $G$-grading of $B$ is over $\bk$.
However, this is not true for arbitrary $\bk$ and $G$.
For example, the $\Comp$-algebra $\Comp$ admits a $\Integ/2\Integ$-grading that is not over $\Comp$
(namely, $\Comp = \Reals \oplus \Reals i$).
\end{bigremark}

\begin{bigremark}
Let $\bk$ be a field and $B = \bigoplus_{i \in G}B_i$ an affine $\bk$-domain graded by an abelian group $G$.
If the grading is over $\bk$ then $B$ is finitely generated as a $B_0$-algebra and $\G(B)$ is a finitely generated group.
Here is an example showing that these conclusions are not necessarily valid when the grading is not over $\bk$.

Let $R = \Comp[ (\Rat,+) ]$ be the group ring of the group $(\Rat,+)$.
The elements of $R$ are
formal sums $\sum_{q \in \Rat} a_q t^q$ where $a_q \in \Comp$ for all $q$ and $a_q \neq 0$ for at most finitely many $q$,
and where the family $(t^q)_{q \in \Rat}$ satisfies $t^q t^r = t^{q+r}$ for all $q,r \in \Rat$.
Consider the subfield $\bk = \Comp( (t^q)_{q \in \Rat} )$ of the field of fractions of $R$;
note that $t$ is transcendental over $\Comp$ and that $\bk / \Comp(t)$ is an algebraic extension of infinite degree.
For each $q \in I = [0,1) \cap \Rat$, consider the subspace $V_q = \Comp(t)t^q$ of the vector space $\bk$ over $\Comp(t)$.
Then $\bk = \bigoplus_{q \in I} V_q$.
Let $G = \Rat/\Integ$ and let $\pi : \Rat \to G$ be the canonical homomorphism of the quotient.
Then $\pi|_I : I \to G$ is bijective; let $\phi : G \to I$ be the inverse of $\pi|_I$.
So we have $\bk = \bigoplus_{g \in G} V_{\phi(g)} = \bigoplus_{g \in G} \bk_g$ where we define $\bk_g = V_{\phi(g)}$ for all $g \in G$.
This is a $G$-grading of $\bk$ such that $\bk_0 = \Comp(t)$.
So $\bk$ is a $G$-graded affine $\bk$-domain, $\bk$ is not finitely generated as a $\bk_0$-algebra, and $\G(\bk) = \Rat/\Integ$ is not finitely generated.
\end{bigremark}

\begin{lemma}[Lemma 2.2 of \cite{Chitayat-Daigle:cylindrical}]  \label {kcjviyc2e2w45w6d2o9kmgvczxwer}
Let $G$ be an abelian group, $B = \bigoplus_{i \in G} B_i$ a $G$-graded ring, and $R$ a subring of $B_0$.
If $B$ is finitely generated as an $R$-algebra then so is $B^{(H)}$ for every subgroup $H$ of $G$.
\end{lemma}

\begin{corollary}  \label {8Qp98781hy2uexihvcx2qw2wUxq3ew0pokjmnbvaw2345}
If $\bk$ is a field and $B$ is an affine $\bk$-domain graded over $\bk$ by an abelian group $G$,
then $B^{(H)}$ is an affine $\bk$-domain for every subgroup $H$ of $G$.
\end{corollary}

\begin{proof}
Since the grading is over $\bk$, this follows from Lemma \ref{kcjviyc2e2w45w6d2o9kmgvczxwer}.
\end{proof}

\section*{The set $\bbT(B)$}

\begin{notation}
Given a domain $B$ graded by an abelian group $G$,
we write $\bbT(B)$ for the set of subgroups $H$ of $G$ such that $\G(B)/(H \cap \G(B))$ is torsion.
\end{notation}

\begin{bigremark}  \label {kDjkbwevCdcojBh6qvmAxfea7u}
Let $B$ be a domain graded by an abelian group $G$.
\begin{enumerate}

\item Let $\Tscr(G)$ be the set of subgroups $H$ of $G$ such that $G/H$ is torsion.
Then $\Tscr(G) \subseteq \bbT(B)$, and equality holds if and only if $G/\G(B)$ is torsion.
Thus, if $\G(B)=G$ then $\bbT(B) = \Tscr(G)$.

\item $\G(B)$ is torsion if and only if $\bbT(B)$ is the set of all subgroups of $G$.
In particular, if $G$ is a finite group then $\bbT(B)$ is the set of all subgroups of $G$.

\item If $G = \Integ$ and the grading is nontrivial then  $\bbT(B)$ is the set of all nonzero subgroups of $\Integ$.

\item If $H \in \bbT(B)$ then $B$ is an integral extension of $B^{(H)}$.
Indeed, if $b \in B \setminus \{0\}$ is homogeneous then $\deg(b) \in \G(B)$, so there exists $n\ge1$ such that $n \deg(b) \in H$
and hence $b^n \in B^{(H)}$.

\end{enumerate}
\end{bigremark}

\begin{bigremark}
When dealing with a single $G$-graded ring $B$, it is often convenient to assume that $\G(B)=G$,
as this typically simplifies definitions, results, and even proofs. 
In the present work, however, we must formulate many definitions and results (for instance the definition of $\bbT(B)$)
without this simplifying assumption, because we frequently consider several graded rings simultaneously---for example,
inclusions $B^{(H)} \subseteq B^{(K)} \subseteq B$ where $H \subseteq K$ are subgroups of $G$. 
\end{bigremark}

\begin{lemma}  \label {87654esxcvs9dfok33me4rtu5ortyo0934urud3r24}
Let $B$ be a noetherian normal domain graded by an abelian group $G$, and let $H \in \bbT(B)$.
\begin{enumerata}

\item $\haut(J) = \haut(J \cap B^{(H)})$ for every ideal $J$ of $B$.

\item The map $f : \Spec B \to \Spec B^{(H)}$, $f(\Pgoth) = \Pgoth \cap B^{(H)}$,
is surjective and satisfies $f^{-1} \big( \Spec^1 B^{(H)} \big) = \Spec^1 B$.

\item If $G$ is torsion-free then $f^{-1}(Z') = Z$,
where $Z$ (resp.\ $Z'$) denotes the set of height $1$ homogeneous prime ideals of $B$ (resp.\ $B^{(H)}$).
Moreover, $f|_Z : Z \to Z'$ is bijective.

\end{enumerata}
\end{lemma}

\begin{proof}
By Lemma \ref{98767gh11dmdffkll0ahcchvr0vfesh8263543794663cnir}, $B^{(H)}$ is a noetherian normal domain;
moreover, $B$ is integral over $B^{(H)}$ by Rem.\ \ref{kDjkbwevCdcojBh6qvmAxfea7u}(4).
By \cite[(5.E)]{Matsumura}, the Going-Down Theorem holds for $B^{(H)} \subseteq B$;
so assertion (a) follows from \cite[(13.C)]{Matsumura}, and (b) follows from (a).
Now assume that $G$ is torsion-free. If $\Pgoth \in \Spec^1 B$ and $\pgoth \in \Spec^1 B^{(H)}$ are such that $f(\Pgoth) = \pgoth$,
then Lemma \ref{0Edlv2y6ctwbeofyvrbqleuo245612nxr7m54}(b) implies that $\Pgoth$ is homogeneous if and only if $\pgoth$ is homogeneous.
So $f^{-1}(Z') = Z$ by part (b), which also implies that $f|_Z : Z \to Z'$ is surjective.
Consider $\pgoth \in Z'$ and $\Pgoth_1, \Pgoth_2 \in Z$ such that $f( \Pgoth_1 ) = \pgoth = f( \Pgoth_2 )$.
If $x$ is a nonzero homogeneous element of $\Pgoth_1$ then we may choose $d\ge1$ such that $x^d \in B^{(H)}$;
then $x^d \in \pgoth \subseteq \Pgoth_2$ and hence $x \in \Pgoth_2$, showing that $\Pgoth_1 \subseteq \Pgoth_2$. By symmetry, $\Pgoth_1 = \Pgoth_2$.
So $f|_Z$ is bijective and (c) is proved.
\end{proof}

\begin{lemma}  \label {jBh9823hedAcv78n2erj20}
Let $B$ be a domain graded by an abelian group $G$.

If $H \in \bbT(B)$ and $M$ is a submonoid of $\G(B)$ then $\langle H \cap M \rangle = H \cap \langle M \rangle$.
\end{lemma}

\begin{proof}
It is clear that  $\langle H \cap M \rangle \subseteq H \cap \langle M \rangle$.
For the reverse inclusion, consider $i \in H \cap \langle M \rangle$.
Since $i \in \langle M \rangle$ and $M$ is a monoid, there exist $j_1,j_2 \in M$ such that $i = j_1 - j_2$.
Since $j_1 \in \G(B)$ and $H \in \bbT(B)$, there exists $n>0$ such that $(n+1) j_1 \in H$.
Define $i_1 = (n+1)j_1$ and $i_2 = j_2 + nj_1$. Then $i_1 \in H \cap M$, $i_2 \in M$, and
$$
H \ni i = (j_1 + nj_1) - (j_2 + nj_1) = i_1 - i_2,
$$
so $i_2 \in H$, i.e., $i_2 \in H \cap M$. So $i =  i_1 - i_2 \in \langle H \cap M \rangle$.
\end{proof}

\begin{lemma}  \label {GOfBH}
Let $B$ be a domain graded by an abelian group $G$.
$$
\G( B^{(H)} ) = H \cap \G(B) \quad \text{for all $H \in \bbT(B)$.}
$$
\end{lemma}

\begin{proof}
Consider the submonoid $M = \setspec{ i \in G }{ B_i \neq 0 }$ of $\G(B)$. Then $\G(B) = \langle M \rangle$ and
$\G( B^{(H)} ) = \langle H \cap M \rangle = H \cap \langle M \rangle = H \cap \G(B)$ by Lemma \ref{jBh9823hedAcv78n2erj20}.
\end{proof}

\begin{corollary}  \label {765tgc2wvtegrlkjmn098b27qw6qesrxcvac23rjpa}
Let $B$ be a domain graded by an abelian group $G$. 
For every subgroup $H$ of $G$,
$$
\text{$H \in \bbT(B)$ $\iff$ $\G(B) / \G( B^{(H)} )$ is torsion.}
$$
\end{corollary}

\begin{proof}
It is clear that $\G( B^{(H)} ) \subseteq H \cap \G(B)$, so there is a surjective group homomorphism
$\G(B) / \G( B^{(H)} ) \to \G(B)/(H \cap \G(B))$ and consequently $(\Leftarrow)$ is true.
Conversely, if $H \in \bbT(B)$ then  $\G(B) / \G( B^{(H)} ) = \G(B)/(H \cap \G(B))$ by Lemma \ref{GOfBH},
and since  $\G(B)/(H \cap \G(B))$ is torsion it follows that $\G(B)/\G( B^{(H)} )$ is torsion, so $(\Rightarrow)$ is true.
\end{proof}

\section*{Derivations}

The notion of {\it locally nilpotent derivation\/} is defined in the introduction, and so is the notation $\lnd(B)$ for any ring $B$.

Let $A \subseteq B$ be domains. We say that $A$ is {\it factorially closed in $B$} if
the conditions $x,y \in B \setminus \{0\}$ and $xy \in A$ imply that $x,y \in A$.
Clearly, if $A$ is factorially closed in $B$ then $B^*=A^*$ and $B \cap \Frac(A) = A$.

It is well known that if $B$ is a domain of characteristic $0$, $D \in \lnd(B)$ and $A = \ker(D)$, then
(i) $A$ is factorially closed in $B$; (ii) if $a \in A$ then the derivation $aD : B \to B$ is locally nilpotent;
(iii) if $S$ is a multiplicative set of $A$ then the derivation $S^{-1}D : S^{-1}B \to S^{-1}B$  is locally nilpotent
and $\ker(S^{-1}D) = S^{-1}A$.

The following is well known and easy to prove:

\begin{lemma} \label {6543wxcvbnc3wvlplpo0i9kijh7waq4rtmlb87hg}
Let $S$ be a multiplicative set of a domain $R$ of characteristic $0$ and let $D : S^{-1}R \to S^{-1}R$ be a derivation.
Suppose that $R$ is finitely generated as an algebra over the ring $R \cap \ker(D)$.
\begin{enumerata}

\item There exists $s \in S$ such that the derivation $sD : S^{-1}R \to S^{-1}R$ maps $R$ into itself.
Consequently, $(sD)|_R : R \to R$ is a derivation of $R$.

\item If $D$ is locally nilpotent then so are $sD$ and $(sD)|_R$.

\end{enumerata}
\end{lemma}

Part (a) of the following fact appeared in \cite[Section 1]{Nouaze-Gabriel_1967} and \cite[Prop.~2.1]{Wright:JacConj};
part (b) easily follows from (a).

\begin{lemma} \label {SliceThm}
Let $B$ be a $\Rat$-algebra, $D \in \lnd(B)$ and $A = \ker(D)$.
\begin{enumerata}

\item If $s\in B$ satisfies $Ds=1$ then $B = A[s] = A^{[1]}$ and $D = \frac{d}{ds} : A[s] \to A[s]$.

\item Let $t \in B$ be such that $D(t) \neq 0$ and $D^2(t)=0$, and let $a = D(t) \in A\setminus\{0\}$.
Then $B_a = A_a[t] = (A_a)^{[1]}$.

\end{enumerata}
\end{lemma}

The following is a slightly improved\footnote{The result in \cite{VASC} assumes that $\Rat \subseteq A$.}
version of the Theorem of Vasconcelos \cite{VASC}:

\begin{lemma}[Lemma 2.8 of \cite{Chitayat-Daigle:cylindrical}] \label {ldkjxhcvi82ewdj0wd}
Let $A \subseteq B$ be domains of characteristic $0$ such that $B$ is integral over $A$.
If $\delta : A \to A$ is a locally nilpotent derivation and $D : B \to B$ is a derivation that extends $\delta$,
then $D$ is locally nilpotent.
\end{lemma}

\begin{definition} \label {8hnvnBbslijvxbgxsab8AbN7p03msS726rfXC}
Let $B = \bigoplus_{i \in G} B_i$ be a ring graded by an abelian group $G$.
\begin{enumerate}

\item A derivation $D : B \to B$ is {\it homogeneous\/} if there exists $d \in G$ such that $D(B_i) \subseteq B_{i+d}$ holds for all $i \in G$;
if $D$ is homogeneous and nonzero then $d$ is unique, and is called the {\it degree\/} of $D$.

\item $\hlnd(B) = \setspec{ D \in \lnd(B) }{ \text{$D$ is homogeneous} }$

\item $\khlnd(B) = \setspec{ \ker(D) }{ D \in \hlnd(B) \text{ and } D \neq 0 }$

\item We say that $B$ is {\it rigid\/} if $\lnd(B) = \{0\}$, and {\it graded-rigid\/} if $\hlnd(B) = \{0\}$.

\end{enumerate}
\end{definition}

Graded rings $B$ satisfying $\hlnd(B)=\{0\}$ and $\lnd(B) \neq\{0\}$ do exist (see for instance Propositions 6.5 and 6.6 of \cite{DaiFreudMoser});
according to our definitions, such rings are graded-rigid and non-rigid.
The following well-known fact states that rigidity and graded-rigidity are equivalent under certain assumptions on $B$ and $G$.

\begin{lemma}[Lemma 2.7 of \cite{Chitayat-Daigle:cylindrical}]  \label {876543wsdfgxhnm290cAo}
Let $\bk$ be a field of characteristic $0$ and $B$ an affine $\bk$-domain graded by a torsion-free abelian group $G$.
Then,
$$
\lnd(B) = \{0\} \ \iff \ \hlnd(B) = \{0\} .
$$
\end{lemma}

\section{From derivations of $B$ to derivations of $B^{(H)}$}
\label{sectionFromderivationsofBtoderivationsofBH}

The main result of the section is Thm \ref{8237e9d1983hdjyev93}.  For its proof, we need:

\begin{lemma}  \label {876543wqzxcvbnmC9l03ord}
Let $\Beul$ be a domain of characteristic $0$ graded by an abelian group $G$ and suppose that $\Beul = R[t] = R^{[1]}$
where $R$ is a graded subring of $\Beul$ and $t$ is homogeneous.
For each $H \in \bbT(\Beul)$, we have $R^{(H)} \in \khlnd\big(\Beul^{(H)}\big)$.
\end{lemma}

\begin{proof}
Let $H \in \bbT(\Beul)$ and note that $I=\setspec{ m \in \Integ }{ m \deg(t) \in H+\G(R) }$ is a nonzero ideal of $\Integ$.
Let $n>0$ be such that $I = n\Integ$.  We claim that 
\begin{equation} \label {i765re5dsdv78394itjgh}
\text{there exists $i \in G$ such that $R_i \neq 0$ and $i+ n\deg(t) \in H$.}
\end{equation}
Indeed, we have $n \deg(t) \in \G(R) + H$, so we may choose $g \in \G(R)$ and $h \in H$ such that $n\deg(t) = g+h$.
Since $M=\setspec{ i \in G }{ R_i \neq 0 }$ is closed under addition and the subgroup of $G$ generated by $M$ is $\G(R)$,
we have $g = i_1 - i_2$ for some $i_1, i_2 \in M$. Since $H \in \bbT(\Beul)$, there exists $m \in \Nat\setminus\{0\}$ such that $m i_1 \in H$.
Define $i = (m-1)i_1 + i_2$, then $i \in M$ (so $R_i\neq0$) and $i+n\deg(t) = (m-1)i_1 + i_2 + g + h = mi_1+h \in H$,
which proves \eqref{i765re5dsdv78394itjgh}.

Choose $i$ as in \eqref{i765re5dsdv78394itjgh}, choose $\rho \in R_i \setminus \{0\}$,
and choose $d \in \Nat\setminus\{0\}$ such that $d i \in H$ ($d$ exists because $H \in \bbT(\Beul)$).
Define $\xi = \rho^d$ and $\tau = \rho t^n$; these are nonzero homogeneous elements of $R^{(H)}$ and $\Beul^{(H)}$ respectively,
and $\tau$ is transcendental over $R^{(H)}$.
We claim that 
\begin{equation}  \label {78ytu765ededyv3pmr0i9}
R^{(H)}[\tau] \subseteq \Beul^{(H)} \subseteq \big(R^{(H)}\big)_\xi [\tau] .
\end{equation}
The first part of \eqref{78ytu765ededyv3pmr0i9} is clear.
To prove the second part, note that
each element of $\Beul^{(H)}$ is a finite sum of elements of the form $rt^m$ where $r$ is a homogeneous element of $R\setminus\{0\}$, $m \in \Nat$,
and $\deg(rt^m) \in H$. For such an element $rt^m$, we have $\deg(r) + m\deg(t) \in H$ and $\deg(r) \in \G(R)$,
so $m \in I = n\Integ$ and hence $m = nk$ for some $k \in \Nat$.
Then $rt^m = r (t^n)^k = r \rho^{-k} \tau^k$.
Choose $\ell \in \Nat$ such that $\ell d -k \ge0$; then 
\begin{equation}  \label {p8utr4w3q2mc0nbtd}
\textstyle
r t^m = \frac{r \rho^{ \ell d-k }}{\xi^\ell} \tau^k \, .
\end{equation}
Since $\deg(rt^m), \deg(\xi), \deg(\tau) \in H$, \eqref{p8utr4w3q2mc0nbtd} gives $\deg( r \rho^{ \ell d-k } ) \in H$,
so $r \rho^{ \ell d-k } \in R^{(H)}$ and hence $rt^m \in \big(R^{(H)}\big)_\xi [\tau]$ by \eqref{p8utr4w3q2mc0nbtd} again.
This proves $\Beul^{(H)} \subseteq \big(R^{(H)}\big)_\xi [\tau]$ and completes the proof of \eqref{78ytu765ededyv3pmr0i9}.

Now  \eqref{78ytu765ededyv3pmr0i9} implies that $\big(\Beul^{(H)}\big)_\xi = S[\tau] = S^{[1]}$ where $S= \big(R^{(H)}\big)_\xi$,
so we may consider the derivative $\frac{d}{d\tau} : S[\tau] \to S[\tau]$.
Let $\Delta = \xi \frac{d}{d\tau} : (\Beul^{(H)} )_\xi \to (\Beul^{(H)} )_\xi$;
then $\Delta \in \hlnd \big( (\Beul^{(H)} )_\xi  \big)$ and $\ker \Delta = S$.
We claim that $\Delta(\Beul^{(H)}) \subseteq \Beul^{(H)}$.
To see this, consider the same element $rt^m$ as before and let us check that $\Delta(rt^m) \in \Beul^{(H)}$.
If $m=0$ then $\deg(r) = \deg(rt^m) \in H$, so $rt^m = r \in R^{(H)}$ and hence $\Delta(rt^m) = 0 \in \Beul^{(H)}$.
If $m>0$ then $k>0$, so \eqref{p8utr4w3q2mc0nbtd} gives
$$
\textstyle
(\Beul^{(H)} )_\xi \ni \Delta(rt^m)
=  \xi \frac{d}{d\tau}\big( \frac{r \rho^{ \ell d-k }}{\xi^\ell} \tau^k \big)
= \xi \cdot \frac{r \rho^{ \ell d-k }}{\xi^\ell} \cdot  k \tau^{k-1} = k r \rho^{d-1} t^{n(k-1)} \in \Beul .
$$
Since $\Beul \cap (\Beul^{(H)} )_\xi = \Beul^{(H)}$ by Lemma \ref{98767gh11dmdffkll0ahcchvr0vfesh8263543794663cnir}, we get $\Delta(rt^m) \in \Beul^{(H)}$.
Since every element of $\Beul^{(H)}$ is a finite sum of elements  $rt^m$ of this type, we obtain $\Delta(\Beul^{(H)}) \subseteq \Beul^{(H)}$.
Let $D : \Beul^{(H)} \to \Beul^{(H)}$ be the restriction of $\Delta$.
Then $D \in \hlnd\big( \Beul^{(H)} \big)$ and $\ker D = \Beul^{(H)} \cap \ker\Delta = \Beul^{(H)} \cap \big(R^{(H)}\big)_\xi$.
Since $R$ is factorially closed
in $\Beul = R^{[1]}$, it follows that $R^{(H)}$ is factorially closed in $\Beul^{(H)}$,
so $\Beul^{(H)} \cap \big(R^{(H)}\big)_\xi = R^{(H)}$ and hence $\ker D = R^{(H)}$.
We have $D(\tau) = \xi \frac{d}{d\tau}(\tau) = \xi \neq 0$, so $D \neq 0$ and hence $R^{(H)} \in \khlnd\big(\Beul^{(H)}\big)$.
\end{proof}

\begin{theorem}  \label {8237e9d1983hdjyev93}
Let $\bk$ be a field of characteristic $0$, let $B$ be an affine $\bk$-domain graded over $\bk$ by an abelian group $G$, and let $H \in \bbT(B)$.
\begin{enumerata}

\item For each $A \in \khlnd(B)$, we have $A^{(H)} \in \khlnd(B^{(H)})$.

\item The map $\khlnd(B) \to \khlnd(B^{(H)})$, $A \mapsto A^{(H)}$, is injective.

\item If $\hlnd(B) \neq \{0\}$ then $\hlnd(B^{(H)}) \neq \{0\}$.

\item If $G$ is torsion-free and $B$ is non-rigid then $B^{(H)}$ is non-rigid.

\end{enumerata}
\end{theorem}

\begin{proof}
(a) Let $A \in \khlnd(B)$.
Choose $D \in \hlnd(B) \setminus \{0\}$ such that $\ker D = A$, and choose a homogeneous element $t \in B$ such that $D(t)\neq 0$ and $D^2(t)=0$.
Consider the homogeneous element $a = D(t)$ of $A\setminus\{0\}$ and
let $\Beul = B_a$ and $R=A_a$; then $\Beul = R[t] = R^{[1]}$ by Lemma \ref{SliceThm}(b).
Since $\G(\Beul) = \G(B)$, we have $\bbT(\Beul) = \bbT(B)$ and hence $H \in \bbT(\Beul)$;
so  $R^{(H)} \in \khlnd\big(\Beul^{(H)}\big)$ by Lemma \ref{876543wqzxcvbnmC9l03ord}.
Choose $\Delta \in \hlnd( \Beul^{(H)} )$ such that $\ker(\Delta) = R^{(H)}$.

Since $H \in \bbT(B)$, there exists $n \in \Nat\setminus\{0\}$ such that $\deg(a^n) \in H$; let $\alpha = a^n$ and observe that $\Beul = B_\alpha$.
By Lemma \ref{ze3xqrctwvyjmkiu0u9inj7b2q3q01iqu},  we have $\Beul^{(H)} = \big( B_\alpha \big)^{(H)} = \big( B^{(H)} \big)_\alpha$, so
$\Delta \in \hlnd\big( ( B^{(H)} )_\alpha \big)$.
Since (by Cor.\ \ref{8Qp98781hy2uexihvcx2qw2wUxq3ew0pokjmnbvaw2345}) $B^{(H)}$ is $\bk$-affine,
Lemma \ref{6543wxcvbnc3wvlplpo0i9kijh7waq4rtmlb87hg} implies that there exists $\ell \in \Nat$
such that $\alpha^\ell\Delta : ( B^{(H)} )_\alpha \to ( B^{(H)} )_\alpha$ maps $B^{(H)}$ into itself.
Let $\delta : B^{(H)} \to B^{(H)}$ be the restriction of $\alpha^\ell\Delta$; then $\delta \in \hlnd( B^{(H)} )$
and $\ker(\delta) = B^{(H)} \cap \ker(\Delta)  = B^{(H)} \cap R^{(H)}$. To finish the proof of (a), it suffices to
verify that 
\begin{equation}  \label {f4ergc3t9ruw9eidmfcngh}
B^{(H)} \cap R^{(H)} = A^{(H)} \quad \text{and} \quad A^{(H)} \neq B^{(H)} .
\end{equation}
The fact that $A$ is factorially closed in $B$ implies that $B \cap R = B \cap A_a = A$, so 
$$
A^{(H)} \subseteq B^{(H)} \cap R^{(H)}  \subseteq B^{(H)} \cap B \cap R = B^{(H)} \cap A = A^{(H)} ,
$$
which proves the first part of \eqref{f4ergc3t9ruw9eidmfcngh}.
To prove the second part, recall that $D(t) \neq 0$.
Since $H \in \bbT(B)$, we can pick $m \in \Nat\setminus\{0\}$ such that $t^m \in B^{(H)}$.
Since  $D(t^m) = mt^{m-1}D(t)\neq0$, we have $t^m \notin A$, so $t^m \notin A^{(H)}$ and hence $A^{(H)} \neq B^{(H)}$.
So \eqref{f4ergc3t9ruw9eidmfcngh} is true and assertion (a) is proved.

(b) Suppose that $A_1,A_2 \in \khlnd(B)$ are such that $A_1^{(H)} =A_2^{(H)}$.
Consider a nonzero homogeneous element $a \in A_1$. Since $H \in \bbT(B)$, we can pick $m \in \Nat\setminus\{0\}$ such that $a^m \in A_1^{(H)}$.
Then $a^m \in A_2^{(H)} \subseteq A_2$, so $a \in A_2$ since $A_2$ is factorially closed in $B$.
This shows that $A_1 \subseteq A_2$, and $A_1=A_2$ follows by symmetry. This proves (b).

Part (c) is an obvious consequence of part (a).

(d) Since $G$ is torsion-free and $B$ is $\bk$-affine and non-rigid,
Lemma \ref{876543wsdfgxhnm290cAo} implies that $\hlnd(B) \neq \{0\}$, so $\hlnd(B^{(H)}) \neq \{0\}$ by (c), so $B^{(H)}$ is non-rigid.
\end{proof}

\begin{bigremark}
It is easy to produce examples showing that the converses to (c) and (d) are false.
If (a) is reformulated as an ``if-then'' statement, its converse is also false.
Moreover, the Theorem itself is false if we do not assume that $H \in \bbT(B)$.
For instance, consider the polynomial ring $\Comp[X,Y]$ with the $\Integ^2$-grading defined by declaring that $X$ and $Y$ are 
homogeneous of degrees $(1,0)$ and $(0,1)$, respectively. Let $B$ be the $\Integ^2$-graded subalgebra $\Comp[X^2,X^3,Y]$ of $\Comp[X,Y]$.
Then $0 \neq \frac{\partial}{\partial Y} \in \hlnd(B)$ but if $H$ is the subgroup of $\Integ^2$ generated by $(1,0)$
then $B^{(H)} = \Comp[X^2,X^3]$ is rigid. So all assertions of the Theorem are false for this choice of $H$.

\end{bigremark}

\begin{bigremark} \label {8765r984f5f7637w}
If $B$ is a ring of characteristic $0$ (graded or not), the 
{\it Makar-Limanov invariant\/} of $B$ (denoted $\ML(B)$)
is the intersection of the kernels of all elements of $\lnd(B)$.
If $B$ is graded, define the {\it homogeneous Makar-Limanov invariant\/} of $B$ (denoted $\HML(B)$)
to be the intersection of the kernels of all elements of $\hlnd(B)$.
Note that if $B$ is graded then $\ML(B) \subseteq \HML(B)$.

Thm \ref{8237e9d1983hdjyev93} implies that
if $\bk$ is a field of characteristic $0$ and $B$ an affine $\bk$-domain graded over $\bk$ by an abelian group $G$,
then $\HML(B^{(H)}) \subseteq \HML(B)^{(H)}$ for every $H \in \bbT(B)$.
\end{bigremark}

\section*{The set $\Xscr(B)$}

\begin{notation}  \label {7823689wuehdb4xei8ibc}
If $B$ is a domain of characteristic $0$ graded by an abelian group, let 
$$
\Xscr(B) = \setspec{ H \in \bbT(B) }{ \hlnd( B^{(H)} ) \neq \{0\} } .
$$
\end{notation}

\begin{bigremark} \label {6543wscvhyuhoblgmrne4b3vf432wsdfw}
Let $B$ be a domain of characteristic $0$ graded by an abelian group.
\begin{enumerata}

\item $\Xscr(B) \subseteq \setspec{ H \in \bbT(B) }{ \text{$B^{(H)}$ is non-rigid} }$

\item If $\G(B)$ is torsion-free and $B$ is a finitely generated algebra over a field,
then equality holds in part (a), by Lemma \ref{876543wsdfgxhnm290cAo}.

\end{enumerata}
\end{bigremark}

Our next objective is to show that if $H' \subseteq H$ belong to $\bbT(B)$ and $H \in \Xscr(B)$, then $H' \in \Xscr(B)$.
This is part (a) of Cor.\ \ref{8A273gLbv9ncxcavcbo3w982376re5}, and will easily follow from Thm \ref{8237e9d1983hdjyev93}.
The following notation is convenient.

\begin{notation}  \label {Zlknckjhi7evoenvnbCgEybgdCgbfpr9t839of}
Let $B$ be a domain graded by an abelian group $G$.  For each $H \in \bbT(B)$, we define 
$$
\bbT_H(B) = \setspec{ H' \in \bbT(B) }{ H' \subseteq H } .
$$
\end{notation}

\begin{lemma}  \label {3Pkj762cx7bXckayscxcvi8Y025esxcZxbkq}
Given a domain $B$ graded by an abelian group $G$,
\begin{equation*}
\bbT( B^{(H)} ) = \bbT_H(B) \quad \text{for all $H \in \bbT(B)$.}
\end{equation*}
\end{lemma}

\begin{proof}
Let $H \in \bbT(B)$.
Since $B^{(H)}$ is an $H$-graded ring, $\bbT( B^{(H)} )$ is a collection of subgroups of $H$;
clearly,  $\bbT_H(B)$ too is a collection of subgroups of $H$.
So it suffices to show that, for each subgroup $H'$ of $H$, we have $H' \in \bbT( B^{(H)} ) \Leftrightarrow H' \in \bbT(B)$.
Let $H'$ be a subgroup of $H$.
Let $U = H' \cap \G( B)$, $V = H \cap \G(B)$ and $W = \G(B)$, and note that $U \subseteq V \subseteq W$ and (since $H \in \bbT(B)$) $W/V$ is torsion;
so $V/U$ is torsion if and only if $W/U$ is torsion.
Since $\G(B^{(H)}) = H \cap \G(B)$ by Lemma \ref{GOfBH}, we have
$\G( B^{(H)} ) / ( H' \cap \G( B^{(H)} ) ) = (H \cap \G(B)) / ( H' \cap H \cap \G( B) ) = V/U$,
so $H' \in \bbT( B^{(H)} )$ $\Leftrightarrow$ $V/U$ is torsion.  Hence,
$H' \in \bbT( B^{(H)} ) \Leftrightarrow \text{$V / U$ is torsion} \Leftrightarrow \text{$W / U$ is torsion} \Leftrightarrow H' \in \bbT(B)$,
as desired.
\end{proof}

\begin{corollary} \label {8A273gLbv9ncxcavcbo3w982376re5}
Let $\bk$ be a field of characteristic $0$ and $B$ an affine $\bk$-domain graded over $\bk$ by an abelian group $G$.
\begin{enumerata}

\item If $H \in \Xscr(B)$ then $\bbT_H(B) \subseteq \Xscr(B)$.

\item Let $H \in \bbT(B)$ and let $T$ be the torsion subgroup of $H \cap \G(B)$.
If $B^{(H)}$ is non-rigid then so is $B^{(H')}$ for every $H' \in \bbT(B)$ such that $T \subseteq H' \cap \G(B) \subseteq H$.

\end{enumerata}
\end{corollary}

\begin{proof}
(a) Let $H \in \Xscr(B)$.  By Cor.\ \ref{8Qp98781hy2uexihvcx2qw2wUxq3ew0pokjmnbvaw2345}, $\Beul = B^{(H)}$ is an affine $\bk$-domain; clearly,
$\Beul$ is graded by $H$ and this grading is over $\bk$.
Since $\hlnd(\Beul) \neq \{0\}$ (because $H \in \Xscr(B)$), 
Thm \ref{8237e9d1983hdjyev93} implies that $\hlnd( \Beul^{(H')} ) \neq \{0\}$ for all $H' \in \bbT(\Beul)$;
we have $\bbT(\Beul) = \bbT_H(B)$ by Lemma \ref{3Pkj762cx7bXckayscxcvi8Y025esxcZxbkq},
and for each $H' \in \bbT_H(B)$ we have $H' \subseteq H$ and hence $\Beul^{(H')} = B^{(H')}$,
so $\hlnd( B^{(H')} ) \neq \{0\}$ (and hence $H' \in \Xscr(B)$) for all $H' \in \bbT_H(B)$. 

(b) As in part (a), the ring $\Beul = B^{(H)}$ is an affine $\bk$-domain and is graded over $\bk$ by $H$.
We have $\G( \Beul ) = H \cap \G(B)$ by Lemma \ref{GOfBH}, so the torsion subgroup of $\G( \Beul )$ is $T$ and the group $\Omega = \G( \Beul ) / T$ is torsion-free.
Let $\pi : \G(\Beul) \to \Omega$ be the canonical homomorphism of the quotient and let
$\bar\Beul$ denote the ring $\Beul$ equipped with the $\Omega$-grading $\bar\Beul = \bigoplus_{\omega \in \Omega} \bar\Beul_\omega$
where (for each $\omega \in \Omega$) $\bar\Beul_\omega = \bigoplus_{ \pi(i)=\omega } \Beul_i$.
Since $\bar\Beul$ is a non-rigid affine $\bk$-domain and $\Omega$ is torsion-free,
Lemma \ref{876543wsdfgxhnm290cAo} gives $\hlnd( \bar\Beul ) \neq \{0\}$ and hence $\Omega \in \Xscr(\bar\Beul)$.
Since the grading of $\bar\Beul$ is over $\bk$, (a) implies that  $\bbT_\Omega(\bar\Beul) \subseteq \Xscr(\bar\Beul)$,
i.e., $\bbT(\bar\Beul) \subseteq \Xscr(\bar\Beul)$.
In particular, $\bar\Beul^{(W)}$ is non-rigid for each $W \in \bbT(\bar\Beul)$;
we have the equalities of non-graded rings $\bar\Beul^{(W)} = \Beul^{(\pi^{-1}(W))} = (B^{(H)})^{(\pi^{-1}(W))} = B^{(\pi^{-1}(W))}$, so 
\begin{equation}  \label {98ujedws65tgawieryhqstrdfuuiwis09oiubvcccftbmhyxdb65}
\text{$B^{(\pi^{-1}(W))}$ is non-rigid for every  $W \in \bbT(\bar\Beul)$.}
\end{equation}
Consider $H' \in \bbT(B)$ such that $T \subseteq H' \cap \G(B) \subseteq H$; we show that $B^{(H')}$ is non-rigid.
We have $B^{(H')} = B^{(H'')}$ where we set $H'' = H' \cap \G(B)$.
Since $\ker\pi \subseteq H'' \subseteq \dom\pi$, if we define $W = \pi(H'')$ then $H'' = \pi^{-1}(W)$.
We claim that $\Omega / W$ is torsion.
To see this, consider $\omega \in \Omega$.
Choose $i \in \G( \Beul ) = H \cap \G(B)$ such that $\pi(i)=\omega$.
Since $H' \in \bbT(B)$, there exists $n\ge1$ such that $ni \in H' \cap \G(B) = H''$; then $n\omega = \pi(ni) \in \pi(H'') = W$,
which shows that $\Omega / W$ is torsion and hence $W \in \bbT(\bar\Beul)$.
By \eqref{98ujedws65tgawieryhqstrdfuuiwis09oiubvcccftbmhyxdb65},
$B^{(H')} = B^{(H'')} = B^{(\pi^{-1}(W))}$ is non-rigid.
\end{proof}

\section{The group $\bG(B)$}
\label {SectionTheGroupbGB}

This section is devoted to defining the subgroup $\bG(B)$ of $\G(B)$ and proving some of its properties.
We will see that $\bG(B)$ is related to locally nilpotent derivations in several ways;
this idea is first encountered in Thm \ref{3498huv8oe98y3876wgebf09}, and is further developed in subsequent sections.

Recall that we defined $\Spec^1(B) = \setspec{ \pgoth \in \Spec B }{ \haut\pgoth=1 }$ for any ring $B$.

\begin{definition} \label {oi8y81726et73e8f0i349rtfyh-sat}
Let $B = \bigoplus_{i \in G} B_i$ be a domain graded by an abelian group $G$.
Given $\pgoth \in \Spec B$,
the set $\bbM(B,\pgoth) = \setspec{ i \in G }{ B_i \nsubseteq \pgoth }$ is a submonoid of $\G(B)$.
The subgroup of $\G(B)$ generated by $\bbM(B,\pgoth)$ is denoted $\overline{\bbM}(B,\pgoth)$.
We define the subgroup $\bG(B)$ of $\G(B)$ by
$$
\bG(B) = \begin{cases}
\bigcap_{ \pgoth \in \Spec^1(B) }\overline{\bbM}(B,\pgoth) & \text{if $\Spec^1(B) \neq \emptyset$,} \\
\G(B) & \text{if $\Spec^1(B) = \emptyset$.}
\end{cases}
$$
We say that $B$ is {\it saturated in codimension $1$\/} if $\bG(B) = \G(B)$.
\end{definition}

\begin{lemma}  \label {o72ytg8ed623orh9yft8v1238293y}
Let $B = \bigoplus_{i \in G} B_i$ be a domain graded by an abelian group $G$.
If $\pgoth$ is a homogeneous prime ideal of $B$ then $\overline{\bbM}(B,\pgoth) = \G(B/\pgoth)$.
\end{lemma}

\begin{proof}
Clearly, $\G(B/\pgoth)$ is the subgroup of $G$ generated by
$\setspec{ i \in G }{ (B/\pgoth)_i \neq 0 } = \setspec{ i \in G }{ B_i \nsubseteq \pgoth } = \bbM(B,\pgoth)$,
so $\G(B/\pgoth) = \overline{\bbM}(B,\pgoth)$.
\end{proof}

\begin{theorem} \label {3498huv8oe98y3876wgebf09}
Let $B$ be a domain of characteristic $0$ graded by an abelian group~$G$.
\begin{enumerata}

\item For each $D \in \hlnd(B) \setminus \{0\}$,
there exists a height $1$ homogeneous prime ideal $\pgoth$ of $B$ satisfying:
\begin{enumerata}
\item some homogeneous element $t \in \pgoth$ satisfies $D(t) \neq 0$ and $D^2(t)=0$;
\item $\G(B/\pgoth) = \G( \ker D )$.
\end{enumerata}

\item $\bG(B) \subseteq \G( \ker D ) \subseteq \G(B)$ for all $D \in \hlnd(B)$.

\item If $B$ is saturated in codimension $1$ then $\G( \ker D ) = \G(B)$ for all $D \in \hlnd(B)$.

\end{enumerata}
\end{theorem}

\begin{proof}
(a) Let $D \in \hlnd(B) \setminus \{0\}$ and let $A = \ker D$.
Choose a homogeneous $t \in B$ such that $Dt \in A\setminus\{0\}$. 
Let $S$ be the set of homogeneous elements of $A \setminus \{0\}$.
Note that $S^{-1}B$ and $S^{-1}A$ are $G$-graded $\Rat$-domains, $S^{-1}D \in \hlnd( S^{-1}B ) \setminus \{0\}$ and $\ker(S^{-1}D) = S^{-1}A$.
Since $(S^{-1}D)(t)$ is a unit of $S^{-1}A$, we have $S^{-1}B = (S^{-1}A)[t] = (S^{-1}A)^{[1]}$ by Lemma \ref{SliceThm}.
The condition $S^{-1}B = (S^{-1}A)[t] = (S^{-1}A)^{[1]}$ implies that
$\Pgoth = t S^{-1}B$ is a height $1$ prime ideal of $S^{-1}B$.
By basic properties of localization, it follows that $\pgoth = B \cap \Pgoth$ is a height $1$ prime ideal of $B$.
It is clear that $\Pgoth$ is homogeneous, so $\pgoth$ is homogeneous.
We have $\G( B/\pgoth ) = \G( S^{-1}(B/\pgoth) ) = \G( S^{-1}B/tS^{-1}B ) = \G(S^{-1}A) = \G(A)$ and $t \in \pgoth$, so (a) is proved.

(b) Let $D \in \hlnd(B)$. It is clear that $\G( \ker D ) \subseteq \G(B)$, so it's enough to prove
that $\bG(B) \subseteq \G( \ker D )$.  If $D = 0$ then the claim is obvious.
If $D\neq0$ then part (a) implies that there exists a height $1$ homogeneous prime ideal $\pgoth$ of $B$
such that $\G( \ker D ) = \G(B/\pgoth)$.
We have $\G(B/\pgoth) = \overline{\bbM}( B, \pgoth )$ by Lemma \ref{o72ytg8ed623orh9yft8v1238293y}, so $\G( \ker D ) = \overline{\bbM}( B, \pgoth )$.
Since $\pgoth \in \Spec^1(B)$ we have $\Spec^1(B) \neq \emptyset$, so 
$\bG(B) =\bigcap_{ \qgoth \in \Spec^1(B) } \overline{\bbM}( B, \qgoth ) \subseteq\overline{\bbM}( B, \pgoth ) = \G( \ker D )$,
which proves (b).

(c) The assumption means that $\bG(B) = \G(B)$, so this follows from (b).
\end{proof}

We shall now give several properties of $\bG(B)$.
We begin by giving descriptions of $\bG(B)$ in two special cases: when $\G(B)$ is torsion (Lemma \ref{eibuw9x6hAvpjwua6w3rh61s298746})
and when $\G(B)$ is torsion-free (Prop.\ \ref{o98hfng54ye9c72ye6fv11vbdi38}).

\begin{lemma}  \label {eibuw9x6hAvpjwua6w3rh61s298746}
Let $B = \bigoplus_{i \in G} B_i$ be a domain graded by an abelian group $G$.
If $\G(B)$ is torsion and $\Spec^1(B) \neq \emptyset$ then 
$$
\bG(B) = \bigcap_{ \pgoth \in \Spec^1(B) } \!\! {\bbM}(B,\pgoth)
\ = \ \setspec{ i \in G }{ \text{$B_i \nsubseteq \pgoth$ for every $\pgoth \in \Spec^1(B)$} } .
$$
\end{lemma}

\begin{proof}
Since every submonoid of a torsion group is a group, we have $\bbM(B,\pgoth) = \overline{\bbM}(B,\pgoth)$
for each $\pgoth \in \Spec^1(B)$. The conclusion follows.
\end{proof}

\begin{proposition}  \label {o98hfng54ye9c72ye6fv11vbdi38}
Let $B$ be a domain graded by an abelian group $G$.
Assume that $\G(B)$ is torsion-free and let $Z$ denote the set of all height $1$ homogeneous prime ideals of $B$.
If $Z \neq \emptyset$ then $\bG(B) = \bigcap_{ \pgoth \in Z } \G( B/\pgoth )$,
and if $Z = \emptyset$ then $\bG(B) = \G(B)$.
\end{proposition}

\begin{proof}
Let the notation be $B = \bigoplus_{i \in G} B_i$.
Observe that the claim is true if $\Spec^1(B) = \emptyset$, and assume that $\Spec^1(B) \neq \emptyset$.
Then $\bG(B) = \bigcap_{\pgoth \in \Spec^1(B)} \overline{\bbM}(B,\pgoth)$, so it suffices to show that
\begin{equation} \label {2wsedqrwge0oihkvmnbtfrewqaswdfgxvbcnxmz}
\overline{\bbM}(B,\pgoth) = \begin{cases}
\G(B/\pgoth) & \text{if $\pgoth \in Z$,} \\
\G(B) & \text{if $\pgoth \in  \Spec^1(B) \setminus Z$.}
\end{cases}
\end{equation}
If $\pgoth \in Z$ then $\G(B/\pgoth) = \overline{\bbM}(B,\pgoth)$ by Lemma \ref{o72ytg8ed623orh9yft8v1238293y}.
If $\pgoth \in \Spec^1(B) \setminus Z$ then Lemma \ref{0Edlv2y6ctwbeofyvrbqleuo245612nxr7m54}(b)
(together with Rem.\ \ref{iuytytsxc3welqmo9cVbvbAfvVGhue83983teij} and the fact that $\G(B)$ is torsion-free)
implies that $B_i \cap \pgoth = \{0\}$ for all $i \in G$,
so $B_i \nsubseteq \pgoth$ for all $i \in G$ such that $B_i\neq0$, so $\overline{\bbM}(B,\pgoth) = \G(B)$.
\end{proof}

Part (d) of the following result gives a practical way to compute $\bG(B)$.

\begin{proposition}  \label {89fDCN584uKHJu7u59034uf}
Let $G$ be an abelian group and let $B = \bigoplus_{i \in G} B_i$ be a $G$-graded domain that is finitely generated as a $B_0$-algebra.
\begin{enumerata}

\item Let $S$ be a finite generating set of the $B_0$-algebra $B$ such that each element of $S$ is nonzero and homogeneous.
Let $\pgoth \in \Spec(B)$ and define  $D_\pgoth = \setspec{ \deg(x) }{ x \in S \setminus \pgoth }$.
Then ${\bbM}(B,\pgoth)$ is the submonoid of $G$ generated by $D_\pgoth$,
and $\overline{\bbM}(B,\pgoth)$ is the subgroup of $G$ generated by $D_\pgoth$.

\item The set $\setspec{ \overline{\bbM}( B, \pgoth ) }{ \pgoth \in \Spec(B) }$ is finite.

\item The set $\setspec{ \G(B/\pgoth) }{ \text{$\pgoth \in \Spec(B)$ and $\pgoth$ is homogeneous} }$ is finite.

\item Let $S$ be as in part {\rm(a)} and let $P$ be the set of all $\pgoth \in \Spec^1(B)$ satisfying $\pgoth \cap S \neq \emptyset$.  Then
$$
\bG(B) = \begin{cases}
\bigcap_{\pgoth \in P} \langle D_\pgoth \rangle  & \text{if $P \neq \emptyset$,} \\
\G(B)  & \text{if $P = \emptyset$.}
\end{cases}
$$

\end{enumerata}
\end{proposition}

\begin{proof}
It is straightforward to see that if $\pgoth \in \Spec(B)$ then ${\bbM}(B,\pgoth)$ is the submonoid of $G$ generated by $D_\pgoth$
(we leave this to the reader).
This implies that $\overline{\bbM}(B,\pgoth)$ is the subgroup of $G$ generated by $D_\pgoth$, so (a) follows.
Since $D = \setspec{ \deg(x) }{ x \in S }$ is a finite set,
and since $\setspec{ D_\pgoth }{ \pgoth \in \Spec(B) }$ is a collection of subsets of $D$,
we see that $\setspec{ D_\pgoth }{ \pgoth \in \Spec(B) }$ is a finite set; so (b) follows from (a).
In view of Lemma \ref{o72ytg8ed623orh9yft8v1238293y}, (c) follows from (b).
To prove (d), we first note that the claim is true if $\Spec^1(B)=\emptyset$; so we may assume that $\Spec^1(B) \neq \emptyset$,
which implies that
$\bG(B) = \bigcap_{\pgoth \in \Spec^1(B)} \overline{\bbM}(B,\pgoth)
= \bigcap_{\pgoth \in \Spec^1(B)} \langle D_\pgoth \rangle$ by (a).
Moreover, we have $\langle D_\pgoth \rangle = \G(B)$ for each $\pgoth \in  \Spec^1(B) \setminus P$, so the desired conclusion follows.
\end{proof}

\begin{bigremark}  \label {iutr3nic34ri34765io21p}
If $B$ is a noetherian domain graded by a finitely generated abelian group,
then Thm \ref{983765432wevvajwklw92980} implies that $B$ is finitely generated as a $B_0$-algebra,
so the hypothesis of Proposition \ref{89fDCN584uKHJu7u59034uf} is satisfied.
\end{bigremark}

\begin{example}
It is usually fairly easy to compute $\bG(B)$.
As an example, let $\bk$ be a field and $R = \bk[ U, V, X, Y ] = \bk^{[4]}$ the polynomial ring graded by $G = \Integ/42\Integ$,
where $\bk \subseteq R_{\bar0}$ and $U, V, X, Y$ are homogeneous of degrees $\overline{15}, \overline{14}, \overline{18}, \overline{24}$ respectively,
where $\overline{m}$ is the image of $m \in \Integ$ by the canonical epimorphism $\Integ \to G$.
Then $f = U^4 V^6 + X^3 Y^2 \in R_{\overline{18}}$, so $B = R/fR = \bk[u,v,x,y]$ is a $G$-graded domain.
We use Prop.\ \ref{89fDCN584uKHJu7u59034uf}(d) with $S = \{u,v,x,y\}$. Then $P = \{ (u,x) , (u,y) , (v,x) , (v,y) \}$.
If $\pgoth = (u,x)$ then $\langle D_\pgoth \rangle = \langle \overline{14}, \overline{24} \rangle = \langle \bar{2} \rangle$. 
In this way, we find $\setspec{ \langle D_\pgoth \rangle }{ \pgoth \in P } = \{ \langle \bar{2} \rangle, \langle \bar{3} \rangle \}$
and hence $\bG(B) = \langle \bar{2} \rangle \cap \langle \bar{3} \rangle  = \langle \bar{6} \rangle$.
Note that $\G(B) = G$, so $B$ is not saturated in codimension $1$.
\end{example}

Compare the following result to Lemma \ref{GOfBH}.

\begin{lemma}  \label {GandGbarOfBH}
If $B$ is a noetherian normal domain graded by an abelian group,
$$
\bG( B^{(H)} ) = H \cap \bG(B) \quad \text{for all $H \in \bbT(B)$.}
$$
\end{lemma}

\begin{proof}
Consider the map $f : \Spec(B) \to \Spec(B^{(H)})$, $\Pgoth \mapsto \Pgoth \cap B^{(H)}$,
and the sets $Z = \Spec^1(B)$  and $Z' = \Spec^1(B^{(H)})$.
By Lemma \ref{87654esxcvs9dfok33me4rtu5ortyo0934urud3r24}(b), $f$ is surjective and $f^{-1}(Z') = Z$. Thus,
\begin{equation}  \label {P892hXDrkynmXyZaMgnubvbn0MiNh587645r6uiw}
\text{$f|_Z : Z \to Z'$ is well-defined and surjective.}
\end{equation}
In particular, $Z = \emptyset$ $\Leftrightarrow$ $Z' = \emptyset$.
If $Z = \emptyset = Z'$ then $\bG(B) = \G(B)$ and $\bG(B^{(H)}) = \G(B^{(H)})$, so Lemma \ref{GOfBH} gives
$\bG(B^{(H)}) = \G(B^{(H)}) = H \cap \G(B) = H \cap \bG(B)$, as desired.
If $Z = \emptyset = Z'$ is false then $Z \neq \emptyset$ and $Z' \neq \emptyset$, so
$\bG(B) = \bigcap_{ \Pgoth \in Z } \overline{\bbM}( B, \Pgoth )$
and $\bG(B^{(H)}) = \bigcap_{ \pgoth \in Z' } \overline{\bbM}( B^{(H)}, \pgoth )$.
We claim:
\begin{equation}  \label {8xm1evgfc1e2910rcpawijgqzw}
H \cap \overline{\bbM}( B, \Pgoth ) = \overline{\bbM}( B^{(H)}, f(\Pgoth) ) \quad \text{for all $\Pgoth \in Z$.}
\end{equation}
Indeed, if $\Pgoth \in Z$ then
$$
\bbM( B^{(H)}, f(\Pgoth) )
= \setspec{ i \in H }{ B_i \nsubseteq f(\Pgoth) }
= \setspec{ i \in H }{ B_i \nsubseteq \Pgoth }
= H \cap \bbM(B,\Pgoth) ,
$$
so $\overline{\bbM}( B^{(H)}, f(\Pgoth) )
= \langle H \cap \bbM(B,\Pgoth) \rangle
= H \cap \langle \bbM(B,\Pgoth) \rangle
= H \cap \overline{\bbM}( B, \Pgoth )$ by Lemma \ref{jBh9823hedAcv78n2erj20},
which proves \eqref{8xm1evgfc1e2910rcpawijgqzw}.
Statements \eqref{8xm1evgfc1e2910rcpawijgqzw} and \eqref{P892hXDrkynmXyZaMgnubvbn0MiNh587645r6uiw} give the second and third equalities in:
\begin{align*}
H \cap \bG(B) &= \textstyle  \bigcap_{ \Pgoth \in Z } \big( H \cap \overline{\bbM}( B, \Pgoth ) \big)
=  \bigcap_{ \Pgoth \in Z } \overline{\bbM}( B^{(H)}, f(\Pgoth) ) \\
&= \textstyle \bigcap_{ \pgoth \in Z' } \overline{\bbM}( B^{(H)}, \pgoth )
= \bG(B^{(H)}) .
\end{align*}
\end{proof}

\section{From derivations of $B^{(H)}$ to derivations of $B$}
\label{sectionFromderivationsofBHtoderivationsofB}

The following is Notation 3.3 of \cite{Chitayat-Daigle:cylindrical}:

\begin{notation} \label {98w3gvddb9wudjcudj}
We write $(A,B) \in \EXT$ as an abbreviation for:
\begin{quote}
\it $A$ is a ring, $B$ is an $A$-algebra, and for every derivation $\delta : A \to A$ there exists a unique derivation $D: B \to B$
that makes the following diagram commute:
$$
\xymatrix@R=12pt{
B \ar[r]^D & B \\
A \ar[u] \ar[r]^{\delta} & A \, . \ar[u]
}
$$
\end{quote}
\end{notation}

\begin{example} \label {83uedxc5trfnbvcDFVUirybhnjYMYbgfAcfiu736}
If $L/K$ is a separable algebraic field extension, then $(K,L) \in \EXT$ (see \cite{Zar-Sam_CommAlg1}, Corollary $2'$, p.~125).
\end{example}

This section revolves around the following question: Given a $G$-graded domain $B$ of characteristic $0$
and $H_1,H_2 \in \bbT(B)$ satisfying $H_1 \subseteq H_2$, when do we have $(B^{(H_1)},B^{(H_2)}) \in \EXT$?

A priori, $(B^{(H_1)},B^{(H_2)}) \in \EXT$ is a condition on general derivations of $B^{(H_1)}$ and $B^{(H_2)}$.
To relate this to locally nilpotent or homogeneous derivations, we need:

\begin{lemma}  \label {piud29hdbfdu74heqoir923ft7fbc1oeyt}
Let $B$ be a domain of characteristic $0$ graded by an abelian group $G$.
Suppose that $H_1,H_2 \in \bbT(B)$ satisfy $H_1 \subseteq H_2$ and $(B^{(H_1)},B^{(H_2)}) \in \EXT$.
\begin{enumerata}

\item Suppose that $\delta : B^{(H_1)} \to B^{(H_1)}$ is a derivation and that $D: B^{(H_2)} \to B^{(H_2)}$ is the unique derivation of $B^{(H_2)}$ that extends it.
If $\delta$ is locally nilpotent (resp.\ homogeneous) then so is $D$.

\item If $B^{(H_1)}$ is non-rigid then so is $B^{(H_2)}$.

\item If $\hlnd(B^{(H_1)}) \neq \{0\}$ then $\hlnd(B^{(H_2)}) \neq \{0\}$.

\end{enumerata}
\end{lemma}

\begin{proof}
(a) If $\delta$ is locally nilpotent then, since $B^{(H_2)}$ is integral over $B^{(H_1)}$,
$D$ is locally nilpotent by Lemma \ref{ldkjxhcvi82ewdj0wd}.
Suppose that $\delta$ is homogeneous; let $d \in H_1$ be such that $\delta(B_j) \subseteq B_{j+d}$ for all $j \in H_1$.
Consider $b \in B_i \setminus\{0\}$ where $i \in H_2$.
Choose $n\ge1$ such that $ni \in H_1$.
We have $b^n \in B_{ni} \subseteq B^{(H_1)}$, so $nb^{n-1}D(b)= D(b^n) =\delta(b^n) \in B_{ni+d}$.
Since $nb^{n-1} \in B_{(n-1)i} \setminus \{0\}$ and $nb^{n-1}D(b) \in B_{ni+d}$,
it follows that $D(b) \in B_{i+d}$. This shows that $D(B_i) \subseteq B_{i+d}$ for all $i \in H_2$, i.e., $D$ is homogeneous.
This proves (a), and assertions (b) and (c) easily follow from (a).
\end{proof}

The next two results are Lemmas 3.5 and 3.6 of \cite{Chitayat-Daigle:cylindrical}, respectively.

\begin{lemma} \label {dh2763g5d4fc5a4dsErf56w7}
Let $B$ be a noetherian normal domain and $A$ a subring of $B$.
Suppose that $(\Frac A , \Frac B ) \in \EXT$ and that
there exists a family $( f_i )_{i \in I}$ of elements of $A \setminus \{0\}$ satisfying:
\begin{itemize}

\item $(A_{f_i},B_{f_i}) \in \EXT$ for every $i \in I$;

\item no height $1$ prime ideal of $B$ contains all $f_i$.

\end{itemize}
Then $(A,B) \in \EXT$.
\end{lemma}

\begin{lemma}  \label {nvo2983epdk02qkw0}
Let $A$ be a ring, $A[\mathbf{X}] = A[X_1,\dots,X_n] = A^{[n]}$, $f_1,\dots,f_n \in A[\mathbf{X}]$, $B = A[\mathbf{X}]/(f_1,\dots,f_n)$,
and $\pi : A[\mathbf{X}] \to B$ the canonical homomorphism of the quotient ring.
Let $P \in A[\mathbf{X}]$ be the determinant of the Jacobian matrix $\frac{\partial(f_1,\dots,f_n)}{\partial(X_1,\dots,X_n)}$.
If $\pi(P)$ is a unit of $B$, then $(A,B) \in \EXT$.
\end{lemma}

We shall now establish a sequence of results about rings graded by finite abelian groups.
All these results are superseded by the main result of this section, Thm \ref{o823y48t2309fnb28}.

\begin{lemma}  \label {5eghxWiC2j3mjfnjahvfcd56yui}
Let $B = \bigoplus_{i \in G} B_i$ be a $\Rat$-domain graded by a finite abelian group $G$.
Suppose that there exist homogeneous units $x_1, \dots, x_n \in B^*$ such that
$\langle \deg x_1 \rangle \oplus \cdots \oplus \langle \deg x_n \rangle = \G(B)$.
Then $(B_0,B) \in \EXT$.
\end{lemma}

\begin{proof}
We may assume that $\deg(x_i) \neq 0$ for all $i$.
For each $i=1,\dots,n$, let $e_i\ge2$ be the order of $\deg(x_i)$ in $\G(B)$.
Let $L = \setspec{ ( {\ell_1}, \dots, {\ell_n} ) \in \Nat^n }{ \text{$0 \le \ell_i < e_i$ for all $i = 1,\dots,n$} }$
and observe that the map $( {\ell_1}, \dots, {\ell_n} ) \mapsto \deg( x_1^{\ell_1} \cdots x_n^{\ell_n} )$ from $L$ to $\G(B)$ is bijective.
This implies that if $b$ is a nonzero homogeneous element of $B$ then there is a unique $( {\ell_1}, \dots, {\ell_n} ) \in L$
such that $\deg(b) =  \deg( x_1^{\ell_1} \cdots x_n^{\ell_n} )$; since $x_1,\dots,x_n$ are homogeneous units of $B$, it follows
that $b /( x_1^{\ell_1} \cdots x_n^{\ell_n} ) \in B_0$, so $b = a x_1^{\ell_1} \cdots x_n^{\ell_n}$ for some $a \in B_0$.
This shows that $B$ is a free $B_0$-module with basis
$\Beul = \setspec{ x_1^{\ell_1} \cdots x_n^{\ell_n} }{ ( {\ell_1}, \dots, {\ell_n} ) \in L }$.

Consider the ideal $I =(X_1^{e_1}-x_1^{e_1}, \dots, X_n^{e_n}-x_n^{e_n})$ of the polynomial ring $B_0[X_1,\dots,X_n]$
and the surjective homomorphism of $B_0$-algebras $\phi : B_0[X_1,\dots,X_n] \to B$ given by $\phi(X_i) = x_i$ ($1 \le i \le n$).
Clearly, $I \subseteq \ker\phi$.
If $F \in \ker\phi$ then the division algorithm shows that 
there exists $f \in B_0[X_1,\dots,X_n]$ such that $F \equiv f \pmod{I}$ and $\deg_{X_i}(f)<e_i$ for all $i=1,\dots,n$;
then $f \in \ker\phi$, so $f(x_1,\dots,x_n)=0$, so the fact that $\Beul$ is linearly independent over $B_0$ implies that $f=0$.
This shows that $\ker\phi=I$, so 
$$
B \isom B_0[X_1,\dots,X_n] / (X_1^{e_1}-x_1^{e_1}, \dots, X_n^{e_n}-x_n^{e_n})  .
$$
Let $P \in B_0[X_1,\dots,X_n]$ be the determinant of the Jacobian matrix
$\frac{\partial(X_1^{e_1}-x_1^{e_1}, \dots, X_n^{e_n}-x_n^{e_n})}{\partial(X_1,\dots,X_n)}$.
Since $\Rat \subseteq B$, $P(x_1,\dots,x_n) = \prod_{i=1}^n ( e_i x_i^{e_i-1} )$ is a unit of $B$;
so $(B_0,B) \in \EXT$ by Lemma \ref{nvo2983epdk02qkw0}.
\end{proof}

\begin{definition}
Let $B$ be a domain graded by a finite abelian group $G$.
An element $x$ of $B$ is {\it admissible\/} if there exist nonzero homogeneous elements $x_1,\dots,x_n$ of $B$
such that $x = x_1x_2\cdots x_n$ and $\G(B) = \langle \deg(x_1) \rangle \oplus \cdots \oplus \langle \deg(x_n) \rangle$.
\end{definition}

\begin{lemma}  \label {d8iuyhv1i288wsx9}
Let $B = \bigoplus_{i \in G} B_i$ be a $\Rat$-domain graded by a finite abelian group $G$.
If $x$ is an admissible element of $B$ and $m\ge1$ is such that $x^m \in B_0$,
then $( (B_0)_{x^m} ,B_{x^m} ) \in \EXT$.
\end{lemma}

\begin{proof}
There exist nonzero homogeneous elements $x_1,\dots,x_n$ of $B$
such that $x = x_1x_2\cdots x_n$ and $\G(B) = \langle \deg(x_1) \rangle \oplus \cdots \oplus \langle \deg(x_n) \rangle$.
Consider the $G$-graded domain $\Beul = B_{x^m}$ and note that $x_1, \dots, x_n$ are homogeneous units of $\Beul$ such that
$\langle \deg x_1 \rangle \oplus \cdots \oplus \langle \deg x_n \rangle = \G(B) = \G(\Beul)$.
By Lemma \ref{5eghxWiC2j3mjfnjahvfcd56yui}, we have $(\Beul_0,\Beul) \in \EXT$.
Since (by Lemma \ref{ze3xqrctwvyjmkiu0u9inj7b2q3q01iqu} with $H=0$) $\Beul_0 = (B_0)_{x^m}$, we are done.
\end{proof}

See Def.\ \ref{oi8y81726et73e8f0i349rtfyh-sat} for the concept of saturation in codimension $1$.

\begin{lemma}  \label {98876152r3ediwof93408eh}
Let $B = \bigoplus_{i \in G} B_i$ be a domain graded by a finite abelian group $G$.
The following are equivalent:
\begin{enumerati}

\item $B$ is saturated in codimension $1$;
\item no height $1$ prime ideal of $B$ contains all admissible elements of $B$.

\end{enumerati}
\end{lemma}

\begin{proof}
If $\Spec^1(B)=\emptyset$ then both (i) and (ii) are true, so (i)$\Leftrightarrow$(ii).
From now on, assume that $\Spec^1(B)\neq\emptyset$.
Note that $\bG(B) = \bigcap_{ \pgoth \in \Spec^1(B) } {\bbM}(B,\pgoth)$ by Lemma \ref{eibuw9x6hAvpjwua6w3rh61s298746}.

Suppose that (i) holds.
Then $\G(B) = \bG(B) = \bigcap_{ \pgoth \in \Spec^1(B) } {\bbM}(B,\pgoth)$.
Consider $\pgoth \in \Spec^1(B)$. Then $\G(B) = \bbM(B,\pgoth)$.
Choose $d_1,\dots,d_n \in \G(B)$ such that $\G(B) =  \langle d_1 \rangle \oplus \cdots \oplus \langle d_n \rangle$.
Then $d_1,\dots,d_n \in \bbM(B,\pgoth)$, so for each $j \in \{1,\dots,n\}$ we have $B_{d_j} \nsubseteq \pgoth$
and hence we can choose $x_j \in B_{d_j} \setminus \pgoth$.
Then $x = x_1 \cdots x_n$ is an admissible element of $B$ such that $x \notin \pgoth$.
So (ii) holds.

Conversely, suppose that (ii) holds.
To prove (i), it suffices to show that if $i \in \G(B)$ and $\pgoth \in \Spec^1(B)$, then $B_i \nsubseteq \pgoth$.
So consider $i \in \G(B)$ and $\pgoth \in \Spec^1(B)$.
By (ii), some admissible element $x$ of $B$ satisfies $x \notin \pgoth$.
There exist nonzero homogeneous elements $x_1,\dots,x_n$ of $B$
such that $x = x_1\cdots x_n$ and $\G(B) = \langle \deg(x_1) \rangle \oplus \cdots \oplus \langle \deg(x_n) \rangle$.
For some $j_1,\dots,j_n \in \Nat$, we have $i = j_1\deg(x_1) + \cdots + j_n \deg(x_n)$, so $x_1^{j_1} \cdots x_n^{j_n} \in B_i$.
Since $x \notin \pgoth$, we have  $x_1^{j_1} \cdots x_n^{j_n} \notin \pgoth$ and hence $B_i \nsubseteq \pgoth$.
So (i) holds.
\end{proof}

\begin{proposition}  \label {83hdtbe924239ued7owtee6}
Let $B = \bigoplus_{i \in G} B_i$ be a noetherian normal $\Rat$-domain graded by a finite abelian group $G$.
If $B$ is saturated in codimension $1$ then $(B_0,B) \in \EXT$.
\end{proposition}

\begin{proof}
Let $d$ be a positive integer such that $di=0$ for all $i \in G$,
and consider the subset $Y = \setspec{ x^d }{\text{$x$ is an admissible element of $B$}}$ of $B_0$.
Since $\Frac(B) / \Frac(B_0)$ is an algebraic extension of fields of characteristic $0$,
we have $(\Frac(B_0), \Frac(B)) \in \EXT$ by Ex.\ \ref{83uedxc5trfnbvcDFVUirybhnjYMYbgfAcfiu736}.
By Lemma \ref{d8iuyhv1i288wsx9}, $( (B_0)_y, B_y ) \in \EXT$ for every $y \in Y$.
By Lemma \ref{98876152r3ediwof93408eh}, no height $1$ prime ideal of $B$ contains $Y$.
So Lemma \ref{dh2763g5d4fc5a4dsErf56w7} implies that $(B_0,B) \in \EXT$.
\end{proof}

\begin{theorem}  \label {o823y48t2309fnb28}
Let $B$ be a noetherian normal $\Rat$-domain graded by a finitely generated abelian group $G$.
\begin{enumerata}

\item If $H \in \bbT(B)$ then $\big( B^{(H)} , B^{( H + \bG(B) )} \big) \in \EXT$.

\item If $H \in \Xscr(B)$ then $H + \bG(B) \in \Xscr(B)$.

\end{enumerata}
\end{theorem}

\begin{proof}
(a) Let $B = \bigoplus_{i \in G} B_i$ be the notation.
Consider the groups $X = H + \bG(B)$ and $Y = X/H$, and  the canonical homomorphism of the quotient group $\pi : X \to Y$.
Since $H \in \bbT(B)$, the group $Y = (H + \bG(B))/H \subseteq (H + \G(B))/H \isom \G(B) / (H \cap \G(B))$ is torsion and hence finite.
Let $A = B^{( X )}$. Since $A$ has an $X$-grading, it also has a $Y$-grading, namely
$A = \bigoplus_{ y \in Y } A_y $ where for each $y \in Y$ we define $A_y = \bigoplus_{ i \in X_y } B_i$ (with $X_y = \setspec{i \in X}{ \pi(i) = y }$).
{\it We always regard $A$ as being graded by the finite group $Y$.}
Since $A_0 = B^{(H)}$, it suffices to show that $(A_0,A) \in \EXT$.
By Lemma \ref{98767gh11dmdffkll0ahcchvr0vfesh8263543794663cnir},
$A$ is a noetherian normal $\Rat$-domain; so, by Prop.\ \ref{83hdtbe924239ued7owtee6},
it suffices to show that $A$ is saturated in codimension~$1$.\footnote{The
$X$-graded ring $B^{( X )}$ is not necessarily saturated in codimension $1$,
but we will show that  the $Y$-graded ring $A$ is.}

Consider the inclusion $A \subseteq B$; by Lemma \ref{87654esxcvs9dfok33me4rtu5ortyo0934urud3r24}(b),
$\pgoth \mapsto \pgoth \cap A$ is a surjective map from $\Spec^1(B)$ to $\Spec^1(A)$.
In particular,  $\Spec^1(B) \neq \emptyset$ $\Leftrightarrow$ $\Spec^1(A) \neq \emptyset$.
If $\Spec^1(A) = \emptyset$ then (by Def.\ \ref{oi8y81726et73e8f0i349rtfyh-sat}) $A$ is saturated in codimension $1$ and we are done.
So we may assume that $\Spec^1(A) \neq \emptyset$. It follows that $\Spec^1(B) \neq \emptyset$,
$\bG(B) = {\textstyle \bigcap_{ \pgoth \in \Spec^1(B)} \overline{\bbM}(B,\pgoth)}$ and
$\bG(A) = {\textstyle \bigcap_{ \qgoth \in \Spec^1(A)} \overline{\bbM}(A,\qgoth) }$.

Consider $\pgoth \in \Spec(B)$. 
We claim that 
\begin{equation}  \label {87654efg172193p01d2rbtfvcxvoo39}
\pi \big( \overline{\bbM}(B,\pgoth) \cap X \big) = \pi \big( {\bbM}(B,\pgoth) \cap X \big) .
\end{equation}
Indeed, if $W$ is a submonoid of $X$ then $\pi(W)$ is a submonoid of the finite group $Y$, so $\pi(W)$ is a group and hence
$\pi(W) = \langle \pi(W) \rangle = \pi( \langle W \rangle )$;
so $\pi \big( {\bbM}(B,\pgoth) \cap X \big) = \pi \big( \langle {\bbM}(B,\pgoth) \cap X \rangle \big)$.
We have $\langle \bbM(B,\pgoth) \cap X \rangle = \overline{\bbM}(B,\pgoth) \cap X$
by Lemma \ref{jBh9823hedAcv78n2erj20} (note that $X \in \bbT(B)$, because $X \supseteq H \in \bbT(B)$);
so \eqref{87654efg172193p01d2rbtfvcxvoo39} follows.

For each $i \in {\bbM}(B,\pgoth) \cap X$, we have
$B_i \nsubseteq \pgoth$, so $A_{ \pi(i) } \nsubseteq \pgoth$ because $B_i \subseteq A_{ \pi(i) }$,
so $A_{ \pi(i) } \nsubseteq \pgoth \cap A$ and hence $\pi(i) \in \bbM(A,\pgoth \cap A)$.
This shows that $\pi \big( {\bbM}(B,\pgoth) \cap X \big) \subseteq {\bbM}(A,\pgoth \cap A)$.
In view of \eqref{87654efg172193p01d2rbtfvcxvoo39}, and since $\pgoth \in \Spec(B)$ is arbitrary, this gives
\begin{equation} \label {8237ter7672893eufhf9}
\pi \big( \overline{\bbM}(B,\pgoth) \cap X \big) \subseteq \overline{\bbM}(A,\pgoth \cap A)  \quad \text{for every $\pgoth \in \Spec(B)$.}
\end{equation}
Keeping in mind that $\Spec^1(B) \neq \emptyset$, we obtain
\begin{align}  \label {uhgu561378e90dewjdbv2aX0}
Y &= \pi( H + \bG(B) ) = \pi( \bG(B) )
= \pi\bigg( {\textstyle \bigcap_{ \pgoth \in \Spec^1(B)} \overline{\bbM}(B,\pgoth) \bigg)} \\
\notag &= \pi\bigg( {\textstyle \bigcap_{ \pgoth \in \Spec^1(B)}  \big( \overline{\bbM}(B,\pgoth) \cap X \big) \bigg)} \\
\notag &\subseteq {\textstyle \bigcap_{ \pgoth \in \Spec^1(B)} \pi\big( \overline{\bbM}(B,\pgoth) \cap X \big)}
\subseteq {\textstyle \bigcap_{ \pgoth \in \Spec^1(B)}  \overline{\bbM}(A,\pgoth \cap A) } ,
\end{align}
where the last inclusion follows from \eqref{8237ter7672893eufhf9}.
Since $\pgoth \mapsto \pgoth \cap A$ is a surjective map from $\Spec^1(B)$ to $\Spec^1(A)$, we have
$$
{\textstyle \bigcap_{ \pgoth \in \Spec^1(B)} \overline{\bbM}(A,\pgoth \cap A) } 
= {\textstyle \bigcap_{ \qgoth \in \Spec^1(A)} \overline{\bbM}(A,\qgoth) } = \bG(A) ,
$$
so $Y \subseteq \bG(A)$ by \eqref{uhgu561378e90dewjdbv2aX0}.
We have $\bG(A) \subseteq \G(A) \subseteq Y$ by definition, so we obtain $\bG(A)=\G(A)$, i.e., $A$ is saturated in codimension $1$. This proves (a).

(b) Suppose that $H \in \Xscr(B)$. We have $H + \bG(B) \supseteq H \in \bbT(B)$, so $H + \bG(B) \in \bbT(B)$.
Since $\hlnd(B^{(H)}) \neq \{0\}$ (because $H \in \Xscr(B)$) and $\big( B^{(H)} , B^{( H + \bG(B) )} \big) \in \EXT$ (by (a)),
Lemma \ref{piud29hdbfdu74heqoir923ft7fbc1oeyt} implies that $\hlnd( B^{( H + \bG(B) )} ) \neq \{0\}$, so $H + \bG(B) \in \Xscr(B)$.
\end{proof}

\begin{corollary}
Let $B$ be a noetherian normal $\Rat$-domain graded by a finitely generated abelian group $G$.
If $B$ is saturated in codimension $1$ then the following hold.
\begin{enumerata}

\item $(B^{(H)},B) \in \EXT$ for every $H \in \bbT(B)$.

\item If $\Xscr(B) \neq \emptyset$ then  $G \in \Xscr(B)$ and $\hlnd(B) \neq \{0\}$.

\end{enumerata}
\end{corollary}

\begin{proof}
If $H \in \bbT(B)$ then $B^{( H + \bG(B) )} = B^{( H + \G(B) )} = B$,
so Thm \ref{o823y48t2309fnb28}(a) implies that $(B^{(H)},B) = \big( B^{(H)} , B^{( H + \bG(B) )} \big) \in \EXT$, which proves (a).
If $H \in \Xscr(B)$ then $H + \G(B) = H + \bG(B) \in \Xscr(B)$ by Thm \ref{o823y48t2309fnb28}(b), so $\hlnd(B^{(H + \G(B))}) \neq \{0\}$;
since $B^{(H + \G(B))} = B = B^{(G)}$, we get  $\hlnd(B) \neq \{0\}$ and $G \in \Xscr(B)$, proving (b).
\end{proof}

\section{Partial description of the set $\Xscr(B)$}
\label {sectionThesetXscrB}

If $B$ is a domain of characteristic $0$ graded by an abelian group,
let $\Mscr(B)$ denote the set of maximal elements of the poset $(\Xscr(B),\subseteq)$.

We shall now combine the main results of Sections \ref{sectionFromderivationsofBtoderivationsofBH} and \ref{sectionFromderivationsofBHtoderivationsofB}
to obtain Thm \ref{987654esdfghjdj4edoodji2d0we9iu9w12twf5}.
Part (a) of that result reduces the problem of describing $\Xscr(B)$ to that of describing $\Mscr(B)$, and (b) gives a small piece of information about $\Mscr(B)$.

\begin{theorem}  \label {987654esdfghjdj4edoodji2d0we9iu9w12twf5}
Let $\bk$ be a field of characteristic $0$ and $B$ an affine $\bk$-domain graded over $\bk$ by a finitely generated abelian group $G$.
\begin{enumerata}

\item $\Xscr(B) = \bigcup_{ H \in \Mscr(B) } \bbT_H(B)$

\item If $B$ is normal then each element $H$ of $\Mscr(B)$ satisfies $H \supseteq \bG(B)$.

\end{enumerata}
\end{theorem}

\begin{proof}
Let us abbreviate $\Xscr(B)$ and $\Mscr(B)$ to $\Xscr$ and $\Mscr$ respectively.
We may assume that $\Xscr \neq \emptyset$, otherwise the claim holds trivially.

(a) By Cor.\ \ref{8A273gLbv9ncxcavcbo3w982376re5}, we have  $\bbT_H(B) \subseteq \Xscr$ for each $H \in \Xscr$.
Thus, $\Xscr = \bigcup_{ H \in \Xscr } \bbT_H(B) $.
Since $G$ is finitely generated,
the poset $( \Xscr , \subseteq )$ satisfies the ascending chain condition.
This implies that for each $H_1 \in \Xscr$, there exists $H_2 \in \Mscr$ such that $H_1 \subseteq H_2$ and hence $\bbT_{H_1}(B) \subseteq \bbT_{H_2}(B)$.
Thus, $\Xscr = \bigcup_{ H \in \Xscr } \bbT_H(B) = \bigcup_{ H \in \Mscr } \bbT_H(B)$, which proves (a).

(b) Assume that $B$ is normal and consider $H \in \Mscr$.  Then $H + \bG(B) \in \Xscr$ by Thm \ref{o823y48t2309fnb28}(b).
So we have $\Mscr \ni H \subseteq H + \bG(B) \in \Xscr$, which implies $H=H + \bG(B)$ and hence $H \supseteq \bG(B)$.
This proves (b).
\end{proof}

\begin{corollary}  \label {87654edfgyudi30piwnf3bwvf2tyweuiklc9Bmo3vA34b6getw63ye}
Let $\bk$ be a field of characteristic $0$ and $B$ a normal affine $\bk$-domain graded over $\bk$ by a finitely generated abelian group $G$.
If $B$ is saturated in codimension $1$ then $\Xscr(B) = \emptyset$ or $\Xscr(B) = \bbT(B)$.
\end{corollary}

\begin{proof}
Suppose that $\Xscr(B) \neq \emptyset$.
Then $\Mscr(B) \neq \emptyset$ by Thm \ref{987654esdfghjdj4edoodji2d0we9iu9w12twf5}(a).
Consider $H \in \Mscr(B)$.
We have $H \supseteq \bG(B)$ by Thm \ref{987654esdfghjdj4edoodji2d0we9iu9w12twf5}(b),
and  $\bG(B) = \G(B)$ because $B$ is saturated in codimension $1$,
so $H \supseteq \G(B)$ and consequently $B^{(H)} = B$.
Since $H \in \Xscr(B)$ we have $\hlnd( B^{(H)} ) \neq \{0\}$ and hence $\hlnd(B) \neq \{0\}$.
It then follows from Thm \ref{8237e9d1983hdjyev93} that $\hlnd( B^{(H')} ) \neq \{0\}$ for all $H' \in \bbT(B)$,
i.e., $\Xscr(B) = \bbT(B)$.
\end{proof}

\begin{remark}
The conclusion ``$\Xscr(B) = \emptyset$ or $\Xscr(B) = \bbT(B)$'' of Cor.\ \ref{87654edfgyudi30piwnf3bwvf2tyweuiklc9Bmo3vA34b6getw63ye}
means that all rings $B^{(H)}$ with $H \in \bbT(B)$ have the same graded-rigidity status,
i.e., if one of them is graded-rigid then all of them are (see Def.\ \ref{8hnvnBbslijvxbgxsab8AbN7p03msS726rfXC}).
If we furthermore assume that $G$ is torsion-free then (by Lemma \ref{876543wsdfgxhnm290cAo}) all those rings also have the same rigidity status.
\end{remark}

\section{$\Integ$-gradings}
\label {sectionInteggradings}

In the first part of this section (up to Cor.\ \ref{o87r62390cJj99xjhs263bfe}),
we use the general results of the previous sections to derive consequences for $\Integ$-graded rings, 
adding new pieces of information where possible.
The material covered from Lemma \ref{87t6yfjcbbvo39f8y7-2} to the end of the section goes beyond the results of the previous sections.

The roles played by the groups $\G(B)$ and $\bG(B)$ in the previous sections are played by the
natural numbers $e(B)$ and $\bar e(B)$ in the context of $\Integ$-gradings.
These numbers are defined in Notation \ref{5Q43wsdbxhj3403o2iuwyetrd567b8niVuk54hwsert2hj3k},
and their relation to $\G(B)$ and $\bG(B)$ is given in Lemma \ref{983bvkslwlia293tfg62zgbnm3kry}.

\begin{notation}
If $B = \bigoplus_{i \in \Integ} B_i$ is a $\Integ$-graded ring and $d$ is a positive integer then
we define $B^{(d)} = \bigoplus_{i \in d\Integ} B_i$.  (In other words, $B^{(d)} = B^{(H)}$ with $H = d\Integ$.)
This is called the {\it $d$-th Veronese subring\/} of $B$.
\end{notation}

\begin{bigremark}  \label {So87gfXtkdfLO982678wgdyfcfu4ygf654W09}
We will make tacit use of the following facts.
\begin{enumerata}

\item \it If $\bk$ is a field and $B$ is a $\bk$-domain, then every $\Integ$-grading of $B$ is over $\bk$.

\item \it If $B$ is a $\Integ$-graded ring then $\bbT(B) = \setspec{ d\Integ }{ d \ge 1 }$ if the grading is non-trivial, 
and $\bbT(B) = \setspec{ d\Integ }{ d \ge 0 }$ if the grading is trivial.
Consequently, $\setspec{ B^{(H)} }{ H \in \bbT(B) } = \setspec{ B^{(d)} }{ d\ge1 }$ in both cases.

\end{enumerata}
Part (a) follows from Rem.\ \ref{ijn3eoocnnqppieudbc}.
In part (b), note that if the grading is trivial then $\setspec{ B^{(H)} }{ H \in \bbT(B) } = \{B\} = \setspec{ B^{(d)} }{ d\ge1 }$.
\end{bigremark}

It is obvious that Thm \ref{8237e9d1983hdjyev93} has the following consequence:

\begin{corollary}  \label {z0zoxjf5wehghwhedcgcn26sbcdtzezp8q0}
Let $\bk$ be a field of characteristic $0$, let $B$ be a $\Integ$-graded affine $\bk$-domain,
and let $d$ be a positive integer.
\begin{enumerata}

\item For each $A \in \khlnd(B)$, we have $A^{(d)} \in \khlnd(B^{(d)})$.

\item The map $\khlnd(B) \to \khlnd(B^{(d)})$, $A \mapsto A^{(d)}$, is injective.

\item If $B$ is non-rigid then $B^{(d)}$ is non-rigid.

\end{enumerata}
\end{corollary}

\begin{notation} \label {5Q43wsdbxhj3403o2iuwyetrd567b8niVuk54hwsert2hj3k}
Let $B = \bigoplus_{i \in \Integ} B_i$ be a $\Integ$-graded domain.
\begin{enumerate}

\item We define $e(B) = \gcd\setspec{ n \in \Integ }{ B_n \neq 0 }$.
The number $e(B)$ is called the {\it saturation index\/} of $B$.

\item Let $Z$ temporarily denote the set of all height $1$ homogeneous prime ideals of $B$.
We define $\bar e(B) = \lcm\setspec{ e(B/\pgoth) }{ \pgoth \in Z }$ if $Z \neq \emptyset$, and $\bar e(B) = e(B)$ if $Z = \emptyset$.
We call $\bar e(B)$ the {\it codimension $1$ saturation index\/} of $B$.

\end{enumerate}
\end{notation}

The following example shows that it is possible to have  $\bar e(B) = 0 \neq e(B)$.
We will see in Lemma \ref{87t6yfjcbbvo39f8y7-2} that the rings that satisfy that condition are very special.

\begin{example} \label {u6543wsxcvnj28c9w5i746bey} 
The ring $B = \Comp[U,V,X,Y] / (UY-VX)$ is a normal domain, and is $\Nat$-graded by declaring that $U,V,X,Y$ are homogeneous of degrees $0,0,1,1$.
Clearly, $e(B)=1$. 
Define $Z$ as in Notation \ref{5Q43wsdbxhj3403o2iuwyetrd567b8niVuk54hwsert2hj3k}.
The prime ideal $B_+ = \bigoplus_{ i > 0 } B_i$ belongs to $Z$ and $e( B/B_+ ) = 0$,
so $0 \in \setspec{ e(B/\pgoth) }{ \pgoth \in Z }$ and hence $\bar e(B)=0$.
\end{example}

\begin{bigremark}  \label {78rc12e2536w7894mcnfbsgnshfr}
Let $B = \bigoplus_{ i \in \Integ } B_i$ be a $\Integ$-graded domain,
and let $Z$ denote the set of all height $1$ homogeneous prime ideals of $B$.
\begin{enumerate}

\item We have $e(B) \in \Nat$ and $\G(B) = e(B)\Integ$.
Moreover, $e(B)=0$ if and only if the grading is trivial.

\item We have $\bar e(B) \in \Nat$ and $e(B) \mid \bar e(B)$. Moreover, $\bar e(B)=0$ if and only if one of the following holds:
\begin{itemize}

\item $e(B)=0$;

\item $e(B/\pgoth) = 0$ for some $\pgoth \in Z$;

\item the set $\setspec{ e(B/\pgoth) }{ \pgoth \in Z }$ is infinite.

\end{itemize}

\item {\it If $B$ is noetherian then $\setspec{ e(B/\pgoth) }{ \pgoth \in Z }$ is a finite set.}
Indeed, Prop.\ \ref{89fDCN584uKHJu7u59034uf} and Rem.\ \ref{iutr3nic34ri34765io21p} imply that $\setspec{ \G(B/\pgoth) }{ \pgoth \in Z }$ is a finite set.
By part (1) we have $\G(B/\pgoth) = e(B/\pgoth)\Integ$ for all $\pgoth \in Z$, so $\setspec{ e(B/\pgoth) }{ \pgoth \in Z }$ is finite.

\end{enumerate}
\end{bigremark}

\begin{lemma}  \label {983bvkslwlia293tfg62zgbnm3kry}
Let $B$ be a $\Integ$-graded domain.
\begin{enumerata}

\item $\G(B) = e(B)\Integ$ and $\bG(B) = \bar e(B)\Integ$.

\item $B$ is saturated in codimension $1$ if and only if $\bar e(B) = e(B)$.

\end{enumerata}
\end{lemma}

\begin{proof}
We already noted that $\G(B) = e(B)\Integ$.
Let $Z$ denote the set of all height $1$ homogeneous prime ideals of $B$.
If $Z = \emptyset$ then Prop.\ \ref{o98hfng54ye9c72ye6fv11vbdi38} gives $\bG(B)=\G(B)$ and the definition of $\bar e(B)$ gives $\bar e(B) = e(B)$,
so $\bar e(B) \Integ = e(B) \Integ = \G(B) = \bG(B)$, which proves (a) in the case $Z = \emptyset$.
If $Z \neq \emptyset$ then we claim that the following equalities hold:
$$
\bG(B) 
= \bigcap_{ \pgoth \in Z } \G( B/\pgoth )
= \bigcap_{ \pgoth \in Z } \bigg( e( B/\pgoth ) \Integ \bigg)
= \lcm \setspec{ e( B/\pgoth ) }{ \pgoth \in Z } \Integ
= \bar e(B)\Integ \, .
$$
Indeed, the first equality is Prop.\ \ref{o98hfng54ye9c72ye6fv11vbdi38},
the second one follows from the fact that $\G( B/\pgoth ) = e( B/\pgoth ) \Integ$ for each $\pgoth \in Z$,
the third one is the definition of ``lcm'' and the last one is the definition of $\bar e(B)$.
This proves (a).  Assertion (b) follows from (a) and Def.\ \ref{oi8y81726et73e8f0i349rtfyh-sat}.
\end{proof}

\begin{bigremark} \label {i765edcvbnmkiuyghytrew234tyuiojbvcxsertgbnjuh}
In \cite{Chitayat-Daigle:cylindrical}, saturation in codimension $1$ is only defined for $\Integ$-graded domains,
and the definition is as in part (b) of Lemma \ref{983bvkslwlia293tfg62zgbnm3kry}.
So our notion of saturation in codimension $1$ (Def.\ \ref{oi8y81726et73e8f0i349rtfyh-sat}) generalizes that of \cite{Chitayat-Daigle:cylindrical}.
\end{bigremark}

\begin{corollary}  \label {mMpPiq13096gxvbw5resip0tfs}
Let $B$ be a $\Integ$-graded domain and let $d \in \Nat\setminus\{0\}$.
\begin{enumerata}

\item $e( B^{(d)} ) = \lcm( e(B), d )$

\item If $B$ is noetherian and normal then $\bar e( B^{(d)} ) = \lcm( \bar e(B), d )$.

\end{enumerata}
\end{corollary}

\begin{proof}
Follows from Lemmas \ref{GOfBH}, \ref{GandGbarOfBH} and \ref{983bvkslwlia293tfg62zgbnm3kry}.
\end{proof}

\begin{corollary}  \label {65rfgs01oqksmnxcuyf5g42qywsvnf8}
Let $B$ be a $\Integ$-graded domain of characteristic $0$.
\begin{enumerata}

\item For each $D \in \hlnd(B) \setminus \{0\}$, there exists a height $1$ homogeneous prime ideal $\pgoth$ of $B$ satisfying:
\begin{enumerata}
\item some homogeneous element $t \in \pgoth$ satisfies $D(t) \neq 0$ and $D^2(t)=0$;
\item  $e(B/\pgoth) = e(\ker D)$.
\end{enumerata}

\item $e(B) \mid e( \ker D ) \mid \bar e(B)$ \ \  for all $D \in \hlnd(B)$.

\item If $B$ is saturated in codimension $1$ then $e( \ker D ) = e(B)$ for all $D \in \hlnd(B)$.

\end{enumerata}
\end{corollary}

\begin{proof}
Follows from Thm \ref{3498huv8oe98y3876wgebf09} and Lemma \ref{983bvkslwlia293tfg62zgbnm3kry}.
\end{proof}

Note that Corollaries 4.2--4.4 of \cite{DaiFreudMoser} are special cases of Cor.\ \ref{65rfgs01oqksmnxcuyf5g42qywsvnf8}(c),
which is itself a special case of  Thm \ref{3498huv8oe98y3876wgebf09}(c).

\begin{notation}  \label {NOTA8h7gt23er5w67xxde4c2t3gb9jrdhfgbs8u}
Given a prime number $p$ and any $m \in \Integ \setminus \{0\}$, we set 
$$
v_p(m) = \max\setspec{ i \in \Nat }{ \text{$p^i$ divides $m$} }.
$$
We also define $v_p(0) = \infty$, so $v_p(0) > v_p(m)$ for all $m \in \Integ\setminus\{0\}$.

Given $m \in \Nat \setminus \{0\}$ and $n \in \Nat$, we define the element $m \odiv n$ of $\Nat \setminus \{0\}$ by 
$$
\textstyle m \odiv n = \prod\limits_{p \in P(m,n)} p^{ v_p(m) }
$$
where $P(m,n)$ is the set of prime numbers $p$ such that $v_p(m) > v_p(n)$.
\end{notation}

\begin{lemma}  \label {8h7gt23er5w67xxde4c2t3gb9jrdhfgbs8u}
The operation $\odiv$ has the following properties.
\begin{enumerata}

\item If $m \in \Nat\setminus\{0\}$ then $m \odiv 1 = m$ and $m \odiv 0 = 1$.

\item If $m \in \Nat\setminus\{0\}$ and $n \in \Nat$ then
\begin{itemize}

\item $m \odiv n = 1 \Leftrightarrow m \mid n$,

\item $(m \odiv n) \mid m \mid \lcm(m \odiv n , n)$.

\end{itemize}

\item If $B$ is a $\Integ$-graded domain then $B^{(m \odiv e(B))} = B^{(m)}$ for all $m \in \Nat\setminus\{0\}$.

\end{enumerata}
\end{lemma}

\begin{proof}
Assertions (a) and (b) are easy.  We prove (c).
We have  $(m \odiv e(B)) \mid m$ by (b), so $B^{(m \odiv e(B))} \supseteq B^{(m)}$.
To prove the reverse inclusion, consider $x \in  B^{(m \odiv e(B))} \setminus \{0\}$ homogeneous of degree $d$.
Then $0 \neq B_d \subseteq B^{(m \odiv e(B))}$,
so $d$ is a multiple of both $m \odiv e(B)$ and $e(B)$, i.e., $\lcm( m \odiv e(B) , e(B) ) \mid d$. 
Since $m \mid \lcm( m \odiv e(B) , e(B) )$ by (b), we get $m \mid d$, so $x \in B^{(m)}$.
\end{proof}

The main result of Section \ref{sectionFromderivationsofBHtoderivationsofB} has the following consequence:

\begin{corollary}  \label {8y3655431qsxdsf2shccvmmlot405cvxj4u1}
Let $B$ be a $\Integ$-graded noetherian normal $\Rat$-domain.
Let $d \in \Nat\setminus\{0\}$ and define $d' = \gcd(d,\bar e(B))$ and $d'' = d' \odiv e(B)$.
Then 
$$
\big( B^{(d)} , B^{(d')} \big) \in \EXT \quad \text{and} \quad \big( B^{(d)} , B^{(d'')} \big) \in \EXT .
$$
\end{corollary}

\begin{proof}
We have $\big( B^{(d)} , B^{(d')} \big) \in \EXT$ by Thm \ref{o823y48t2309fnb28} and Lemma \ref{983bvkslwlia293tfg62zgbnm3kry}.
It follows that $\big( B^{(d)} , B^{(d'')} \big) \in \EXT$, because  $B^{(d'')} = B^{(d')}$
by Lemma \ref{8h7gt23er5w67xxde4c2t3gb9jrdhfgbs8u}(c).
\end{proof}

We shall now apply the ideas and results of Section \ref{sectionThesetXscrB} to $\Integ$-gradings.

\begin{notation} \label {98y82990294jend939t9uyf}
For each positive integer $d$, let $\Iscr_d = \{d,2d,3d,\dots\}$.
Let $\preceq$ be the partial order on $\Nat\setminus\{0\}$ defined by declaring that 
$d_1 \preceq d_2$ $\Leftrightarrow$ $\Iscr_{d_1} \subseteq \Iscr_{d_2}$ $\Leftrightarrow$ $d_2 \mid d_1$.

If $B$ is a $\Integ$-graded domain of characteristic $0$, let
$$
\NR(B) = \setspec{ d \in \Nat\setminus\{0\} }{ \text{$B^{(d)}$ is non-rigid} }
$$
and let $M(B)$ be the set of maximal elements of the poset $(\NR(B),\preceq)$.
Note that 
\begin{equation}  \label {oijn23w9e0cnbv2sxcxzq5at}
\text{$B$ is non-rigid $\iff$ $1 \in \NR(B)$ $\iff$ $1 \in M(B)$ $\iff$ $M(B) = \{1\}$.}
\end{equation}
\end{notation}

\begin{lemma}  \label {9876f5d43s09ijknm57ffd}
Let $\bk$ be a field of characteristic $0$ and $B$ a $\Integ$-graded affine $\bk$-domain.
\begin{enumerata}

\item If the grading of $B$ is nontrivial then
$$
\text{$\Xscr(B) = \setspec{ d\Integ }{ d \in \NR(B) }$ \ \ and \ \ $\Mscr(B) = \setspec{ d\Integ }{ d \in M(B) }$.}
$$

\item If the grading of $B$ is trivial, the following hold.
\begin{enumerata}

\item If $B$ is non-rigid then $\NR(B) = \Nat \setminus \{0\}$ and $M(B)=\{1\}$.

\item If $B$ is rigid then  $\NR(B) = \emptyset = M(B)$.

\end{enumerata}

\end{enumerata}
\end{lemma}

\begin{proof}
Assertion (b) is clear. We prove (a).
Since the grading of $B$ is not trivial,
we have $\bbT(B) = \setspec{ d\Integ }{ d \in \Nat\setminus\{0\} }$, and
Rem.\ \ref{6543wscvhyuhoblgmrne4b3vf432wsdfw}(b) gives the first equality in:
\begin{align*}
\Xscr(B) &= \setspec{ H \in \bbT(B) }{ \text{$B^{(H)}$ is non-rigid} } \\
&= \setspec{ d\Integ }{ \text{$d \in \Nat\setminus \{0\}$ and $B^{(d)}$ is non-rigid} } = \setspec{ d\Integ }{ d \in \NR(B) } .
\end{align*}
It follows that $\Mscr(B) = \setspec{ d\Integ }{ d \in M(B) }$, as desired.
\end{proof}

\begin{corollary}  \label {p98qy386e523brxhbfc6ghvc63x2345618q30}
Let $\bk$ be a field of characteristic $0$ and $B$ a $\Integ$-graded affine $\bk$-domain.
\begin{enumerata}

\item $\NR(B) = \bigcup_{ d \in M(B) } \Iscr_d$

\item If $B$ is normal then each element $d$ of $M(B)$ satisfies
$$
d \mid \bar e(B) \quad \text{and} \quad d \odiv e(B) = d .
$$

\end{enumerata}
\end{corollary}

\begin{proof}
If the grading of $B$ is trivial then the result easily follows from Lemma \ref{9876f5d43s09ijknm57ffd}(b).
Assume that the grading is nontrivial.
We have $\Xscr(B) = \setspec{ d\Integ }{ d \in \NR(B) }$ and $\Mscr(B) = \setspec{ d\Integ }{ d \in M(B) }$ by Lemma \ref{9876f5d43s09ijknm57ffd}(a),
so Thm \ref{987654esdfghjdj4edoodji2d0we9iu9w12twf5} and Lemma \ref{983bvkslwlia293tfg62zgbnm3kry} imply that (a) and (b$'$) are true, where 
\begin{enumerata}
\item[(b$'$)] If $B$ is normal then each element $d$ of $M(B)$ satisfies $d \mid \bar e(B)$.
\end{enumerata}
It remains to check that each $d \in M(B)$ satisfies $d \odiv e(B) = d$ (we don't need to assume that $B$ is normal for this part).
Let $d \in M(B)$ and define $d' = d \odiv e(B)$.
We have $B^{(d)} = B^{(d')}$ by Lemma \ref{8h7gt23er5w67xxde4c2t3gb9jrdhfgbs8u}(c), so clearly $d' \in \NR(B)$.
Lemma \ref{8h7gt23er5w67xxde4c2t3gb9jrdhfgbs8u}(b) gives $d' \mid d$, so $d \preceq d'$.
Thus $M(B) \ni d \preceq d' \in \NR(B)$, so $d'=d$, i.e., $d \odiv e(B) = d$.
\end{proof}

In the following result, the conclusion ``$\NR(B) = \emptyset$ or $\NR(B) = \Nat \setminus \{0\}$''
means that all Veronese subrings of $B$ have the same rigidity status, i.e.,
either $B^{(d)}$ is rigid for all $d \in \Nat\setminus\{0\}$ or $B^{(d)}$ is non-rigid for all $d \in \Nat\setminus\{0\}$.

\begin{corollary}  \label {1D0x2mv4yhhtfukCs1unc3t46bisgf9f2}
Let $\bk$ be a field of characteristic $0$ and $B$ a $\Integ$-graded normal affine $\bk$-domain.
If $B$ is saturated in codimension $1$ then $\NR(B) = \emptyset$ or $\NR(B) = \Nat \setminus \{0\}$.
\end{corollary}

\begin{proof}
Cor.\ \ref{87654edfgyudi30piwnf3bwvf2tyweuiklc9Bmo3vA34b6getw63ye} implies that $\Xscr(B) =\emptyset$ or $\Xscr(B)=\bbT(B)$,
so the claim follows from Lemma \ref{9876f5d43s09ijknm57ffd}.
\end{proof}

\begin{corollary} \label {o87r62390cJj99xjhs263bfe}
Let $\bk$ be a field of characteristic $0$ and $B$ a $\Integ$-graded normal affine $\bk$-domain.
\begin{enumerata}

\item For every $d \in \Nat\setminus\{0\}$, the following are equivalent:
\begin{enumerati}

\item $d \in \NR(B)$

\item $\Iscr_d \subseteq \NR(B)$

\item $\gcd(d,\bar e(B)) \in \NR(B)$

\item $d \odiv e(B) \in \NR(B)$.

\end{enumerati}

\item If $\bar e(B) \neq 0$ and $\NR(B) \neq \emptyset$ then $\bar e(B) \in \NR(B)$.

\end{enumerata}
\end{corollary}

\begin{proof}
To prove (a), consider $d \in \Nat\setminus\{0\}$.

If $d \in \NR(B)$ then Cor.\ \ref{p98qy386e523brxhbfc6ghvc63x2345618q30} implies that $d \in \Iscr_{m} \subseteq \NR(B)$ for some $m \in M(B)$.
Then $\Iscr_d \subseteq \Iscr_{m} \subseteq \NR(B)$, so (i) implies (ii). The converse is obvious, so (i) $\Leftrightarrow$ (ii).

If $d \in \NR(B)$ then Cor.\ \ref{p98qy386e523brxhbfc6ghvc63x2345618q30} implies that $d \in \Iscr_{m} \subseteq \NR(B)$ for some $m \in M(B)$,
and moreover $m \mid \bar e(B)$. So $m \mid \gcd(d, \bar e(B))$, so $\gcd(d, \bar e(B)) \in \Iscr_{m} \subseteq \NR(B)$
and hence $\gcd(d,\bar e(B)) \in \NR(B)$.
Conversely, if $\gcd(d,\bar e(B)) \in \NR(B)$ then  define $d' = \gcd(d,\bar e(B))$ and note that $d \in \Iscr_{d'} \subseteq \NR(B)$
(using that (i) implies (ii)), so $d \in \NR(B)$. So  (i) $\Leftrightarrow$ (iii).

Since $B^{(d)} = B^{( d \odiv e(B) )}$ by Lemma \ref{8h7gt23er5w67xxde4c2t3gb9jrdhfgbs8u}(c),
it is clear that (i) $\Leftrightarrow$ (iv).
This proves (a).

To prove (b), suppose that $\bar e(B) \neq 0$ and $\NR(B) \neq \emptyset$.
Pick $d \in \NR(B)$; then $\bar e(B) d \in \Iscr_d \subseteq \NR(B)$ by (a),
so $\gcd( \bar e(B) d, \bar e(B) ) \in \NR(B)$ again by (a), so $\bar e(B) \in \NR(B)$.
\end{proof}

\section*{More about $\bar e(B)$}

The ring $B$ of Ex.\ \ref{u6543wsxcvnj28c9w5i746bey} satisfies $\bar e(B) = 0 \neq e(B)$.
The following result shows that the rings that satisfy that condition are very special.

\begin{lemma}  \label {87t6yfjcbbvo39f8y7-2}
Let $\bk$ be a field and $B = \bigoplus_{ i \in \Integ } B_i$ a $\Integ$-graded affine $\bk$-domain such that $B_j\neq0$ for some $j>0$.
\begin{enumerata}

\item $\bar e(B) = 0$ if and only if 
$B$ is $\Nat$-graded and the prime ideal $B_+ = \bigoplus_{ i > 0 } B_i$ has height $1$.

\item If $B$ is normal and $\bar e(B)=0$, then there exist $s \in B_0 \setminus \{0\}$ and a homogeneous element $t \in B \setminus\{0\}$ of
positive degree such that $B_s = (B_0)_s[t] = \big( (B_0)_s \big)^{[1]}$.

\end{enumerata}
\end{lemma}

\begin{proof}
Before proving (a) or (b), we first note that
\begin{equation} \label {B09F3065-A83D-47C3-9B8B-98259E3CA797}
\text{$\bk \subseteq B_0$ and $B_0$ is $\bk$-affine.}
\end{equation}
Indeed, Rem.\ \ref{ijn3eoocnnqppieudbc} gives the first part of \eqref{B09F3065-A83D-47C3-9B8B-98259E3CA797}
and Cor.\ \ref{8Qp98781hy2uexihvcx2qw2wUxq3ew0pokjmnbvaw2345} with $H=0$ gives the second.

Proof of (a). Let $Z$ denote the set of all height $1$ homogeneous prime ideals of $B$.

If $B$ is $\Nat$-graded and $\haut(B_+) = 1$ then $B_+ \in Z$ and $e( B/B_+ ) = 0$, so $\bar e(B) = 0$
by part (2) of Rem.\ \ref{78rc12e2536w7894mcnfbsgnshfr}.

Conversely, assume that $\bar e(B) = 0$.
Since $\bar e(B) \neq e(B)$, we have $Z \neq \emptyset$;
since $B$ is noetherian, $\setspec{ e(B/\pgoth) }{ \pgoth \in Z }$ is a finite set by part (3) of Rem.\ \ref{78rc12e2536w7894mcnfbsgnshfr};
since $\lcm\setspec{ e(B/\pgoth) }{ \pgoth \in Z } = \bar e(B) = 0$, 
there exists $\pgoth \in Z$ such that $e( B/\pgoth ) = 0$.
Thus,
\begin{equation}  \label {3769ivbx910A28yt}
\text{$B_\ell \subseteq \pgoth$ for all $\ell \in \Integ \setminus \{0\}$.}
\end{equation}
Define $\pgoth_0 = \pgoth \cap B_0$.
To prove that $B$ is $\Nat$-graded, assume the contrary.
Since $B_j \neq 0$ for some $j>0$, this means that there exist $i,j \in \Integ$ such that $i < 0 < j$, $B_i \neq 0$ and $B_j \neq 0$.
Choose $x_i \in B_i \setminus \{0\}$ and $x_j \in B_j \setminus \{0\}$;
since $x_i,x_j \in \pgoth$, we have $x_i^j x_j^{|i|} \in \pgoth_0$, so $\pgoth_0 \neq 0$.
Since $x_j$ is transcendental over $B_0$, we have $\trdeg_{B_0}(B) > 0$;
since $\pgoth_0 \neq 0$ and $B_0$ is $\bk$-affine, we have $\trdeg_\bk(B_0/\pgoth_0) < \trdeg_\bk(B_0) \le \trdeg_\bk(B) - 1$.
On the other hand, we have $B_0/\pgoth_0 \isom B/\pgoth$ by \eqref{3769ivbx910A28yt},
so (using $\haut\pgoth=1$ and $B$ is $\bk$-affine) $\trdeg_\bk(B_0/\pgoth_0) = \trdeg_\bk(B/\pgoth) = \trdeg_\bk(B)-1$,
a contradiction.

So $B$ is $\Nat$-graded and consequently $B_+ = \bigoplus_{ i > 0 } B_i$ is a prime ideal of $B$.
Since $B_j\neq0$ for some $j>0$, we have $B_+ \neq 0$, so $\haut( B_+ ) \ge 1$.
We have $B_+ \subseteq \pgoth$ by \eqref{3769ivbx910A28yt}, so $\haut( B_+ ) = 1$, which proves (a).

Proof of (b). By (a), $B$ is $\Nat$-graded and $\haut(B_+) = 1$.
By \eqref{B09F3065-A83D-47C3-9B8B-98259E3CA797}, $\bk \subseteq B_0$ and both $B$ and $B_0$ are $\bk$-affine.
Since  $B/B_+ \isom B_0$ and $\haut(B_+)=1$, we have $\trdeg_{B_0}(B) = 1$.

We first prove the case where $B_0$ is a field.
Let $S = \bigcup_{ i \in \Nat }( B_i \setminus \{0\} )$ and $\Beul = S^{-1}B$.
Then $\Beul = \bigoplus_{i \in \Integ} \Beul_i$ where $\Beul_0$ is a field that contains $B_0$.
By Lemma  \ref{i8765redfvbnki8765rfghytrew123456789iuhv},
there exists a nonzero homogeneous element $t \in \Beul$ of degree $e(B)$ such that
$\Beul = \Beul_0[t^{\pm1}]$ is the ring of Laurent polynomials over $\Beul_0$.
In particular, $\trdeg_{\Beul_0}(\Beul) = 1$.
We also have $\trdeg_{B_0}(\Beul) = \trdeg_{B_0}(B) = 1$, so $\Beul_0 / B_0$ is an algebraic extension of fields.
Since $B$ is normal, it follows that $\Beul_0 \subseteq B$ and hence that $\Beul_0 = B_0$.
So $B \subseteq \Beul = B_0[t^{\pm1}]$. Since $B$ is $\Nat$-graded, it follows that $B \subseteq B_0[t]$.
Since $B_0$ is a field, it is clear that some monic polynomial $f(t) \in B_0[t]$ belongs to $B$; so $t$ is integral over $B$; since $B$ is normal and $t \in S^{-1}B$,
we get $t \in B$ and hence $B = B_0[t] = B_0^{[1]}$, as desired.

Now consider the general case.
This time, let $S = B_0 \setminus \{0\}$ and $\Beul = S^{-1}B = \bigoplus_{i \in \Nat} \Beul_i$, where $\Beul_0 = \Frac B_0$.
It is easy to see that $\Beul_+ \cap B = B_+$; by basic properties of localization, it follows that $\haut(\Beul_+) = \haut(B_+)=1$.
Since $\Beul_0$ is a field and $\Beul$ is an $\Nat$-graded normal affine $\Beul_0$-domain such that $\haut( \Beul_+ ) = 1$,
the preceding paragraph implies that $\Beul = \Beul_0[t] = \Beul_0^{[1]}$ where $t \in \Beul$ is homogeneous of degree $e(\Beul) = e(B)>0$.
Multiplying $t$ by an element of $S = B_0 \setminus \{0\}$, we may arrange that $t \in B$.
Since $B$ is $\bk$-affine, we have $B = \bk[ f_1(t), \dots, f_n(t) ]$ where $f_i(t) \in \Beul_0[t]$ for all $i$.
Clearly, there exists $s \in S$ such that all coefficients of $f_1(t), \dots, f_n(t)$ belong to $(B_0)_s$.
Thus, $B_s = (B_0)_s[t] = \big( (B_0)_s \big)^{[1]}$, which proves (b).
\end{proof}

\begin{corollary}  \label {i8976r53g47c8ru3yd3g23ev5627fu}
Let $\bk$ be a field of characteristic $0$ and $B = \bigoplus_{ i \in \Integ } B_i$ a $\Integ$-graded normal affine $\bk$-domain.
If $e(B) \neq 0 = \bar e(B)$ then there exists $D \in \hlnd(B) \setminus \{0\}$ such that $\ker(D) = B_0$.  In particular, $B$ is non-rigid.
\end{corollary}

\begin{proof}
First consider the case where $B_j\neq0$ for some $j>0$.
By Lemma \ref{87t6yfjcbbvo39f8y7-2}(b), there exist $s \in B_0 \setminus \{0\}$ and a homogeneous element $t \in B \setminus\{0\}$ of
positive degree such that $B_s = (B_0)_s[t] = \big( (B_0)_s \big)^{[1]}$.
For each $m \in \Nat$, $s^m \frac{d}{d t} : (B_0)_s[t] \to (B_0)_s[t]$ belongs to $\hlnd(B_s) \setminus \{0\}$.
Since $B$ is $\bk$-affine, we can choose $m$ such that $s^m \frac{d}{d t}$ maps $B$ into itself;
then $D = (s^m \frac{d}{d t})|_B : B \to B$ belongs to $\hlnd(B) \setminus \{0\}$ and $\ker(D) = B_0$.
This proves the result in this case.

If no $j>0$ is such that $B_j\neq0$, then
let $B^-$ be the ring $B$ with the opposite grading, i.e., $B^- = \bigoplus_{ i \in \Integ } B_i^-$ where $B_i^- = B_{-i}$ for each $i \in \Integ$.
We have $e(B^-) \neq 0 = \bar e(B^-)$, and there exists $j>0$ such that $B_j^- \neq 0$.
By the preceding paragraph, there exists $D \in \hlnd(B^-) \setminus \{0\}$ such that $\ker(D) = B_0^-$.
Clearly, the same $D$ satisfies  $D \in \hlnd(B) \setminus \{0\}$ and $\ker(D) = B_0$.
\end{proof}

\section*{The set $M(B)$}

By Cor.\ \ref{p98qy386e523brxhbfc6ghvc63x2345618q30}(a), the problem of describing $\NR(B)$ reduces to that of describing $M(B)$.
So it is interesting to ask what can be said about $M(B)$.
Let us agree that a {\it primitive set\/} is a subset $Y$ of $\Nat\setminus\{0\}$
such that the conditions $y,y' \in Y$ and $y \mid y'$ imply $y=y'$.
It is obvious that $M(B)$ is primitive, and Prop.\ \ref{finiteprimitiveset534823i93i} shows that it is finite.

\begin{proposition}  \label {finiteprimitiveset534823i93i}
Let $\bk$ be a field of characteristic $0$
and $B$ a $\Integ$-graded normal affine $\bk$-domain.
\begin{enumerata}

\item $M(B)$ is a finite primitive subset of the set of divisors of $\bar e(B)$.

\item If $e(B)=0$ then  $M(B)$ is either $\emptyset$ or $\{1\}$.

\item If $e(B) \neq 0$ and $\bar e(B)=0$ then $M(B) = \{1\}$.

\end{enumerata}
\end{proposition}

\begin{proof}
(b) If $e(B)=0$ then $B^{(d)} = B$ for all positive integers $d$, so $\NR(B)$ is either $\emptyset$ or $\Nat\setminus\{0\}$
and consequently (b) is true.

(c) If $e(B) \neq 0$ and $\bar e(B)=0$ then
Cor.\ \ref{i8976r53g47c8ru3yd3g23ev5627fu} implies that $B$ is non-rigid, so $1 \in \NR(B)$ and hence $M(B) = \{1\}$.

(a) It is obvious that $M(B)$ is a primitive set.
If $\bar e(B) \neq 0$ then  Cor.\ \ref{p98qy386e523brxhbfc6ghvc63x2345618q30}(b) implies that $M(B)$ is included in the set of divisors of $\bar e(B)$,
so (a) is true in this case.
If $\bar e(B)=0$ then (b) and (c) imply that $M(B)$ is either $\emptyset$ or $\{1\}$, so (a) is true in this case as well.
\end{proof}

\begin{bigremark}  \label {7cer78543432qq32cr98}
Let us say that a subset $Y$ of $\Nat\setminus\{0\}$ is {\it realizable\/}
if there exists a $\Integ$-graded normal affine $\Comp$-domain $B$ satisfying $M(B) = Y$.
Prop.\ \ref{finiteprimitiveset534823i93i} implies that every realizable set is finite and primitive,
but {\it we don't know if every finite primitive set is realizable.}
For instance, we don't know whether $\{6,10\}$ is realizable.
We will learn more about realizable sets in Section \ref{SEC:PhamBrieskornrings}, by studying Pham-Brieskorn rings.
\end{bigremark}

\section{Pham-Brieskorn rings}  \label {SEC:PhamBrieskornrings}

Recall from Cor.\ \ref{p98qy386e523brxhbfc6ghvc63x2345618q30}(a) that the problem of describing $\NR(B)$ reduces to that of describing $M(B)$.
The main result of this section (Thm \ref{eion98h38dhb2d203ojfbdt7}) describes the set $M(B)$ for Pham-Brieskorn rings $B$ that satisfy a certain condition.

\begin{notation}  \label {iu7654wd3fg56hxr39ik}
Given $\ba = (a_1,\dots,a_n) \in ( \Nat\setminus\{0\} )^{n}$ (where $n\ge3$), we define:
\begin{enumerate}

\item $L = \lcm(a_1,\dots,a_n)$

\item For each $i \in \{1,\dots,n\}$,
$J_i = \{1,\dots,n\} \setminus \{i\}$ and $L_i = \displaystyle \lcm\setspec{ a_j }{ j \in J_i }$.

\item $S( \ba ) = S(a_1,\dots,a_n) = \setspec{ j }{ 1 \le j \le n \text{ and } L_j \neq L }$ 

\item The {\it cotype\/} of $\ba$ is the cardinality of $S(\ba)$.  Note that 
$$
\cotype(\ba) = \cotype(a_1,\dots,a_n) \in \{0,1,\dots,n\} . 
$$

\item Let $D$ be the set of positive divisors of $\prod_{i=1}^n \frac{L}{L_i}$ and define a map $f : D \to (\Nat\setminus\{0\})^n$
by declaring that if $d \in D$ then $f(d) = (a_1', \dots, a_n') \in (\Nat\setminus\{0\})^n$ where:
\begin{equation}  \label {6gvhb3ag2n60b6tm9rn83h76b54}
a_i' =  \gcd( a_i, L_i ) \, {\textstyle \left[ \frac{L/L_i}{ \gcd( d, L/L_i) } \right] } , \quad 1 \le i \le n .
\end{equation}
It follows from \eqref{6gvhb3ag2n60b6tm9rn83h76b54} that $a_i = a_i' \gcd(d, L/L_i)$ (and hence $a_i' \mid a_i$) for all $i = 1,\dots,n$.
Consequently, $f(1) = (a_1, \dots, a_n)$.

\end{enumerate}
\end{notation}

\begin{notation}  \label {8r43qasd321qsxcvgttyhnemkrifv90olkmn}
If $\bk$ is a field of characteristic $0$, $n\ge3$ and $a_1,\dots,a_n \in \Nat\setminus\{0\}$, define
$B_{\bk; a_1,\dots,a_n} = \bk[X_1,\dots,X_n] / ( X_1^{a_1} + \cdots + X_n^{a_n} )$.
This is called a {\it Pham-Brieskorn ring}, and it is well known that
$$
\text{$B_{\bk; a_1,\dots,a_n}$ is a normal domain.}
$$
Consider the $\Nat$-grading of $\bk[X_1,\dots,X_n]$ such that (for each $i=1,\dots,n$)
$X_i$ is homogeneous of degree $d_i = L/a_i$, where $L = \lcm(a_1,\dots,a_n)$;
then  $X_1^{a_1} + \cdots + X_n^{a_n}$ is homogeneous and consequently $B_{\bk; a_1,\dots,a_n}$ is $\Nat$-graded.
We have $B_{\bk; a_1,\dots,a_n} = \bk[x_1,\dots,x_n]$, where (for each $i$) $x_i \in B_{\bk; a_1,\dots,a_n}$ denotes the canonical image of $X_i$.
Clearly, $x_i \neq 0$ is homogeneous of degree $d_i$, and if $n\ge4$ then $x_i$ is a prime element of $B_{\bk; a_1,\dots,a_n}$.
\end{notation}

When considering a Pham-Brieskorn ring $B_{\bk; a_1,\dots,a_n}$,
we shall use the notations $L$, $J_i$, $L_i$, $x_i$, $d_i$ and $f : D \to (\Nat \setminus \{0\} )^n$
(defined in \ref{iu7654wd3fg56hxr39ik} and \ref{8r43qasd321qsxcvgttyhnemkrifv90olkmn}) without necessarily recalling their definitions.
Note that all those notations are uniquely determined by $(a_1,\dots,a_n)$.

\begin{nothing*} \label {ConjPn}
For each $n\ge3$, let $\Gamma_n$ be the set of 
$(a_1,\dots,a_n) \in (\Nat \setminus \{0\})^{n}$  such that
$$
\text{$\min(a_1,\dots,a_n)>1$ and at most one $i \in \{1,\dots,n\}$ satisfies $a_i=2$.}
$$
It is well known and easy to see that
\begin{equation}  \label {7y2ew6t2ewoqpoqbgfccxv27}
\text{\it for every $(a_1,\dots,a_n) \in (\Nat \setminus \{0\})^{n} \setminus \Gamma_n$, the ring $B_{\Comp;a_1,\dots,a_n}$ is non-rigid.}
\end{equation}
It is conjectured (\cite[1.22]{Cheltsov-Park-Prokhorov-Zaidenberg_2021}, \cite{Kali-Zaid_2000}, \cite{FlennerZaidenberg_RatCurRatSing_2003})
that the following statement $\Pscr(n)$ is true for all $n\ge3$:
\begin{description}
\item[$\Pscr(n)$\,] \it \ \ For every $(a_1,\dots,a_n) \in \Gamma_n$, the ring $B_{\Comp;a_1,\dots,a_n}$ is rigid.
\end{description}
It is easy to see that if $B_{\Comp;a_1,\dots,a_n}$ is rigid then so is $B_{\bk;a_1,\dots,a_n}$ for any field $\bk$ of characteristic $0$.
So $\Pscr(n)$ can also be written as follows:
\begin{description}
\item[$\Pscr(n)$\,] \it \ \ 
\begin{minipage}[t]{115mm}
\it
For every field $\bk$ of characteristic $0$ and every $(a_1,\dots,a_n) \in \Gamma_n$, the ring $B_{\bk;a_1,\dots,a_n}$ is rigid.
\end{minipage}
\end{description}
The current status of this conjecture is as follows.
\begin{itemize}

\item $\Pscr(3)$ is true, by Lemma 4 of \cite{Kali-Zaid_2000}.

\item $\Pscr(4)$ is true, by \cite{ChitayatPhDThesis} and \cite{ChitayatDubouloz2025}.
(The PhD thesis \cite{ChitayatPhDThesis} shows, among other things, that if 
$B_{ \Comp ; 2,3,4,12 }$ and $B_{ \Comp ; 2,3,5,30 }$ are rigid then $\Pscr(4)$ is true.
These two rings are shown to be rigid in \cite{ChitayatDubouloz2025}.)

\item For $n>4$, only special cases of $\Pscr(n)$ are known to be true. (See for instance Lemma \ref{4uygrfhejrFDfhgsdTHbfd68827djGD763}.)

\end{itemize}
\end{nothing*}

\begin{lemma}[Cor.\ 4.16 of \cite{Chitayat-Daigle:PhamBrieskorn}]  \label {4uygrfhejrFDfhgsdTHbfd68827djGD763}
Let $\bk$ be a field of characteristic $0$, $n\ge3$ and $(a_1, \dots, a_n) \in \Gamma_n$.
If $\cotype(a_1,\dots,a_n)\ge n-2$, then $B_{\bk;a_1,\dots,a_n}$ is rigid.
\end{lemma}

\begin{notation} \label {98if9jvs9d772jfpd0vhf9w93uo0541hf9}
Let $B = \bigoplus_{i \in \Integ} B_i$ and $B' = \bigoplus_{i \in \Integ} B_i'$ be $\Integ$-graded rings.
\begin{enumerate}

\item We write $B \isom B'$ to indicate that there exists an isomorphism of rings $\phi : B \to B'$
satisfying $\phi(B_i) = B_i'$ for all $i \in \Integ$.

\item We write $B \isomdot B'$ to indicate that there exists an isomorphism of rings $\phi : B \to B'$
such that, for every $x \in B$, $x$ is homogeneous in $B$ if and only if $\phi(x)$ is homogeneous in $B'$
(but the degree of $x$ in $B$ is not necessarily equal to that of $\phi(x)$ in $B'$).
Note that if $B \isomdot B'$ then $e(B)$ is not necessarily equal to $e(B')$.
Also note that
$$
\text{if $B \isomdot B'$ then $B$ and $B'$ have the same rigidity status,}
$$
because whether or not $B$ (or $B'$) is rigid is independent of the grading.

\end{enumerate}
\end{notation}

The following result computes $e(B_{\bk;a_1,\dots,a_n})$ and $\bar e(B_{\bk;a_1,\dots,a_n})$,
and shows that if $d$ is a positive divisor of $\bar e(B_{\bk;a_1,\dots,a_n})$
then $( B_{\bk;a_1,\dots,a_n} )^{(d)}$ is again a Pham-Brieskorn ring.
More precisely, we have $( B_{\bk;a_1,\dots,a_n} )^{(d)} \isomdot B_{\bk; a_1', \dots, a_n'}$
where $a_1',\dots,a_n'$ are defined in the statement of part (d) of the Lemma.

\begin{lemma}  \label {o8y7t626rf53deg23489jru3sgbref6}
Let $\bk$ be a field of characteristic $0$, $n\ge3$ and $a_1,\dots,a_n \in \Nat\setminus\{0\}$.
\begin{enumerata}

\item The integers $\frac{L}{L_1}, \dots,\frac{L}{L_n}$ are pairwise relatively prime.

\item $e(B_{\bk;a_1,\dots,a_n}) = 1$ and $\bar e(B_{\bk;a_1,\dots,a_n}) = \prod_{i=1}^n \frac{L}{L_i}$.

\item $B_{\bk;a_1,\dots,a_n}$ is saturated in codimension $1$ if and only if $\cotype(a_1,\dots,a_n)=0$.

\item Let $d$ be a positive divisor of $\prod_{i=1}^n \frac{L}{L_i}$ and consider $f(d) = (a_1',\dots,a_n')$ as in part {\rm (5)} of Notation \ref{iu7654wd3fg56hxr39ik}.
Then $( B_{\bk;a_1,\dots,a_n} )^{(d)} \isomdot B_{\bk; a_1', \dots, a_n'}$.

\end{enumerata}
\end{lemma}

\begin{proof}
Let $B = B_{\bk;a_1,\dots,a_n}$.
For each $i \in \{1,\dots,n\}$, $x_i \neq 0$ is homogeneous of degree $d_i$;
so $e(B) = \gcd(d_1,\dots,d_n) = \gcd( \frac{L}{a_1}, \dots, \frac{L}{a_n} ) = \frac{L}{ \lcm(a_1,\dots,a_n) } = 1$.
Define 
$$
\hat d_i = \gcd\setspec{ d_j }{ j \in J_i } \quad \text{for each $i \in \{1,\dots,n\}$.}
$$
Then $\hat d_1 = \gcd(d_2,\dots,d_n) \linebreak[3]= \gcd( \frac{L}{a_2}, \dots, \frac{L}{a_n} ) = \frac{L}{ \lcm(a_2,\dots,a_n) } = \frac{L}{L_1}$,
and by the same argument we get
$$
\textstyle
\hat d_i = \frac{L}{L_i} \quad \text{for all $i \in \{1,\dots n\}$.}
$$
If $i \neq j$ then $\gcd( L/L_i , L/L_j ) = \gcd(\hat d_i , \hat d_j ) = \gcd(d_1,\dots,d_n) = 1$, which proves (a).

(b)  We already proved that $e(B) = 1$.
If $Z$ denotes the set of height $1$ homogeneous prime ideals of $B$ then:
\begin{itemize}

\item for each $i \in \{1,\dots,n\}$, some $\pgoth \in Z$ satisfies $x_i \in \pgoth$;
\item for each $\pgoth \in Z$, at most one $i \in \{1,\dots,n\}$ satisfies $x_i \in \pgoth$, 
and if  $x_i \in \pgoth$ then $e(B/\pgoth) = \gcd\setspec{d_j}{j \in J_i} = \hat d_i = L/L_i$.

\end{itemize}
It follows that $\setspec{ e(B/\pgoth ) }{ \pgoth \in Z } = \{ 1,  \frac{L}{L_1}, \dots, \frac{L}{L_n} \}$,
so $\bar e(B) = \lcm\big( {\textstyle \frac{L}{L_1}, \dots, \frac{L}{L_n} } \big) = \prod_{i=1}^n \frac{L}{L_i}$, the last equality by (a). 
This proves (b), and (c) follows from (b).

(d)  Let $d$ be a divisor of $\bar e(B) = \prod_{i=1}^n \frac{L}{L_i}$.
Define $\alpha_i = \gcd(d,\hat d_i)$ for each $i \in \{1,\dots,n\}$.
Let us prove that
\begin{equation}  \label {B09876543wexcvbjk0oijhnmlpAoiuy654321qws}
\text{for every $i \in \{1,\dots,n\}$,} \qquad \gcd(d_i,\alpha_i)=1, \quad d \mid d_i \alpha_i  \quad \text{and} \quad  a_i = a_i' \alpha_i.
\end{equation}
Indeed, let $i \in \{1,\dots,n\}$.
Since  $\alpha_i \mid \hat d_i$ and $\gcd(d_i,\hat d_i) = \gcd(d_1,\dots,d_n) = 1$, the first part of \eqref{B09876543wexcvbjk0oijhnmlpAoiuy654321qws} is clear.
To prove the second part, we first show that
\begin{equation}  \label {uhiw37ug3674g8239re0idk}
\textstyle \text{for all $j \in \{1,\dots,n\}$,\ \ $\frac{L}{L_j}$ divides $d_i \hat d_i$.}
\end{equation}
To see this, we note that $\frac{ d_i \hat d_i }{(L/L_j)} = \frac{ (L/a_i) (L/L_i) }{(L/L_j)} =  \frac{ L_j L }{a_i L_i} \in \Integ$
because if $i=j$ then $a_i \mid L$ and $L_i \mid L_j$,
and if $i \neq j$ then $a_i \mid L_j$ and $L_i \mid L$. This proves \eqref{uhiw37ug3674g8239re0idk}.
By (a), (b) and \eqref{uhiw37ug3674g8239re0idk}, it follows that $\bar e(B) = \prod_{j=1}^n \frac{L}{L_j}$ divides $d_i \hat d_i$.
Since $d \mid \bar e(B)$, we obtain $d \mid d_i \hat d_i$, so $d$ divides $\gcd(d_i d, d_i \hat d_i) = d_i \gcd(d,\hat d_i) = d_i \alpha_i$,
which proves the second part of \eqref{B09876543wexcvbjk0oijhnmlpAoiuy654321qws}.
For the third part, note that $a_i' \alpha_i = \gcd( a_i, L_i ) \,  {\textstyle \left[ \frac{L/L_i}{ \gcd( d, L/L_i) } \right] } \, \gcd(d, \hat d_i )$
and $\hat d_i = \frac{L}{L_i}$, so $a_i' \alpha_i = \gcd( a_i, L_i ) (L/L_i)$.
On the other hand, 
$\frac{ a_i L_i }{ \gcd(a_i,L_i) } = \lcm(a_i,L_i) = L$, so $a_i = \gcd( a_i, L_i ) (L/L_i) = a_i'\alpha_i$,
proving \eqref{B09876543wexcvbjk0oijhnmlpAoiuy654321qws}.

Consider a monomial $m = x_1^{\nu_1} \cdots x_n^{\nu_n}$ ($\nu_1,\dots,\nu_n \in \Nat$) such that $d$ divides $\deg(m)= d_1 \nu_1 + \cdots + d_n \nu_n$.
If $i \in \{1,\dots,n\}$ then  $\alpha_i \mid d_i \nu_i$ (because $\alpha_i \mid d \mid \deg(m)$ and $\alpha_i \mid \hat d_i$),
so $\alpha_i \mid \nu_i$ by \eqref{B09876543wexcvbjk0oijhnmlpAoiuy654321qws}. This shows that $m \in  \bk[x_1^{ \alpha_1 }, \dots, x_n^{ \alpha_n } ]$
and hence that $B^{(d)} \subseteq \bk[x_1^{ \alpha_1 }, \dots, x_n^{ \alpha_n } ]$.
By \eqref{B09876543wexcvbjk0oijhnmlpAoiuy654321qws}, we have $d \mid d_i \alpha_i$ and hence $x_i^{\alpha_i} \in B^{(d)}$ for all $i$,
so $B^{(d)} = \bk[x_1^{ \alpha_1 }, \dots, x_n^{ \alpha_n } ]$.
Observing that $( x_1^{ \alpha_1 } )^{a_1'} + \cdots + (x_n^{ \alpha_n })^{a_n'} = 0$ (again by \eqref{B09876543wexcvbjk0oijhnmlpAoiuy654321qws}),
one easily obtains that $\bk[x_1^{ \alpha_1 }, \dots, x_n^{ \alpha_n } ] \isomdot B_{\bk; a_1', \dots, a_n'}$.
This proves (d). 
\end{proof}

The aim of this section is to describe $M(B)$ for as many Pham-Brieskorn rings $B$ as we can.
(Recall from Cor.\ \ref{p98qy386e523brxhbfc6ghvc63x2345618q30}(a) that if we know $M(B)$ then we also know $\NR(B)$.)
We already know that $M(B_{\Comp; a_1, \dots, a_n}) = \{1\}$ whenever $( a_1, \dots, a_n ) \notin \Gamma_n$,
because in that case \eqref{7y2ew6t2ewoqpoqbgfccxv27} implies that $B_{\Comp; a_1, \dots, a_n}$ is non-rigid,
so $M(B_{\Comp; a_1, \dots, a_n}) = \{1\}$ by \eqref{oijn23w9e0cnbv2sxcxzq5at}.
So we may restrict our investigation to the case where $(a_1,\dots,a_n) \in \Gamma_n$.

\begin{theorem}  \label {eion98h38dhb2d203ojfbdt7}
Let $n\ge3$ and $(a_1,\dots,a_n) \in \Gamma_n$.
Let $f : D \to (\Nat\setminus\{0\})^n$ be the map determined by $(a_1,\dots,a_n)$ as in Notation \ref{iu7654wd3fg56hxr39ik}.
Assume that one of the following holds:
\begin{enumerati}  
\item $\Pscr(n)$ is true,
\item $B_{\Comp;f(d)}$ is rigid for each $d \in D$ such that $f(d) \in \Gamma_n$,
\end{enumerati}
and observe that {\rm (i)} implies {\rm (ii)}.
For each $\nu =1,2$, let $I_\nu$ be the set of all $i \in \{1,\dots,n\}$ satisfying $\gcd(a_i,L_i) = \nu$.
Then at most one element $i_0 \in I_2$ is such that $a_{i_0}=2$, and 
$$
M( B_{\Comp; a_1, \dots, a_n} ) = \begin{cases}
\setspec{ a_i }{ i \in I_1 } \cup \setspec{ \textstyle \frac{a_ia_j}4 }{ i,j \in I_2 \text{ and } i \neq j} & \text{if $i_0$ does not exist,} \\[1.5mm]
\setspec{ a_i }{ i \in I_1 } \cup \setspec{ \textstyle \frac{a_i}2}{ i \in I_2 \setminus \{i_0\} } & \text{if $i_0$ exists.} 
\end{cases}
$$
\end{theorem}

\begin{proof}
Let $B = B_{\Comp; a_1, \dots, a_n}$.
Since (i) or (ii) holds and (i) implies (ii), (ii) holds. 
Cor.\ \ref{p98qy386e523brxhbfc6ghvc63x2345618q30} implies that $M(B) \subseteq D$,
and Lemma \ref{o8y7t626rf53deg23489jru3sgbref6}(d) implies that if $d \in D$ then
$B^{(d)} \isomdot B_{\Comp; f(d)}$ and consequently $B^{(d)}$ and $B_{\Comp; f(d)}$ have the same rigidity status.
By (ii) and \eqref{7y2ew6t2ewoqpoqbgfccxv27}, $B_{\Comp; f(d)}$ is rigid if and only if $f(d) \in \Gamma_n$, so:
\begin{equation}  \label {9876543erfcnbvc9x87654}
\text{for each $d \in D$, \quad $d \in \NR(B)$ if and only if  $f(d) \notin \Gamma_n$.}
\end{equation}
Since  $f(1)=(a_1,\dots,a_n) \in \Gamma_n$, \eqref{9876543erfcnbvc9x87654} implies that $1 \notin \NR(B)$, so $1 \notin M(B)$.

Also note that $L = \lcm(a_i,L_i) = a_i L_i / \gcd(a_i,L_i)$ for each $i \in \{1,\dots,n\}$, so
\begin{equation}  \label {7236yhwebdcv5172wu}
L/L_i = a_i / \gcd(a_i,L_i) \quad \text{for all $i \in \{1,\dots,n\}$.}
\end{equation}
Define the sets
$$
U_1 = \setspec{ a_i }{ i \in I_1 } \quad \text{and} \quad  U_2 = \setspec{ \textstyle \frac{a_ia_j}4 }{ i,j \in I_2 \text{ and } i \neq j} .
$$

Suppose that $i \in I_1$. Then \eqref{7236yhwebdcv5172wu} gives $L/L_i = a_i$,
so $a_i \in D$ and the tuple $f(a_i) = (a_1', \dots, a_n')$ satisfies $a_i'=1$ by \eqref{6gvhb3ag2n60b6tm9rn83h76b54};
so $f(a_i) \notin \Gamma_n$ and hence $a_i \in \NR(B)$ by \eqref{9876543erfcnbvc9x87654}. Thus,
\begin{equation}  \label {987432sdxcpo09287h6hxcve}
U_1 \subseteq \NR(B) .
\end{equation}

Suppose that $1,2 \in I_2$.
For each $i \in \{1,2\}$ we have $\frac{a_i}2 = \frac L{L_i}$ by \eqref{7236yhwebdcv5172wu}, so $\frac{a_1 a_2}{4} = \frac L{L_1}  \frac L{L_2} \in D$
and, by \eqref{6gvhb3ag2n60b6tm9rn83h76b54}, the tuple $f(\frac{a_1 a_2}{4}) = (a_1', \dots, a_n')$ satisfies $a_1' = 2 = a_2'$.
So $f(\frac{a_1 a_2}{4}) \notin \Gamma_n$ and hence $\frac{a_1 a_2}{4} \in \NR(B)$ by \eqref{9876543erfcnbvc9x87654}.
More generally, the same argument shows that
\begin{equation}  \label {ZoiugvXEWXzfvWbYHuyt9nJd8r763bncnAt7}
U_2 \subseteq \NR(B) .
\end{equation}

Let $d \in M(B)$; since $M(B) \subseteq D$, $f(d) = (a_1', \dots, a_n')$ is defined and (by \eqref{9876543erfcnbvc9x87654}) does not belong to $\Gamma_n$;
so one of the following holds:
\begin{enumerata}
\item $a_i' = 1$ for some $i \in \{1,\dots,n\}$,
\item $a_i' = 2 = a_j'$ for some distinct $i,j \in \{1,\dots,n\}$.
\end{enumerata}

Suppose that (a) holds.
Then \eqref{6gvhb3ag2n60b6tm9rn83h76b54} implies that $\gcd(a_i,L_i)=1$ (so $i \in I_1$ and $a_i \in U_1$)
and $\frac L{L_i} = \gcd( d , \frac L{L_i} )$, so $\frac L{L_i} \mid d$.
We have $\frac L{L_i} = a_i$ by \eqref{7236yhwebdcv5172wu}, so $a_i \mid d$ and hence $d \preceq a_i$.
We have $a_i \in \NR(B)$ by \eqref{987432sdxcpo09287h6hxcve} and $d$ is a maximal element of $(\NR(B),\preceq)$, so $d=a_i \in U_1$.
So if (a) holds then $d \in U_1$.

Suppose that (b) holds; here, we might as well assume that $a_1' = 2 = a_2'$.
Since $a_i' \mid a_i$ for all $i$, we see that $a_1$ and $a_2$ are even.
Since $a_1 \mid L_2$ and $a_2 \mid L_1$, both $L_1,L_2$ are even, so $\gcd(a_i,L_i)$ is even for $i=1,2$.
By \eqref{6gvhb3ag2n60b6tm9rn83h76b54}, we obtain that for each $i =1,2$, $\gcd(a_i,L_i) = 2$  and $\frac L{L_i} = \gcd( d , \frac L{L_i} )$, and hence $\frac L{L_i} \mid d$.
It also follows that $1,2 \in I_2$, so $\frac{a_1 a_2}4 \in U_2$ and $\frac{a_1 a_2}4 \in \NR(B)$ by \eqref{ZoiugvXEWXzfvWbYHuyt9nJd8r763bncnAt7}.
Since $\gcd( \frac L{L_1} , \frac L{L_2} ) = 1$ and (by \eqref{7236yhwebdcv5172wu}) $\frac L{L_i} = \frac{a_i}2$ for $i=1,2$,
we have $\frac{a_1 a_2}4 \mid d$ and hence $M(B) \ni d \preceq \frac{a_1 a_2}4 \in \NR(B)$, which implies that $d = \frac{a_1 a_2}4 \in U_2$.
So if (b) holds then $d \in U_2$.
This proves the first inclusion in:
$$
M(B) \subseteq U_1 \cup U_2 \subseteq \NR(B) 
$$
where the second inclusion follows from \eqref{987432sdxcpo09287h6hxcve} and \eqref{ZoiugvXEWXzfvWbYHuyt9nJd8r763bncnAt7}.
If we set $t_i = a_i$ for $i \in I_1$ and $t_i = \frac{a_i}2$ for $i \in I_2$, then $t_i = \frac{L}{L_i}$ for all $i \in I_1 \cup I_2$, so
\begin{equation} \label {654edvbnmki87654esdcvbnhy543}
\text{the family $\big( t_i \big)_{i \in I_1 \cup I_2}$ is pairwise relatively prime}
\end{equation} 
by Lemma \ref{o8y7t626rf53deg23489jru3sgbref6}(a). 
Clearly,
\begin{equation} \label {bvcde45676543wsxcvbhuy6567ujnbxsw2}
U_1 = \setspec{ t_i }{ i \in I_1 }  \quad \text{and} \quad  U_2 = \setspec{ t_j t_k }{ j,k \in I_2 \text{ and } j \neq k} .
\end{equation}
Also note that $t_i>1$ for all $i \in I_1$, and that for $i \in I_2$ we have $t_i = 1$ $\Leftrightarrow$ $i=i_0$ (where $i_0$ is defined in the statement of the Proposition).

We will now use the following (easily verified) fact: {\it if $M(B) \subseteq P \subseteq \NR(B)$ and $P$ is a primitive set, then $M(B)=P$.}

Assume that $i_0$ does not exist. Then $t_i>1$ for all $i \in I_1\cup I_2$, so \eqref{654edvbnmki87654esdcvbnhy543} and \eqref{bvcde45676543wsxcvbhuy6567ujnbxsw2}
imply that $U_1 \cup U_2$ is primitive. Since $M(B) \subseteq U_1 \cup U_2 \subseteq \NR(B)$, we have $M(B)=U_1 \cup U_2$ in this case, as desired.

Assume that $i_0$ exists. Then $i_0 \in I_2$ and $t_{i_0}=1$.
For each $i \in I_2 \setminus \{i_0\}$ we have $t_i = t_{i_0}t_i \in U_2$; so $W \subseteq U_2 \subseteq \NR(B)$, where we define $W = \setspec{ t_i }{ i \in I_2 \setminus \{i_0\} }$.
Moreover, each element of $U_2 \setminus W$ is a product of two elements of $W$.
This implies that $M(B) \cap (U_2 \setminus W) = \emptyset$, because no element of $M(B)$ has a proper divisor that belongs to $\NR(B)$.
So $M(B) \subseteq U_1 \cup W \subseteq \NR(B)$.
Since the elements of $U_1 \cup W$ $ = \setspec{ t_i }{ i \in I_1 \cup I_2 \setminus \{i_0\} }$ are pairwise relatively prime and strictly larger than $1$,
$U_1 \cup W$ is primitive. So $M(B) = U_1 \cup W$, as desired.
\end{proof}

Corollaries \ref{8bcnc4n57xdklidsocmo77456} and \ref{C9B83uhAjh309fvn928Ia} are applications of the above Theorem.
In each case, we use Lemma \ref{4uygrfhejrFDfhgsdTHbfd68827djGD763} to verify that assumption (ii) of Thm \ref{eion98h38dhb2d203ojfbdt7} is satisfied.

\begin{corollary}  \label {8bcnc4n57xdklidsocmo77456}
Let $m,n \in \Nat$ be such that $n\ge3$ and $0 \le m \le n-2$.
Let $\ell_1, \dots, \ell_n \ge2$ be pairwise relatively prime integers such that $\ell_1, \dots, \ell_m$ are odd.
Let $B = B_{ \Comp; \ell_1, \dots, \ell_m, 2\ell_{m+1}, \dots, 2\ell_n }$.
Then 
$$
M(B) = \{ \ell_1, \dots, \ell_m \} \cup \setspec{ \ell_j \ell_k }{ m+1 \le j < k \le n } .
$$
\end{corollary}

\begin{proof}
Write $B = B_{\Comp; a_1, \dots, a_n}$ where $(a_1, \dots, a_n) = (\ell_1, \dots, \ell_m, 2\ell_{m+1}, \dots, 2\ell_n )$.
We have $L = 2 \prod_{i=1}^n \ell_i$ and the following hold for all $i \in \{1,\dots,n\}$:
$$
\textstyle
\text{\rien{$L_i = 2 \prod_{ j \in J_i } \ell_j$ , \quad }$L/L_i = \ell_i$ \quad and \quad
$\gcd(a_i,L_i) = \begin{cases} 1, & \text{if $i \le m$,} \\ 2, & \text{if $i > m$.} \end{cases}$}
$$
We claim that $(a_1,\dots,a_n)$ satisfies assumption (ii) of Thm \ref{eion98h38dhb2d203ojfbdt7}.
Indeed, it is easy to see that if $d \in D$ then $f(d) = (\ell_1', \dots, \ell_m', 2\ell_{m+1}', \dots, 2\ell_n')$
where $\ell_i' \mid \ell_i$ for all $i=1,\dots,n$. It follows that $\ell_1', \dots, \ell_n'$ are pairwise relatively prime.
So, if $f(d) \in \Gamma_n$ then $\cotype f(d) \ge n-1$, so $B_{\Comp;f(d)}$ is rigid by Lemma \ref{4uygrfhejrFDfhgsdTHbfd68827djGD763}.
Thus, assumption (ii) of Thm \ref{eion98h38dhb2d203ojfbdt7} is satisfied.

In the notation of Thm \ref{eion98h38dhb2d203ojfbdt7}, we have $I_1 = \{1,\dots,m\}$, $I_2 = \{m+1, \dots, n\}$ and $i_0$ does not exist,
so $M(B) = \{ \ell_1, \dots, \ell_m \} \cup \setspec{ \ell_j \ell_k }{ m+1 \le j < k \le n }$ by that result.
\end{proof}

\begin{corollary}  \label {C9B83uhAjh309fvn928Ia}
Let $Y$ be a finite primitive set whose elements are pairwise relatively prime. 
Then there exists a Pham-Brieskorn ring $B = B_{\Comp;a_1, \dots, a_n}$ satisfying $M(B) = Y$.
\end{corollary}

\begin{proof}
If $Y = \emptyset$ then let $B = B_{\Comp;3,3,3}$.
We have $\bar e(B)=1$ by Lemma \ref{o8y7t626rf53deg23489jru3sgbref6}(b), 
so Cor.\ \ref{p98qy386e523brxhbfc6ghvc63x2345618q30} implies that $M(B)$ is either $\emptyset$ or $\{1\}$.
By $\Pscr(3)$, $B$ is rigid, so  $1 \notin M(B)$ and hence $M(B) = \emptyset = Y$.

If $Y = \{1\}$ then let $B = B_{\Comp;1, 1, 1}$.  Since $B$ is non-rigid, $M(B) = \{1\} = Y$.

From now on, assume that $Y$ is neither $\emptyset$ nor $\{1\}$. Let $m = |Y| \ge 1$ and let $a_1, \dots, a_m$ be the 
distinct elements of $Y$. Then $a_1, \dots, a_m$ are pairwise relatively prime and $a_1, \dots, a_m \ge 2$.
Choose an integer $c\ge3$ such that  $a_1, \dots, a_m, c$ are pairwise relatively prime and let $B = B_{\Comp;a_1, \dots, a_m, c, c}$.

We claim that $(a_1,\dots,a_m,c,c)$ satisfies assumption (ii) of Thm \ref{eion98h38dhb2d203ojfbdt7}.
Indeed, it is easy to see that if $d \in D$ then $f(d) = (a_1',\dots,a_m',c,c)$ where $a_i' \mid a_i$ for all $i = 1,\dots,m$.
Note that $a_1', \dots, a_m', c$ are pairwise relatively prime.
If $f(d) \in \Gamma_{m+2}$ then $a_i'>1$ for all $i$, so $\cotype(a_1', \dots, a_m',c,c) = m$
and $B_{\Comp;a_1',\dots,a_m',c,c}$ is rigid by Lemma \ref{4uygrfhejrFDfhgsdTHbfd68827djGD763}.
Thus, assumption (ii) of Thm \ref{eion98h38dhb2d203ojfbdt7} is satisfied.

In the notation of Thm \ref{eion98h38dhb2d203ojfbdt7}, we have $I_1 = \{1, \dots, m\}$ and $I_2 = \emptyset$, so 
$M(B) =\{ a_1, \dots, a_m \} = Y$ by that result.
\end{proof}

We mentioned in Rem.\ \ref{7cer78543432qq32cr98} that
we don't know whether the set $\{6,10\}$ is realizable.
So it is interesting to note that $\{6\}$, $\{10\}$, $\{15\}$ and $\{6,10,15\}$ are realizable,
by Corollaries \ref{8bcnc4n57xdklidsocmo77456} and \ref{C9B83uhAjh309fvn928Ia}.

\section{Rigidity of $B_{(x)}$}
\label{sectionRigidityofBx}

The main result of this section is Thm \ref{7654cvhgfsd2d34fds23f42dfdghjk8l}. 
We begin with some preliminaries on ramification indices.

\begin{notation}
Let $X$ be an integral scheme that is normal and noetherian.
Then we may consider the group $\Div(X)$ of Weil divisors of $X$, and the group $\Div_\Rat(X) = \Rat \otimes_\Integ \Div(X)$ of $\Rat$-divisors of $X$.
If $f$ is a nonzero element of the function field $K(X)$ of $X$ then the divisor of $f$ is denoted $\div_X(f)$;
so  $\div_X : K(X)^* \to \Div(X)$ is a group homomorphism.
\end{notation}

\begin{nothing*} \label {Pkcnbvc9w3eidjojf0q9w}
(This is paragraph 5.4 in \cite{Chitayat-Daigle:cylindrical}, with minor edits.)
Let $B = \bigoplus_{i \in \Nat} B_i$ be an $\Nat$-graded noetherian normal domain such that the prime ideal $B_+ = \bigoplus_{i > 0} B_i$
has height at least $2$.
Let $X = \Spec B$ and $Y = \Proj B$, which are noetherian normal integral schemes.
We shall now define an injective $\Rat$-linear map $D \mapsto D^*$ from $\Div_\Rat(Y)$ to $\Div_\Rat(X)$.

Note that the function fields of $X$ and $Y$ satisfy $K(X) \supseteq K(Y)$.
Let $\YY$ be the set of homogeneous prime ideals of $B$ of height $1$.
Since $\haut(B_+)>1$, each element of $\YY$ is a point of $Y$; in fact we have $\YY = \setspec{ y \in Y }{ \dim\Oeul_{Y,y}=1 }$.
For each $\pgoth \in \YY$, $B_\pgoth \supset B_{(\pgoth)}$ is an extension of discrete valuation rings; let $e_\pgoth$ denote the ramification index of this extension.
Then $e_\pgoth \in \Nat\setminus \{0\}$.
If $v^Y_\pgoth : K(Y)^* \to \Integ$ and 
$v^X_\pgoth : K(X)^* \to \Integ$ denote the normalized\footnote{The word ``normalized'' means that the maps $v^Y_\pgoth$ and $v^X_\pgoth$ are surjective.}
valuations of $B_{(\pgoth)}$ and $B_\pgoth$ respectively,
then $v^X_\pgoth (\xi) = e_\pgoth v^Y_\pgoth(\xi)$ for all $\xi \in K(Y)^*$.
Let $C_\pgoth^Y$ (resp.~$C_\pgoth^X$) denote the closure of $\{ \pgoth \}$ in $Y$ (resp.\ in $X$).
Then $C_\pgoth^Y$ (resp.~$C_\pgoth^X$) is a prime divisor of $Y$ (resp.\ of $X$),
and every prime divisor of $Y$ is a $C^Y_\pgoth$ for some $\pgoth \in \YY$.
We define $( C_\pgoth^Y )^* = e_\pgoth C_\pgoth^X$ for each $\pgoth \in \YY$,
and extend linearly to  a $\Rat$-linear map $\Div_\Rat(Y) \to \Div_\Rat(X)$, $D \mapsto D^*$.
It is not hard to see that the linear map $D \mapsto D^*$ is injective and has the following property:
\begin{equation}
\label {PWSoidbfci27wyd}
\text{\it  $\big( \div_Y(\xi) \big)^* = \div_X(\xi)$ for all $\xi \in K(Y)^*$.}
\end{equation}
\end{nothing*}

\begin{lemma}  \label {7267e23ro98ifwhe8hfwjd}
Let the assumptions and notations be as in paragraph \ref{Pkcnbvc9w3eidjojf0q9w}.
Let $\pgoth \in \YY$ and consider the ramification index $e_\pgoth$ of $B_\pgoth$ over $B_{(\pgoth)}$.
\begin{enumerata}
\setlength{\itemsep}{1mm}

\item There exists a homogeneous element $g$ of $\pgoth$ such that $\pgoth B_\pgoth = g B_\pgoth$.

\item For any $g$ as in {\rm(a)} we have $\gcd( \deg(g) , e(B/\pgoth) ) = e(B)$.

\item  $\displaystyle e_\pgoth = \frac{  e(B/\pgoth) }{ e(B) }$

\end{enumerata}
\end{lemma}

\begin{proof}
(a) As $B_\pgoth$ is a discrete valuation ring and hence a principal ideal domain, there exists $g \in \pgoth$ such that $\pgoth B_\pgoth = g B_\pgoth$.
We have to show that $g$ can be chosen to be homogeneous.
Note that $v^X_\pgoth(g) = 1$.
Write $g = \sum_{i \in I} g_i$ where $I$ is a finite subset of $\Nat$ and $g_i \in B_i \setminus \{0\}$ for all $i \in I$.
Since $g \in \pgoth$, we have $g_i \in \pgoth$ and hence $v^X_\pgoth(g_i) \ge 1$ for all $i \in I$.
If $v^X_\pgoth(g_i) > 1$ for all $i \in I$ then $v^X_\pgoth(g) > 1$, which is not the case.
So there exists $i \in I$ such that  $v^X_\pgoth(g_i) = 1$. Then  $\pgoth B_\pgoth = g_i B_\pgoth$, which proves (a).

Before proving (b), we show that
\begin{equation}  \label {ikjhhgcgrd6widcjb9wvcr}
\begin{minipage}{.8\textwidth}
If $x,y \in B \setminus \{0\}$ are homogeneous with $x/y \in B_\pgoth^*$, then there exist homogeneous
$\alpha, \beta \in B \setminus \pgoth$ such that $x/y = \alpha/\beta$.
\end{minipage}
\end{equation}
To see this, we first choose $u \in B$ and $v \in B \setminus \pgoth$ such that $x/y = u/v$.
Write $u = \sum_{i \in \Nat} u_i$ and $v = \sum_{i \in \Nat} v_i$ where $u_i,v_i \in B_i$ for all $i$,
and choose $j \in \Nat $ such that $v_j \in B_j \setminus \pgoth$.
Then the equality $x\sum_{i \in \Nat} v_i = y \sum_{i \in \Nat} u_i$ implies that $xv_j = y u_i$ for some $i$.
So $x/y = u_i / v_j$. If $u_i \in \pgoth$ then $u_i/v_j$ belongs to $\pgoth B_\pgoth$, contradicting $x/y \in B_\pgoth^*$.
So $u_i \in B_i \setminus \pgoth$, which proves \eqref{ikjhhgcgrd6widcjb9wvcr}.

(b) Let $g$ be a homogeneous element of $\pgoth$ such that  $\pgoth B_\pgoth = g B_\pgoth$. 
Let $\delta = \gcd( \deg(g), e(B/\pgoth))$.
Pick any $i \in \Nat$ such that $B_i \neq 0$ and pick $h \in B_i \setminus \{0\}$.
For some $\nu \in \Integ$, we have $h/g^\nu \in B_\pgoth^*$;
so \eqref{ikjhhgcgrd6widcjb9wvcr} implies that $h/g^\nu = \alpha/\beta$ for some homogeneous $\alpha,\beta \in B \setminus \pgoth$.
Then $\delta \mid e(B/\pgoth) \mid \deg(\gamma)$ for each $\gamma \in \{\alpha,\beta\}$.
We have $\deg(h) - \nu \deg(g) = \deg(\alpha) - \deg(\beta)$, so $\delta$ divides $\deg(h) = i$.
This being true for each $i \in \Nat$ such that $B_i \neq0$, we conclude that $\delta \mid e(B)$.
Clearly, $e(B) \mid \delta$; so (b) is true.

(c) Let $g$ be as in (a) and (b).
Let $\epsilon = e(B/\pgoth)$ and $d = e(B/\pgoth)/e(B) = \epsilon/e(B)$.
Since $e(B) \mid \deg(g)$,
$\deg(g^d)$ belongs to $\epsilon \Integ$, which is the group generated by $\setspec{ i \in \Nat }{ (B/\pgoth)_i \neq 0 }$;
so there exist $i,j \in \Nat$ such that $(B/\pgoth)_i \neq 0$,  $(B/\pgoth)_j \neq 0$, and $\deg(g^d) = i-j$;
consequently, there exist homogeneous elements $u,v \in B \setminus \pgoth$ such that $\deg(g^d) = \deg(v) - \deg(u)$.
Then $\xi = g^d u / v $ belongs to the maximal ideal $\mgoth$ of $B_{(\pgoth)}$.
Let us prove that $\mgoth = \xi B_{(\pgoth)}$.

Consider an arbitrary nonzero element $\frac xs$ of $\mgoth$, where $x,s \in B \setminus \{0\}$ are homogeneous of the same degree and $s \notin \pgoth$. 
Define  $k = v^X_\pgoth\big( \frac xs \big)$.
Then $k>0$, because $\frac xs \in \mgoth \subseteq \pgoth B_\pgoth$; it follows that $\frac{x}{sg^k} \in B_\pgoth^*$,
so \eqref{ikjhhgcgrd6widcjb9wvcr} implies that $\frac{x}{sg^k} = \frac{b_m}{c_\ell}$ for some
$b_m \in B_m \setminus \pgoth$ and $c_\ell \in B_\ell \setminus \pgoth$.
Since  $c_\ell, b_m \in B \setminus \pgoth$, we have $\ell, m \in \epsilon\Integ$.
Since  $\frac xs = g^k \frac{b_m}{c_\ell}$, we have $k \deg(g) + m - \ell = \deg(x) - \deg(s) = 0$;
it follows that $\epsilon \mid k \deg(g)$, or equivalently, $d \mid k \frac{\deg(g)}{e(B)}$.
Part (b) shows that $d = \frac{\epsilon}{e(B)}$ is relatively prime to $\frac{\deg(g)}{e(B)}$, so $d \mid k$.
Write $k = k_0 d$ (where $k_0>0$), then 
$$
\frac xs = g^{k_0 d} \, \frac{b_m}{c_\ell} 
= \left(g^{d} \frac uv\right)^{k_0} \frac{v^{k_0}b_m}{u^{k_0}c_\ell} 
= \xi^{k_0} \frac{v^{k_0}b_m}{u^{k_0}c_\ell} 
$$
where $\frac{v^{k_0}b_m}{u^{k_0}c_\ell} \in B_{(\pgoth)}^*$.
Since $k_0>0$, we have $\frac xs \in \xi B_{(\pgoth)}$, which shows that $\mgoth = \xi B_{(\pgoth)}$.
Consequently, $e_\pgoth = v^X_\pgoth( \xi ) = v^X_\pgoth(  g^d u / v ) = d$, which proves (c).
\end{proof}

\begin{corollary}
Let the assumptions and notations be as in paragraph \ref{Pkcnbvc9w3eidjojf0q9w}.
Then $B$ is saturated in codimension $1$ if and only if $e_\pgoth = 1$ for all $\pgoth \in \YY$.
\end{corollary}

\begin{proof}
By Lemma \ref{983bvkslwlia293tfg62zgbnm3kry}, $B$ is saturated in codimension $1$ if and only if $e(B/\pgoth) = e(B)$ for all  $\pgoth \in \YY$.
So the claim follows from part (c) of Lemma \ref{7267e23ro98ifwhe8hfwjd}.
\end{proof}

We noted in the Introduction that Thm \ref{jh782uwsdvvhjkqw547182} is a result of Demazure with some extra pieces added to it.
We can now use Lemma \ref{7267e23ro98ifwhe8hfwjd} to prove assertion (a-iv), which is the only part of the Theorem that requires a justification.

\begin{nothing*}  \label {7tyzvcxzcvbcvQcksbvcstdio976543dfsxnbx48}
{\bf Proof of Theorem \ref{jh782uwsdvvhjkqw547182}.}
Part (b) is well known.
In part (a), the fact that $\Delta$ exists, is unique, and is an ample $\Rat$-divisor was proved in \cite[Thm 3.5]{Demazure_1988}.
Assertion (a-ii) is well known (see for instance  \cite[5.20(a)]{Chitayat-Daigle:cylindrical}), and (a-iii) follows from \cite[Cor.\ 5.20(d)]{Chitayat-Daigle:cylindrical}.
Let us prove (a-iv).
The article \cite{Demazure_1988} explains how to compute $\Delta$ from the given $t$, but it is easier to use the statement of \cite[Thm 5.9]{Chitayat-Daigle:cylindrical},
which asserts that $\Delta^* = \div_X(t)$.
By definition of the injective map $D \mapsto D^*$ (see \ref{Pkcnbvc9w3eidjojf0q9w}), this gives the first equality in:
$$
\textstyle
\Delta
= \sum_{ \pgoth \in Y^{(1)} } \frac{ v_\pgoth^X(t) }{ e_\pgoth }  C_\pgoth^Y 
= \sum_{ \pgoth \in Y^{(1)} } \frac{ v_\pgoth^X(t) }{ e(B/\pgoth) }  C_\pgoth^Y , 
$$
where the second equality follows from the assumption $e(B)=1$ together with part (c) of Lemma \ref{7267e23ro98ifwhe8hfwjd}.
So $\Delta \in \Div(Y)$ if and only if
\begin{equation}  \label {8A32uCyifbPMk8590mNBEWw6289}
\text{$e(B/\pgoth) \mid v_\pgoth^X(t)$ \ \  for all $\pgoth \in Y^{(1)}$.}
\end{equation}
Let $\pgoth \in Y^{(1)}$.
By Lemma \ref{7267e23ro98ifwhe8hfwjd}(a), there exists a homogeneous element $g$ of $\pgoth$ such that $\pgoth B_\pgoth = g B_\pgoth$.
Let $\nu = v_\pgoth^X(t)$ and note that $t/g^\nu \in B_\pgoth^*$;
by \eqref{ikjhhgcgrd6widcjb9wvcr}, we have $t/g^\nu = \alpha/\beta$ for some homogeneous $\alpha,\beta \in B \setminus \pgoth$.
Since $1 = \deg(t) = \nu \deg(g) + \deg(\alpha) - \deg(\beta)$ and $e(B/\pgoth)$ divides $\deg(\alpha)$  and $\deg(\beta)$,
$e(B/\pgoth)$ is relatively prime to $\nu =v_\pgoth^X(t)$.
This being true for each $\pgoth \in Y^{(1)}$, we obtain that
\eqref{8A32uCyifbPMk8590mNBEWw6289} is equivalent to $e(B/\pgoth)=1$ for all $\pgoth \in Y^{(1)}$,
which is itself equivalent to $B$ being saturated in codimension $1$.
This proves (a-iv), and completes the proof of Thm \ref{jh782uwsdvvhjkqw547182}.
\hfill\qedsymbol
\end{nothing*}

The following notation is used in the proof of Theorem \ref{7654cvhgfsd2d34fds23f42dfdghjk8l}:

\begin{notation}  \label {iwyr8762367ds2wxfwd2w2f3rfvtm0m00}
Given an $\Nat$-graded ring $A = \bigoplus_{ i \in \Nat } A_i$ and $d \in \Nat\setminus \{0\}$, we define the $\Nat$-graded ring
$A^{\langle d \rangle} = S = \bigoplus_{ i \in \Nat } S_i$ by setting $S_i = A_{di}$ for all $i \in \Nat$.
Note that $A^{\langle d \rangle}$ and $A^{(d)}$ are the same ring, but have different gradings: if $x \in S_i \setminus \{0\}$
then $x$ is a homogeneous element of both $A^{\langle d \rangle}$ and $A^{(d)}$,
the degree of $x$ in $A^{\langle d \rangle}$ is $i$, and the degree of $x$ in $A^{(d)}$ is $di$.
Observe that $A^{(d)} \isomdot A^{\langle d \rangle}$, in the notation of \ref{98if9jvs9d772jfpd0vhf9w93uo0541hf9}.
\end{notation}

\begin{theorem}  \label {7654cvhgfsd2d34fds23f42dfdghjk8l}
Let $\bk$ be a field of characteristic $0$ and
$B = \bigoplus_{n \in \Nat} B_n$ an $\Nat$-graded normal affine $\bk$-domain
such that the prime ideal $B_+ = \bigoplus_{n > 0} B_n$ has height at least $2$.
Let $x$ be a homogeneous prime element of $B$ of degree $d>0$.
If $B_{(x)}$ is non-rigid then so is $(B/xB)^{(d)}$.
\end{theorem}

\begin{bigremark}  \label {iytd43J45628c93j}
If $B$ and $x$ satisfy the hypothesis of Thm \ref{7654cvhgfsd2d34fds23f42dfdghjk8l}
then {\rm(a)}$\Leftrightarrow${\rm(b)}$\Rightarrow${\rm(c)}, where:
\begin{enumerata}

\item $B_{(x)}$ is non-rigid,

\item some basic cylinder of $\Proj(B)$ is included in $\bbD_+(x)$,

\item $\NR(B) \neq \emptyset$.

\end{enumerata}
Indeed, we have (b)$\Rightarrow${\rm(c)} by Thm \ref{edh83yf6r79hvujhxu6wrefji9e}(a).
By \ref{jbhgfdxfewae33w6q920we}, (a) is equivalent to the existence of a basic cylinder in $\Spec( B_{(x)} )$;
so the equivalence of (a) and (b) follows from the following easily checked observation:
{\it an open subset of $\bbD_+(x)$ is a basic open subset of $\Proj(B)$ if and only if its image
by the canonical isomorphism $\bbD_+(x) \to \Spec( B_{(x)} )$ is a basic open subset of $\Spec( B_{(x)} )$.}
\end{bigremark}

\begin{proof}[Proof of Thm \ref{7654cvhgfsd2d34fds23f42dfdghjk8l}]
Let $X = \Spec B$ and $Y = \Proj B$, and let us use the notations of paragraph \ref{Pkcnbvc9w3eidjojf0q9w}
(this is why we need to assume that $\haut(B_+)\ge2$ in Thm \ref{7654cvhgfsd2d34fds23f42dfdghjk8l}).
Let $\pgoth = xB \in \YY$ and consider the valuations $v_\pgoth^X : K(X) \to \Integ \cup \{\infty\}$ and $v_\pgoth^Y : K(Y) \to \Integ \cup \{\infty\}$
and the  ramification index  $e_\pgoth$.
Lemma \ref{7267e23ro98ifwhe8hfwjd} gives
\begin{equation}  \label {cjb738645q9mbr2hnbmadp28570}
\textstyle
\gcd(d, e(B/xB) ) = e(B) \quad \text{and} \quad e_\pgoth  = \frac{ e( B/xB ) }{ e(B) } .
\end{equation}

Consider $C_\pgoth^Y \subset Y$ and $C_\pgoth^X \subset X$ as in \ref{Pkcnbvc9w3eidjojf0q9w};
note that $C_\pgoth^Y = \bbV_+(x) \subset Y$ and let $U = \bbD_+(x) = Y \setminus C_\pgoth^Y$.
Let $R = B_{(x)}$ and note that $R$ is a subring of $K(Y)$.
We claim:
\begin{equation}  \label {E1nfo83aTcbpqc0gfrnxjg}
\text{if $f \in R \setminus \{0\}$ then $\div_U(f) \ge 0$  and  $v_\pgoth^Y(f) \le 0$.}
\end{equation}
Indeed, for each $\qgoth \in Y^{(1)} \setminus \{\pgoth\}$, we have $R = B_{(x)} \subseteq B_{(\qgoth)}$ and hence $v_\qgoth^Y(f) \ge0$; so $\div_U(f) \ge 0$.
Since $f \in R \setminus \{0\}$, we have $f = g/x^n$ for some $n \ge 0$ and $g \in B_{nd} \setminus \{0\}$.
If $v_\pgoth^Y(f) > 0$ then $f = xb/s$ where $b \in B \setminus \{0\}$ and $s \in B \setminus \pgoth$ are homogeneous and $\deg(xb) = \deg(s)$;
then $g/x^n = xb/s$, so $x^{n+1} \mid sg$; since $x$ is prime and $x \nmid s$ in $B$, we obtain $x^{n+1} \mid g$, which contradicts $g \in B_{nd} \setminus \{0\}$.
So \eqref{E1nfo83aTcbpqc0gfrnxjg} is true.
Consequently, the map
$$
\deg_R : R \to \Nat \cup \{ - \infty \}, \quad \text{$\deg_R(f) = -v_\pgoth^Y(f)$ for all $f \in R$}
$$
is a well-defined  degree function.
Consider the associated graded ring $\Gr(R)$ determined by $(R,\deg_R)$;
recall that $\Gr(R) = \bigoplus_{n \in \Nat} R_{\le n} / R_{<n}$ where for each $n \in \Nat$ we set
$$
R_{\le n} = \setspec{ r \in R }{ \deg_R(r) \le n } \quad \text{and} \quad R_{< n} = \setspec{ r \in R }{ \deg_R(r) < n } \quad \text{(so $R_{<0} = 0 $).}
$$
Since $v_\pgoth^Y : K(Y) \to \Integ \cup \{\infty\}$ is surjective and $\Frac R = K(Y)$, it follows that $e( \Gr R ) = 1$.
We claim (see Notation \ref{iwyr8762367ds2wxfwd2w2f3rfvtm0m00}):
\begin{equation}  \label {765wjwejnbfcx67678o}
\Gr(R) \isom \big( B/x B \big)^{\langle e_\pgoth d \rangle} .
\end{equation}
Let us prove \eqref{765wjwejnbfcx67678o}.
Given $f \in R \setminus \{0\}$, let $S(f) = \setspec{ n \in \Nat }{ f x^n \in B }$ and note that $S(f)$ is a nonempty subset of $\Nat$.
Clearly, if $n \in S(f)$ then $n+1 \in S(f)$.
We claim that 
\begin{equation}  \label {A87tr7ghcjomgIbbvncvnsuieruo3834}
\text{for each $f \in R \setminus \{0\}$,} \quad S(f) = \setspec{ n \in \Nat }{ n \ge e_\pgoth \deg_R(f) } .
\end{equation}
Indeed, let $f \in R \setminus \{0\}$. Given any $n \in \Nat$, we have
$n \in S(f) \iff f x^n \in B \iff \div_X(f x^n) \ge 0$.
Consider the $\Rat$-linear map $D \mapsto D^*$ from $\Div_\Rat(Y)$ to $\Div_\Rat(X)$ defined in \ref{Pkcnbvc9w3eidjojf0q9w};
then $(C_\pgoth^Y )^* = {e_\pgoth} C_\pgoth^X$, and \eqref{PWSoidbfci27wyd} gives $( \div_Y(f))^* = \div_X(f)$.
Since 
$$
\textstyle
\div_X(f x^n) 
= \div_X(f) + n \div_X(x)
= \div_X(f) + n C_\pgoth^X
= ( \div_Y(f) + \frac n{e_\pgoth} C_\pgoth^Y )^*,
$$
and since $D^* \ge0 \Leftrightarrow D \ge 0$ for any $D \in \Div_\Rat(Y)$,
we have
\begin{equation*}  
\textstyle
n \in S(f) \iff \div_Y(f) +  \frac n{e_\pgoth} C_\pgoth^Y  \ge 0
\iff v_\pgoth^Y(f) \ge -\frac n{e_\pgoth} \iff \deg_R(f) \le \frac n{e_\pgoth},
\end{equation*}
where the middle equivalence follows from $\div_U(f) \ge 0$ (see \eqref{E1nfo83aTcbpqc0gfrnxjg}).
This proves \eqref{A87tr7ghcjomgIbbvncvnsuieruo3834}.
Now \eqref{A87tr7ghcjomgIbbvncvnsuieruo3834} implies that if $n \in \Nat$ and $f \in R \setminus \{0\}$ then:
\begin{align*}
f \in R_{\le n} & \Leftrightarrow \deg_R(f) \le n \Leftrightarrow e_\pgoth n \in S(f) \Leftrightarrow f x^{e_\pgoth n} \in B ; \\
f \in R_{< n} & \Leftrightarrow \deg_R(f) < n \Leftrightarrow e_\pgoth n > \min S(f) \Leftrightarrow  e_\pgoth n - 1 \in S(f)
\Leftrightarrow  f x^{e_\pgoth n - 1} \in B \\
& \Leftrightarrow  f x^{e_\pgoth n } \in xB.
\end{align*}
So the following holds for each $n \in \Nat$:
\begin{equation*} 
\text{each $f \in R_{\le n}$ satisfies\ \ $\big( f x^{e_\pgoth n} \in B$\ \ and\ \  $f x^{e_\pgoth n} \in xB \Leftrightarrow f \in R_{< n}\big)$.}
\end{equation*}
So, for each $n \in \Nat$, the $\bk$-linear map $\theta_n :R_{\le n} \to B_{d e_\pgoth n}$, $f \mapsto f x^{e_\pgoth n}$, 
is well-defined and induces an injective $\bk$-linear map $\bar\theta_n :  R_{\le n} / R_{<n} \to (B/xB)_{d e_\pgoth n}$. 
Note that $\bar\theta_n$ is also surjective.
Indeed, consider $\xi \in (B/xB)_{d e_\pgoth n}$. 
Then $\xi = h + xB$ for some $h \in B_{d e_\pgoth n}$. 
Obviously, if $h=0$ then $\xi \in \image( \bar\theta_n )$.
Assume that $h \neq0$ and let $f = h/x^{e_\pgoth n} \in R \setminus \{0\}$;
then $e_\pgoth n \in S(f)$ by definition of $S(f)$, so $e_\pgoth n \ge e_\pgoth \deg_R(f)$ by \eqref{A87tr7ghcjomgIbbvncvnsuieruo3834}, so $\deg_R(f) \le n$,
i.e., $f \in R_{\le n}$; we have $\theta_n(f) = h$, so $\bar\theta_n( f + R_{<n} ) = \xi$.
This shows that, for each $n \in \Nat$, the $\bk$-linear map  $\bar\theta_n :  R_{\le n} / R_{<n} \to (B/xB)_{d e_\pgoth n}$ is bijective.
The family $\big(\bar\theta_n \big)_{ n \in \Nat }$ defines an injective $\bk$-linear map $\bar\theta : \Gr(R) \to B/xB$
which is easily seen to preserve multiplication. Moreover, the image of $\bar\theta$ is the subring $\big( B/xB )^{(d e_\pgoth )}$ of $B/xB$.
If we define $S_n =  (B/xB)_{d e_\pgoth n}$ for all $n \in \Nat$, then $\big( B/xB )^{ \langle d e_\pgoth \rangle } = \bigoplus_{n \in \Nat} S_n$,
so $\bar\theta : \Gr(R) \to \big( B/xB )^{ \langle d e_\pgoth \rangle }$
is a degree-preserving isomorphism of graded $\bk$-algebras. 
This proves \eqref{765wjwejnbfcx67678o}.

It follows from \eqref{765wjwejnbfcx67678o} that $\Gr(R)$ is $\bk$-affine
(because $( B/x B \big)^{\langle d e_\pgoth \rangle}$ is $\bk$-affine by Cor.\ \ref{8Qp98781hy2uexihvcx2qw2wUxq3ew0pokjmnbvaw2345}),
so part (b) of Theorem 1.7 of \cite{Dai:TameWild}
implies that $\deg_R$ is tame over $\bk$ (in the terminology of \cite{Dai:TameWild}, this means that 
for every $\bk$-derivation $D : R \to R$, the set $\setspec{ \deg_R( Dr ) - \deg_R(r) }{ r \in R \setminus \{0\} }$ has a greatest element).
As is well known, this implies that each element $D$ of $\lnd(R) \setminus \{0\}$ gives rise to a well-defined $\Gr(D) : \Gr(R) \to \Gr(R)$
which belongs to $\hlnd( \Gr R ) \setminus \{0\}$.
So if $\lnd(R) \neq \{0\}$ then $\hlnd( \Gr R ) \neq \{0\}$,
which (by \eqref{765wjwejnbfcx67678o}) implies $\hlnd\big( \big( B/x B \big)^{\langle d e_\pgoth \rangle} \big) \neq \{0\}$, 
which implies that $\big( B/x B \big)^{\langle d e_\pgoth \rangle}$ is non-rigid. 
Since the graded rings $\big( B/x B \big)^{\langle d e_\pgoth \rangle}$ and $\big( B/x B \big)^{( d e_\pgoth )}$ have the same underlying ring,
and since (by definition) rigidity only depends on the underlying ring, we have shown that 
\begin{equation}  \label {u765r0x9id21wsxyud56rb75v2ubvcexzLu504odkUT3}
\text{if $R = B_{(x)}$ is non-rigid then $\big( B/x B \big)^{( d e_\pgoth )}$ is non-rigid}.
\end{equation}
By Cor.\ \ref{mMpPiq13096gxvbw5resip0tfs}(a) and \eqref{cjb738645q9mbr2hnbmadp28570},
$e\big( (B/xB)^{(d)} \big) = \lcm(d, e(B/xB) ) = \frac{ d \, e(B/xB) }{ e(B) } = d e_\pgoth$,
which implies that $(B/xB)^{(d)} = (B/xB)^{(d e_{\pgoth})}$.
This together with \eqref{u765r0x9id21wsxyud56rb75v2ubvcexzLu504odkUT3} completes the proof.
\end{proof}

Note that this also proves Thm \ref{7654cvhgfsd2d34fds23f42dfdghjk8l-geometric}, 
since it is equivalent to Thm \ref{7654cvhgfsd2d34fds23f42dfdghjk8l}.

The next fact is an application of Theorem \ref{7654cvhgfsd2d34fds23f42dfdghjk8l}.
Refer to  \ref{iu7654wd3fg56hxr39ik} and \ref{8r43qasd321qsxcvgttyhnemkrifv90olkmn} for the notation.

\begin{proposition} \label {u654edcvbhyu89okmnbgr32qazdcftgyu934}
Let $n\ge4$ and $\ba = (a_1, \dots, a_n) \in ( \Nat \setminus \{0\} )^n$.
Let $\bk$ be a field of characteristic $0$ and consider the Pham-Brieskorn ring 
$B_{ \bk; \ba } = B_{ \bk ; a_1, \dots, a_n } = \bk[x_1,\dots,x_n]$.
For each $i \in \{1,\dots,n\}$, let $\ba(i) = (a_1, \dots , \widehat{a_i} , \dots , a_n) \in ( \Nat \setminus \{0\} )^{n-1}$
and consider the corresponding ring $B_{ \bk ; \ba(i) } = B_{ \bk ; a_1, \dots , \widehat{a_i} , \dots , a_n }$.
Then the implication 
\begin{equation}  \label {765ghcnvbnx3928yrtg}
\text{$B_{ \bk ; \ba(i) }$ is rigid} \implies \text{$(B_{ \bk ; \ba })_{(x_i)}$ is rigid}
\end{equation}
is valid for each $i \in \{1,\dots,n\}$ that satisfies $S(\ba) \subseteq \{i\}$.
\end{proposition}

\begin{remark}
By definition, $\cotype(\ba) = |S(\ba)|$.
Clearly, if $\cotype(\ba)=0$ then all $i \in \{1,\dots,n\}$ satisfy  $S(\ba) \subseteq \{i\}$;
if $\cotype(\ba)=1$ then exactly one $i \in \{1,\dots,n\}$ satisfies  $S(\ba) \subseteq \{i\}$;
and if $\cotype(\ba)>1$ then no element $i$ of $\{1,\dots,n\}$ satisfies  $S(\ba) \subseteq \{i\}$.
\end{remark}

\begin{proof}[Proof of Prop.\ \ref{u654edcvbhyu89okmnbgr32qazdcftgyu934}]
We may assume that $i=n$, so it suffices to prove that 
$$
\text{\it if $S(\ba) \subseteq \{n\}$ and $(B_{ \bk ; \ba })_{(x_n)}$ is non-rigid, then $B_{ \bk ; \ba(n) }$ is non-rigid.}
$$
Suppose that $S(\ba) \subseteq \{n\}$ and that $(B_{ \bk ; \ba })_{(x_n)}$ is non-rigid.
Thm \ref{7654cvhgfsd2d34fds23f42dfdghjk8l} implies that $( B_{ \bk ; \ba } / x_n B_{ \bk ; \ba } )^{(d_n)}$ is non-rigid.
Note that $B_{ \bk ; \ba } / x_n B_{ \bk ; \ba } \isomdot B_{ \bk ; \ba(n) }$ (see Notation \ref{98if9jvs9d772jfpd0vhf9w93uo0541hf9});
we claim:
\begin{equation}  \label {9879686754evcxyui9q}
\big( B_{ \bk ; \ba } / x_n B_{ \bk ; \ba } \big)^{(d_n)} \isomdot \big( B_{ \bk ; \ba(n) } \big)^{(d_n)} .
\end{equation}
This has to be carefully argued, because in general the condition $A \isomdot B$ does not imply that $A^{(d)} \isomdot B^{(d)}$ for all $d$.\footnote{For instance,
let $A = \bk[X_1,X_2] = \kk2$ and $B = \bk[Y_1,Y_2] = \kk2$ where $\deg(X_i)=1$ and $\deg(Y_i)=2$.
Then $A \isomdot B$ is true but $A^{(2)} \isomdot B^{(2)}$ is false, as $B^{(2)} = B$ is a polynomial ring but $A^{(2)}$ is not.}
If we use the notations $\bar B =B_{ \bk ; \ba } / x_n B_{ \bk ; \ba }$, $C = B_{ \bk ; \ba(n) }$ and $\epsilon = \gcd(d_1,\dots,d_{n-1})$,
then $e(\bar B) = \epsilon$
and there is an isomorphism of $\bk$-algebras $\phi : \bar B \to C$ satisfying:
$$
\phi( \bar B_{i} ) = \begin{cases}
C_{i/\epsilon} & \text{if $i \in \epsilon\Nat$,} \\
0 & \text{if $i \in \Nat \setminus \epsilon\Nat$.}
\end{cases}
$$
Let $\ell = \lcm(\epsilon,d_n)$; then $e( \bar B^{(d_n)} ) = \ell$ by Cor.\ \ref{mMpPiq13096gxvbw5resip0tfs}(a), so $\bar B^{(d_n)} = \bar B^{(\ell)}$ and hence
$$
\textstyle
\phi\big( \bar B^{(d_n)} \big)
= \phi\big( \bar B^{( \ell )} \big)
= \phi\big( \bigoplus_{ i \in \ell\Nat } \bar B_{i} \big)
= \bigoplus_{ i \in \ell\Nat } C_{i/\epsilon}
= \bigoplus_{ j \in (\ell/\epsilon)\Nat } C_{j} = C^{( \ell/\epsilon )} ,
$$
showing that $( \bar B )^{(d_n)} \isomdot C^{(\ell/\epsilon)}$.
Since $\gcd(\epsilon,d_n) = \gcd(d_1, \dots, d_n) = 1$, we have $\ell/\epsilon = d_n$,
so $( \bar B )^{(d_n)} \isomdot C^{(d_n)}$.  This proves \eqref{9879686754evcxyui9q}.
Next, we prove:
\begin{equation}  \label {i76f5de4ws0A0olciunmnbvce231zx9}
\gcd( d_n , \bar e(B_{ \bk ; \ba(n) }) ) = 1 .
\end{equation}
We have $\bar e(B_{ \bk ; \ba(n) }) = \prod_{j=1}^{n-1} \frac{\lcm(a_1,\dots,a_{n-1})}{\lcm(a_1,\dots, \widehat{a_j} , \dots, a_{n-1})}$
by Lemma \ref{o8y7t626rf53deg23489jru3sgbref6}.
Suppose that $p$ is a prime number that divides both $d_n$ and $\bar e(B_{ \bk ; \ba(n) })$.
Since $p \mid \prod_{j=1}^{n-1} \frac{\lcm(a_1,\dots,a_{n-1})}{\lcm(a_1,\dots, \widehat{a_j} , \dots, a_{n-1})}$,
there exists $j \in \{1,\dots,n-1\}$ such that $v_p(a_j) >  \max( v_p(a_1) , \dots, \widehat{ v_p(a_j) } , \dots, v_p(a_{n-1}) )$.
Since $p$ divides $d_n = L/a_n$, we have $v_p(a_n) < \max( v_p(a_1) , \dots, v_p(a_{n-1}) ) = v_p(a_j)$, so
$v_p(a_j) >  \max( v_p(a_1) , \dots, \widehat{ v_p(a_j) } , \dots, v_p(a_{n}) ) = v_p( L_j )$.
It follows that $v_p(L) \ge v_p(a_j) > v_p( L_j )$, so $j \in S(\ba)$, which contradicts $S(\ba) \subseteq \{n\}$.
So $p$ does not exist, and this proves \eqref{i76f5de4ws0A0olciunmnbvce231zx9}.

Since $\big( B_{ \bk ; \ba } / x_n B_{ \bk ; \ba } \big)^{(d_n)}$ is non-rigid,
\eqref{9879686754evcxyui9q} implies that $(B_{ \bk ; \ba(n) })^{(d_n)}$ is non-rigid.
Thus, $d_n \in \NR( B_{ \bk ; \ba(n) } ) = \bigcup_{ d \in M( B_{ \bk ; \ba(n) } ) } \Iscr_d$,
where the equality is Cor.\ \ref{p98qy386e523brxhbfc6ghvc63x2345618q30}(a).
It follows that there exists $d \in M( B_{ \bk ; \ba(n) } )$ such that $d \mid d_n$.
Cor.\ \ref{p98qy386e523brxhbfc6ghvc63x2345618q30}(b) gives $d \mid \bar e(B_{ \bk ; \ba(n) })$,
so $d=1$ by \eqref{i76f5de4ws0A0olciunmnbvce231zx9}.
This means that $1 \in M( B_{ \bk ; \ba(n) } )$, so $B_{ \bk ; \ba(n) }$ is non-rigid, as desired.
\end{proof}

\begin{example}  \label {lBoibhgxfd7ii03A8vevcCxEvv82}
Let $\bk$ be a field of characteristic $0$ and let $n$ be an integer such that:\footnote{Observe that if $n \in \{4,5\}$
then $n$ satisfies \eqref{conditionOnn}, because $\Pscr(n-1)$ is true by \ref{ConjPn}.}
\begin{equation} \label {conditionOnn}
\text{$n\ge4$ and $\Pscr(n-1)$ is true.}
\end{equation}
Let $\ba = (a_1,\dots,a_n) \in \Gamma_n$ be such that $\cotype(\ba)=0$ and consider $B = B_{\bk;\ba} = B_{ \bk ; a_1,\dots,a_n }$.
Then:
\begin{enumerata}

\item $B_{(x_i)}$ is rigid for every $i \in \{1,\dots,n\}$;
\item for each basic cylinder $U$ of $\Proj(B)$ and each  $i \in \{1,\dots,n\}$, we have $\bbV_+(x_i) \cap U \neq \emptyset$.

\end{enumerata}
Indeed, for each $i \in \{1,\dots,n\}$ we have $\ba(i) \in \Gamma_{n-1}$, so $B_{ \bk  ; \ba(i) }$ is rigid because $\Pscr(n-1)$ is true;
since $S(\ba) = \emptyset \subseteq \{i\}$, Prop.\ \ref{u654edcvbhyu89okmnbgr32qazdcftgyu934} implies that $B_{(x_i)}$ is rigid.  This proves (a).
By Rem.\ \ref{iytd43J45628c93j}, (b) follows from (a).
\end{example}

\section{Fibers of a homogeneous polynomial}
\label{sectionFibersofahomogeneouspolynomial}

The title of this section refers to Prop.\ \ref{u65vmklio909543qsdfh4vue}, which is the main result of the section.

\begin{lemma}[\cite{SchinzelBook2000}, Chap.~3, \S~3, Cor.~1]
\label {Matsudp23rwije}
Let $\bk$ be an algebraically closed field, $n\ge1$ and $F \in A = \kk n$. Then the set
$ \setspec{ \lambda \in \bk }{ \text{$F-\lambda$ is not irreducible in $A$} } $
is either finite or equal to $\bk$, and it is equal to $\bk$ if and only if
$F = P(G)$ for some $G \in A$ and some univariate polynomial $P(T) \in \bk[T]$
such that $\deg_T P(T) > 1$.
\end{lemma}

\begin{lemma}[Lemma 6.1 of \cite{Daigle:StructureRings}]  \label {cccy5hc7cnc8c4cg}
Let $A$ be an algebra over a field $\bk$, 
let $K/\bk$ be an algebraic Galois extension and write $G=\Gal(K/\bk)$ and $A_K = K \otimes_\bk A$.
For each $\theta \in G$, let $\tilde\theta : A_K \to A_K$ be the  $A$-automorphism of $A_K$ given by
$\tilde\theta (\lambda \otimes a ) =\theta(\lambda) \otimes a $ ($\lambda\in K$, $a\in A$).
If $b$ is an element of $A_K$ satisfying
\begin{equation} \label {cijvno2i3JGFdWKJ732e83}
\forall_{ \theta \in G }\ \exists_{\lambda \in K^*}\  \tilde\theta(b) = \lambda b 
\end{equation}
then there exists $\lambda \in K^*$ such that $\lambda b \in A$.
\end{lemma}

\begin{lemma}  \label {9823bSnby5is5wckdsohf}
Let $\bk$ be a field of characteristic $0$, $\ck$ an algebraic closure of $\bk$, $n\ge1$ and $R = \bk[X_1,\dots,X_n] = \bk^{[n]}$.
Endow $R$ with an $\Nat$-grading such that each $X_i$ is homogeneous and let $f$ be a homogeneous prime element of $R$ of positive degree.
Then $f-\lambda$ is irreducible over $\ck$ for every $\lambda \in \ck^*$.
\end{lemma}

\begin{proof}
Let $\bar R = \ck[X_1,\dots,X_n] = \ck^{[n]}$.
Let $m=\deg(f)>0$ and $d_i = \deg(X_i)$ ($1 \le i \le n$).
For each $\lambda \in \ck^*$, consider the $\ck$-automorphism $\mu_\lambda$ of $\bar R$ defined by
$\mu_\lambda(X_i) = \lambda^{d_i} X_i$ ($1 \le i \le n$).
If $\lambda, \lambda_0 \in \ck^*$ then $\mu_\lambda( f - \lambda^m \lambda_0 ) = \lambda^m(f-\lambda_0)$,
so the $\ck$-algebras $\bar R / (f - \lambda_0)$ and $\bar R / (f - \lambda^m \lambda_0)$ are isomorphic.
Since $m>0$, it follows that
\begin{equation}  \label {765432erghxdnjkrg67h8j90monk}
\text{for all $\lambda_1,\lambda_2 \in \ck^*$, \quad $\bar R / (f - \lambda_1)$ is $\ck$-isomorphic to $\bar R / (f - \lambda_2)$.}
\end{equation}

Arguing by contradiction, suppose that $f-\lambda$ is reducible in $\bar R$, for some $\lambda \in \ck^*$.
Then, by \eqref{765432erghxdnjkrg67h8j90monk}, $f-\lambda$ is reducible in $\bar R$ for all $\lambda \in \ck^*$.
So Lemma \ref{Matsudp23rwije} implies that $f = P(g)$ for some $g \in \bar R$ and some univariate polynomial $P(T) \in \ck[T]$ such
that $\deg_T P(T) > 1$. Using that $f$ is homogeneous of positive degree, we see that $g$ may be chosen so as to have $P(T) = T^r$, where $r\ge2$.
Thus, $f = g^r$.
For each $\theta \in \Gal(\ck/\bk)$, define $\tilde\theta : \bar R \to \bar R$ as in Lemma \ref{cccy5hc7cnc8c4cg}.
Then $\tilde\theta(g)^r = f$, so $( \tilde\theta(g) / g )^r = 1$.
This shows that for each $\theta \in \Gal(\ck/\bk)$ there exists $\lambda_\theta \in \ck^*$ such that $\tilde\theta(g)=\lambda_\theta g$.
By Lemma \ref{cccy5hc7cnc8c4cg}, it follows that there exists $\lambda \in \ck^*$ such that $\lambda g \in R$.
So $f = (\lambda^{-r}) (\lambda g)^r$ is not irreducible in $R$, a contradiction.
\end{proof}

\begin{lemma}  \label {87t7weg5r237r8jwngbvcb37wj83y}
Let $\bk$ be a field of characteristic $0$, $\ck$ an algebraic closure of $\bk$, $n\ge2$,
and $f$ an irreducible element of $R = \bk[X_1, \dots, X_{n}] = \kk n$.
Then the zero-set of $( f, \frac{\partial f}{\partial X_1}, \dots, \frac{\partial f}{\partial X_n} )$ in $\aff^n_{ \ck }$ has dimension
at most $n-2$.
\end{lemma}

\begin{proof}
If the conclusion is false then there exists a prime element $g$ of $\bar R = \ck[X_1, \dots, X_{n}]$ that is a common divisor
of $f$, $\frac{\partial f}{\partial X_1}$, \dots, $\frac{\partial f}{\partial X_n}$ in  $\bar R$.
Choose $i$ such that $\frac{\partial g}{\partial X_i} \neq 0$ and note that $\frac{\partial f}{\partial X_i} \neq 0$.
Let $K = \bk(X_1, \dots, \widehat{X_i}, \dots, X_n)$ and $\bar K = \ck(X_1, \dots, \widehat{X_i}, \dots, X_n)$.
Note that $f$ (resp.\ $g$) is an irreducible element of $K[X_i] = K^{[1]}$ (resp.\  of $\bar K[X_i] = \bar K^{[1]}$).
It follows that $\gcd( f, \frac{\partial f}{\partial X_i}) = 1$ in $K[X_i]$, so $uf+v \frac{\partial f}{\partial X_i} = 1$ for some $u,v \in K[X_i]$,
so $\gcd( f, \frac{\partial f}{\partial X_i}) = 1$ in $\bar K[X_i]$, contradicting the fact that $g$ is 
an irreducible element of $\bar K[X_i]$ that divides both $f$ and $\frac{\partial f}{\partial X_i}$.
\end{proof}

We say that a $\bk$-algebra $A$ {\it has trivial units\/} if $A^* = \bk^*$.

\begin{proposition}  \label {u65vmklio909543qsdfh4vue}
Let $\bk$ be a field of characteristic $0$ and $R = \bk[X_1, \dots, X_n] = \kk n$, where $n\ge2$.
Endow $R$ with an $\Nat$-grading such that each $X_i$ is homogeneous and $X_1$ has positive degree.
Let $f$ be a homogeneous prime element of $R = \bigoplus_{i \in \Nat} R_i$ such that $f \notin R_0[X_1]$.
\begin{enumerata}

\item If $c \in \bk$ then $R/(f - c)R$ is a domain with trivial units.

\item If $R/fR$ is rigid then so is $R/(f - c)R$ for every $c \in \bk$.

\end{enumerata}
\end{proposition}

\begin{proof}
Since $S := R/fR = \bigoplus_{i \in \Nat}S_i$ is an $\Nat$-graded domain, $S^* = S_0^*$.
The fact that $\deg(f)>0$ implies that $S_0 = R_0$, and we have $R_0 = \kk r$ for some $r\ge0$, so $S^* = S_0^* = R_0^* = \bk^*$,
i.e., $R/fR$ has trivial units. So the case $c=0$ of (a) is true.
Obviously, the case $c=0$ of (b) is also true. 

Until the end of the proof, we assume that $c \in \bk^*$.
Lemma \ref{9823bSnby5is5wckdsohf} implies that $R/(f-c)$ is a domain.
To complete the proof of the Proposition, it remains to show that $R/(f-c)$ has trivial units, and that if $R/fR$ is rigid then so is $R/(f-c)$.

Let $m=\deg(f)>0$ and $d_i = \deg(X_i) \in \Nat$ ($1 \le i \le n$).
We claim that $f - c X_{n+1}^m$ is irreducible in $\ck[X_1, \dots, X_{n+1}] = \ck^{[n+1]}$, where  $\ck$ is an algebraic closure of $\bk$.
Indeed, suppose that $G,H \in \ck[X_1, \dots, X_{n+1}]$ are such that $f - c X_{n+1}^m = GH$.
For each $\lambda \in \ck^*$, define $G(\lambda) = G(X_1,\dots,X_n,\lambda)$ and $H(\lambda) = H(X_1,\dots,X_n,\lambda)$.
Then for each $\lambda \in \ck^*$ we have  $f - c \lambda^m = G(\lambda)H(\lambda)$.
Since $f - c \lambda^m$ is irreducible over $\ck$ by Lemma \ref{9823bSnby5is5wckdsohf},
it follows that  $G(\lambda)\in \ck^*$ or $H(\lambda)\in\ck^*$.
So there exists $W \in \{G,H\}$ such that $W(\lambda)\in \ck^*$ is true for infinitely many $\lambda \in \ck$.
This implies that $W \in \ck[ X_{n+1}]$.
There exists $V \in \{G,H\}$ such that  $WV = f-cX_{n+1}^m$;
pick $i \in \{1,\dots,n\}$ such that $\frac{\partial f}{\partial X_i} \neq 0$, then 
$W\frac{\partial V}{\partial X_i} = \frac{\partial f}{\partial X_i} \in \bk[X_1,\dots,X_n] \setminus \{0\}$ shows that $W \in \ck^*$,
which, in turn, shows that $f - c X_{n+1}^m$ is irreducible in $\ck[X_1, \dots, X_{n+1}]$.
So both $B = \bk[X_1, \dots, X_{n+1}] / (f - c X_{n+1}^m)$ and $\bar B = \ck \otimes_\bk B$ are domains.

We claim that $B$ and $\bar B$ are normal. This is clear if $m=1$, so assume that $m\ge2$.
Then the singular locus of the affine variety $X = V(f - c X_{n+1}^m) \subset \aff_{\ck}^{n+1}$ 
is $V(X_{n+1}, f, \frac{\partial f}{\partial X_1}, \dots, \frac{\partial f}{\partial X_n} ) \subseteq \aff_{\ck}^{n+1}$.
By Lemma \ref{87t7weg5r237r8jwngbvcb37wj83y}, we see that the codimension of $\Sing(X)$ in $X$ is greater than $1$;
since $X$ is a hypersurface of $\aff^{n+1}_{\ck}$, it follows that $X$ is normal. Since $X \isom \Spec \bar B$, $\bar B$ is normal.
It is well known that $\bar B \cap \Frac(B) = B$, so $B$ is normal.

Define an $\Nat$-grading on $\bk[X_1, \dots, X_{n+1}]$ by declaring that $X_{n+1}$ is homogeneous of degree $1$
and (for $1 \le i \le n$) $X_i$ is homogeneous of degree $d_i$.
Then $f - c X_{n+1}^m$ is homogeneous, so $B$ is $\Nat$-graded.
Write $B = \bk[x_1, \dots, x_{n+1}] = \bigoplus_{i \in \Nat} B_i$ where $x_i$ is the canonical image of $X_i$ in $B$.
Since $x_1$ and $x_{n+1}$ have positive degrees and are algebraically independent over $B_0$ (because $f - c X_{n+1}^m \notin R_0[X_1,X_{n+1}]$),
$B$ has transcendence degree at least $2$ over $B_0$; so the prime ideal $B_+ = \bigoplus_{i>0} B_i$ has height at least $2$.
We showed:
$$
\text{$B$ is an $\Nat$-graded normal affine $\bk$-domain such that $\haut(B_+)\ge2$.}
$$
Since $B/x_{n+1}B \isom R/fR$ is a domain, $x_{n+1}$ is a homogeneous prime element of $B$ of degree $1$.
We claim that 
\begin{equation}  \label {K5f43d32rcnA391Hi02r5}
B_{(x_{n+1})} \isom R/(f - c) .
\end{equation}
To see this, let $\phi : R \to B_{(x_{n+1})}$ be the surjective $\bk$-homomorphism defined by $\phi( X_i ) = x_i / x_{n+1}^{d_i}$ ($1 \le i \le n$).
We have $\phi(f) = f\big( {x_1}/{x_{n+1}^{d_1}} , \dots, {x_n}/{x_{n+1}^{d_n}} \big) = f / x_{n+1}^m = c$,
so $(f-c) \subseteq \ker\phi$.
Lemma \ref{9823bSnby5is5wckdsohf} implies that $f-c$ is a prime element of $R$, so $(f-c)$ is a height $1$ prime ideal of $R$.
Since $\dim R = n$ and $\dim B_{(x_{n+1})} = \dim \Proj B = n-1$, $\ker\phi$ is a height $1$ prime ideal of $R$, so $\ker\phi = (f-c)$,
which proves \eqref{K5f43d32rcnA391Hi02r5}.

Since $B$ is an $\Nat$-graded domain and $x_{n+1}$ is a homogeneous prime element of $B$ of positive degree, we have $\big( B_{(x_{n+1})} \big)^* = B_0^*$.
Since $\deg(f - cX_{n+1}^m)>0$, $B_0 = \bk[X_1, \dots, X_{n+1}]_0 = \kk r$ for some $r\ge0$, so $\big( B_{(x_{n+1})} \big)^* = B_0^* = \bk^*$.
It then follows from \eqref{K5f43d32rcnA391Hi02r5} that $R/(f-c)$ has trivial units.

Suppose that $R/fR$ is rigid and 
let us apply Thm \ref{7654cvhgfsd2d34fds23f42dfdghjk8l} to $B$ and $x = x_{n+1}$ (note that $d = \deg(x) = 1$).
By that result and the fact that $(B/xB)^{(d)} = B/xB = B/x_{n+1}B \isom R/fR$ is rigid, we obtain that $B_{(x_{n+1})}$ is rigid,
so $R/(f-c)$ is rigid by \eqref{K5f43d32rcnA391Hi02r5}.
\end{proof}


Before giving an application of Prop.\ \ref{u65vmklio909543qsdfh4vue}, we recall some facts about the Fermat cubics:

\begin{bigremark} \label {24799F47-1201-4FF5-A50D-93EB60154D1C} 
For each $n\ge1$, consider the $n$-dimensional Fermat cubic 
$$
\Feul_n = \bbV_{\!+}(X_0^3 + \cdots + X_{n+1}^3) \subset \proj^{n+1} .
$$
\begin{enumerata}

\item If $n$ is even then $\Feul_n$ is rational, by \cite[Thm 1.1]{Massarenti_Fermat_2026}.

\item It is known that $\Feul_3$ does not contain a cylinder.
We thank Michael Chitayat for explaining to us that this 
follows from the fact that $\Feul_3$ is rationally connected \cite{Campana_ConnexRatFano_1992}
and not rational \cite{ClemensGriffithIntermJacobian_1972}.

\end{enumerata}
\end{bigremark}

\begin{example}
Let $n \ge 2$ and define
$$
B_n = \Comp[X_0, \dots, X_n] / (X_0^3 + \cdots + X_n^3 )
\quad \text{and} \quad
A_n = \Comp[X_0, \dots, X_n] / (X_0^3 + \cdots + X_n^3 - 1) .
$$
\begin{enumerata}

\item $A_n$ and $B_n$ are domains with trivial units, and of Krull dimension $n$.

\item $\Spec A_n$ is a smooth affine variety, and if $n$ is even then $\Spec A_n$ is rational.

\item $\Spec B_n$ is a normal affine variety, and if $n$ is odd then $\Spec B_n$ is rational.

\rien{
\item If $n$ satisfies one of the following conditions
\begin{enumerati}
\item $\Pscr(n+1)$ holds (see \ref{ConjPn}),
\item $\Feul_{n-1}$ does not contain a cylinder,
\end{enumerati}
then $B_n^{(d)}$ is rigid for all $d\ge1$ and $A_n$ is rigid.
}

\item If $n \in \{2,3,4\}$ then  $B_n^{(d)}$ is rigid for all $d\ge1$ and $A_n$ is rigid.

\end{enumerata}
To prove these claims, consider
$f_n = X_0^3 + \cdots + X_n^3 \in R_n = \Comp[X_0, \dots, X_n] = \Comp^{[n+1]}$ where $X_0,\dots,X_{n}$ are homogeneous of degree $1$.
We have $B_n = R_n/(f_n)$ and $A_n = R_n/(f_n-1)$, so Proposition \ref{u65vmklio909543qsdfh4vue} implies (in particular)
that $A_n$ and $B_n$ are domains with trivial units. So (a) is true.
It is clear that $\Spec A_n$ is smooth and that $\Spec B_n$ is normal.
Note that the open subset ``$X_{n+1} \neq 0$'' of the Fermat cubic $\Feul_n = \bbV_{\!+}(X_0^3 + \cdots + X_n^3 + X_{n+1}^3) \subset \proj^{n+1}$
is isomorphic to $\Spec A_n$;
if $n$ is even then $\Feul_n$ is rational by Rem.\ \ref{24799F47-1201-4FF5-A50D-93EB60154D1C}, so $\Spec A_n$ is rational. This proves (b).
Since $\Proj B_n = \Feul_{n-1}$, we have the birational equivalences
$\Spec B_n \sim \proj^1 \times \Proj B_n \sim \proj^1 \times \Feul_{n-1}$;
if $n$ is odd then $\Feul_{n-1}$ is rational by Rem.\ \ref{24799F47-1201-4FF5-A50D-93EB60154D1C}, so $\Spec B_n$ is rational. This proves (c).

It is easy to see that $B_n$ is saturated in codimension $1$ (for instance this follows from Lemma \ref{o8y7t626rf53deg23489jru3sgbref6}).
So Cor.\ \ref{1D0x2mv4yhhtfukCs1unc3t46bisgf9f2} implies that, for each $n\ge2$,
the rings $B_n^{(d)}$ ($d \ge 1$) all have the same rigidity status.
Combining that with part (b) of Prop.\ \ref{u65vmklio909543qsdfh4vue} gives:
$$
\text{if $n$ is such that $B_n$ is rigid, then  $B_n^{(d)}$ is rigid for all $d\ge1$ and $A_n$ is rigid.}
$$
So, in order to prove (d), it suffices to show that $B_2$, $B_3$ and $B_4$ are rigid.
Since $\Pscr(3)$ and $\Pscr(4)$ are true (see \ref{ConjPn}), the rings $B_2 \isom B_{\Comp; 3,3,3}$ and $B_3 \isom B_{\Comp; 3,3,3,3}$ are rigid.
Since $\Proj B_4 = \Feul_3$ does not contain a cylinder by \ref{24799F47-1201-4FF5-A50D-93EB60154D1C},
it follows that $B_4$ is rigid by Thm \ref{edh83yf6r79hvujhxu6wrefji9e}.
This proves (d).
(More generally, if $\Pscr(n+1)$ is true or $\Feul_{n-1}$ does not contain a cylinder then $B_n$ is rigid.)


Note that the rigidity of $A_3$ can be proved without invoking Proposition~\ref{u65vmklio909543qsdfh4vue}:
since $\Spec A_3$ is an open subset of $\Feul_3$,
\ref{24799F47-1201-4FF5-A50D-93EB60154D1C} implies that $\Spec A_3$ does not contain a cylinder, so $A_3$ is rigid by \ref{jbhgfdxfewae33w6q920we}.
In contrast, the proofs of rigidity for $A_2$ and $A_4$ appear to require the Proposition.
Moreover, the rigidity of $A_2$ and $A_4$ is noteworthy, since $\Spec A_2$ and $\Spec A_4$ are smooth, rational, and have trivial units.
\end{example}

\section{A result about $\bG(B)$}
\label {SEC:AresultaboutbGB}

The aim of this section is to prove Thm \ref{p0987uhmnbvccjiuseyt428}.

Given a subgroup $H$ of an abelian group $G$, we define the {\it torsion-closure of $H$ in $G$} to be
$\torc(H,G) = \setspec{ x \in G }{ \text{there exists $m \in \Nat\setminus\{0\}$ such that $mx \in H$} }$.

\begin{lemma}  \label {1hfcvbnn92m3n456789df7nm534jh5f6d78}
Let $H$ be a subgroup of a finitely generated abelian group $G$ such that $G/H$ is not torsion.
Let $\Heul(H,G)$ be the set of all subgroups $H'$ of $G$ such that $H \subseteq H' \subset G$ and $G/H' \isom \Integ$.
Then $\Heul(H,G) \neq \emptyset$ and the intersection of all elements of $\Heul(H,G)$ is $\torc(H,G)$.
\end{lemma}

\begin{proof}
Viewing $G$ as a $\Integ$-module, let $S = \Integ \setminus \{0\}$ and consider the $\Rat$-vector space $G_\Rat = S^{-1}G$
and the canonical homomorphism $\lambda : G \to G_\Rat$, $x \mapsto x/1$.
For each subgroup $K$ of $G$, let $K_\Rat$ be the subspace $S^{-1}K$ of $G_\Rat$ and note that $\lambda^{-1}( K_\Rat ) = \torc(K,G)$.

Since  $G/H$ is not torsion, there exists $y \in G \setminus \torc(H,G)$.
Consider any such $y$.  Consider the vector $\lambda(y) \in G_\Rat$ and note that $\lambda(y) \notin H_\Rat$.
Choose a vector subspace $V_y$ of $G_\Rat$ such that $H_\Rat \subseteq V_y \subset G_\Rat$, $\dim_\Rat(G_\Rat / V_y) = 1$, and $\lambda(y) \notin V_y$.
Let $H_y = \lambda^{-1}(V_y)$. Then $y \notin H_y$ and $H \subseteq H_y \subset G$; since $G/H_y$ is torsion-free (and hence free, since $G$ is finitely generated),
and since $S^{-1}( G/H_y ) \isom G_\Rat / V_y \isom \Rat$, we have $G/H_y \isom \Integ$; so $H_y \in \Heul(H,G)$. 
This shows that $\Heul(H,G) \neq \emptyset$ and $y \notin I$, where $I$ denotes the intersection of all elements of $\Heul(H,G)$.
This being true for each $y \in G \setminus \torc(H,G)$, we obtain $I \subseteq \torc(H,G)$, and the reverse inclusion is clear.
\end{proof}

\begin{lemma}  \label {jhgf09wk2mnddm9ijh2wedfgj34b}
Let $H_1, \dots, H_n$ be subgroups of an abelian group $G$.  Then
$$
\torc(H_1,G) \cap \dots \cap \torc(H_n,G) = \torc( H_1 \cap \dots \cap H_n , G ) .
$$
\end{lemma}

We leave it to the reader to verify Lemma \ref{jhgf09wk2mnddm9ijh2wedfgj34b}.

\begin{nothing*}  \label {8A76f5swaz23edafgwuvperol0ut83m395ff}
Let $B = \bigoplus_{i \in G} B_i$ be a domain graded by an abelian group $G$ and let $H$ be a subgroup of $\G(B)$.
Let $\Omega = \G(B)/H$ and let $\pi : \G(B) \to \Omega$ be the canonical homomorphism of the quotient group.
Let $\Beul$ denote the ring $B$  endowed with the $\Omega$-grading defined by
$\Beul = \bigoplus_{\omega \in \Omega} \Beul_\omega$, where $\Beul_\omega = \bigoplus_{ \pi(i)=\omega } B_i$ for each $\omega \in \Omega$.
We refer to this as ``the natural $\G(B)/H$-grading''.
It satisfies $\G( \Beul ) = \Omega$. Observe that $\Spec^1( \Beul ) = \Spec^1(B)$ and let us show that
\begin{equation} \label {765O043wsdvbnm32M91q0oijhbLvcl}
\pi\big( \overline{\bbM}( B, \pgoth ) \big) =  \overline{\bbM}( \Beul , \pgoth ) \quad \text{for every $\pgoth \in \Spec^1(B)$.}
\end{equation}
Indeed, let  $\pgoth \in \Spec^1(B)$. If $i \in {\bbM}( B, \pgoth )$ then $B_i \nsubseteq \pgoth$,
so $\Beul_{ \pi(i) }  \nsubseteq \pgoth$ (because $B_i \subseteq \Beul_{ \pi(i) }$) and hence $\pi(i) \in \bbM( \Beul, \pgoth )$.
Thus, $\pi \big(  {\bbM}( B, \pgoth ) \big) \subseteq \bbM( \Beul, \pgoth )$.
Conversely, consider $\omega \in \bbM( \Beul, \pgoth )$. Then $\Beul_\omega  = \bigoplus_{ \pi(i)=\omega } B_i$ is not included in $\pgoth$,
so there exists $i \in \G(B)$ such that $\pi(i)=\omega$ and $B_i \nsubseteq \pgoth$; then $i \in {\bbM}( B, \pgoth )$ and hence
$\omega \in \pi \big(  {\bbM}( B, \pgoth ) \big)$.
This shows that $\pi \big( {\bbM}( B, \pgoth ) \big) = \bbM( \Beul, \pgoth )$, and it follows that
$\pi \big(  \overline{\bbM}( B, \pgoth ) \big) = \overline\bbM( \Beul, \pgoth )$.
So \eqref{765O043wsdvbnm32M91q0oijhbLvcl} is proved.
({\it Caution:} Assertion \eqref{765O043wsdvbnm32M91q0oijhbLvcl} implies that $\pi\big( \bG(B) \big) \subseteq \bG(\Beul)$,
but equality does not necessarily hold.)
\end{nothing*}

The {\it rank\/} of an abelian group $G$ is the dimension of the $\Rat$-vector space $\Rat \otimes_\Integ G$.
See Rem.\ \ref{8765r984f5f7637w} for the definition of $\ML(B)$.

\begin{theorem}   \label {p0987uhmnbvccjiuseyt428}
\it Let $\bk$ be a field of characteristic $0$ and $B$ a normal affine $\bk$-domain graded over $\bk$ by an abelian group $G$.
Assume that $\G(B)/\bG(B)$ is not torsion and define $r = \rank\big( \G(B) / \bG(B) \big)$ and $W = \torc(\bG(B),\G(B))$.
\begin{enumerata}

\item There exists $\pgoth \in \Spec^1(B)$ such that $\G(B)/\overline{\bbM}( B, \pgoth )$ is not torsion.
Consider such a $\pgoth$ and define
$\Heul_\pgoth = \Heul\big( \overline{\bbM}( B, \pgoth ) , \G(B) \big)$ and $T_\pgoth = \torc( \overline{\bbM}( B, \pgoth ) , \G(B) )$.
Then the following hold.
\begin{enumerata}

\item $\Heul_\pgoth \neq \emptyset$

\item For each $H \in \Heul_\pgoth$, there exists $D \in \lnd(B)$ such that $\ker(D) = B^{(H)}$.

\item There exists a subset $\Delta_\pgoth$ of $\lnd(B)$ such that $\bigcap_{D \in \Delta_\pgoth} \ker(D) = B^{ ( T_\pgoth ) }$.

\end{enumerata}

\item There exists a subset $\Delta$ of $\lnd(B)$ such that $\bigcap_{D \in \Delta} \ker(D) = B^{ (W) }$.

\item There exists a field $K$ such that
$$
\text{$\ML(B) \subseteq B^{(W)} \subseteq K \subseteq \Frac(B)$ \ \ and \ \ $\Frac(B) = K^{(r)}$.}
$$
In particular, $\trdeg( B : \ML(B) ) \ge r$.

\item $B$ is non-rigid, and if $G$ is torsion-free then $\Xscr(B) = \bbT(B)$.

\end{enumerata}
\end{theorem}

\begin{proof}
Since $B$ is $\bk$-affine and the grading is over $\bk$, $B$ is finitely generated as a $B_0$-algebra and $\G(B)$ is finitely generated.
Consequently, $r$ is finite (so $r \in \Nat \setminus \{0\}$).

(a) If $\Spec^1(B) = \emptyset$ then $\bG(B) = \G(B)$ by definition, contradicting the assumption that $\G(B)/\bG(B)$ is not torsion.
So $\Spec^1(B) \neq \emptyset$.

By Prop.\  \ref{89fDCN584uKHJu7u59034uf}, the nonempty set $\setspec{ \overline{\bbM}( B, \pgoth ) }{ \pgoth \in \Spec^1(B) }$ is finite.
So we can choose $\pgoth_1, \dots, \pgoth_n \in \Spec^1(B)$ such that $\bG(B) =  \bigcap_{i=1}^n\overline{\bbM}( B, \pgoth_i )$.
There must exist $i \in \{1,\dots,n\}$ such that $\G(B)/\overline{\bbM}( B, \pgoth_i )$ is not torsion.
This proves the first claim in assertion (a).

Consider $\pgoth \in \Spec^1(B)$ such that $\G(B)/\overline{\bbM}( B, \pgoth )$ is not torsion and define
$\Heul_\pgoth$ and $T_\pgoth$ as in the statement of (a).
Lemma \ref{1hfcvbnn92m3n456789df7nm534jh5f6d78} implies that $\Heul_\pgoth \neq \emptyset$ (which proves (a-i)) and
\begin{equation}  \label {1Adcdb9bdNTs438ywte6v2w89unvBNZjxmnbevCe49831}
\textstyle \bigcap_{ H \in \Heul_\pgoth } H  = T_\pgoth .
\end{equation}
Let $H \in \Heul_\pgoth$. Let $\pi : \G(B) \to \G(B)/H$ be the canonical homomorphism.
The definition of $\Heul_\pgoth$ implies that $\overline{\bbM}( B, \pgoth ) \subseteq H = \ker(\pi)$ and $\G(B)/H \isom \Integ$.
Let $\Beul(H)$ denote the ring $B$ endowed with the natural $\G(B)/H$-grading (see paragraph \ref{8A76f5swaz23edafgwuvperol0ut83m395ff})
and note that $\G( \Beul(H) ) = \G(B)/H \isom \Integ$.
Since $\Spec^1(B) = \Spec^1( \Beul(H) )$, we have $\pgoth \in \Spec^1( \Beul(H) )$; thus,
$$
\textstyle
\bG( \Beul(H) ) \ = \ \bigcap_{ \qgoth \in \Spec^1( \Beul(H) ) } \overline{\bbM}( \Beul(H), \qgoth ) \ \subseteq \  
\overline{\bbM}( \Beul(H), \pgoth )  = \pi\big( \overline{\bbM}( B, \pgoth ) \big) = 0 
$$
where the penultimate equality follows from \eqref{765O043wsdvbnm32M91q0oijhbLvcl}.
Since $\G( \Beul(H) ) \isom \Integ$ and $\bG( \Beul(H) ) = 0$,
Cor.\ \ref{i8976r53g47c8ru3yd3g23ev5627fu} implies that there exists $D \in \hlnd( \Beul(H) )$ such that $\ker( D ) = \Beul(H)_0 = B^{(H)}$.
Since $D \in \lnd(B)$, this proves (a-ii).
Since $H$ is a proper subgroup of $\G(B)$, we have $B^{(H)} \neq B$ and hence $D \neq 0$. So $B$ is non-rigid (this is claimed in (d)).

For each $H \in \Heul_\pgoth$, choose $D_H \in \lnd(B)$ such that $\ker(D_H) = B^{(H)}$ (possible by (a-ii)).
Define $\Delta_\pgoth = \setspec{ D_H }{ H \in \Heul_\pgoth }$.
Then $\bigcap_{D \in \Delta_\pgoth} \ker(D) = \bigcap_{H \in \Heul_\pgoth} B^{(H)} = B^{ ( T_\pgoth ) }$,
the last equality by \eqref{1Adcdb9bdNTs438ywte6v2w89unvBNZjxmnbevCe49831}.
This proves (a-iii) and completes the proof of (a).

(b) Recall that we chose $\pgoth_1, \dots, \pgoth_n \in \Spec^1(B)$ such that $\bG(B) =  \bigcap_{i=1}^n\overline{\bbM}( B, \pgoth_i )$,
and that we noted that $I \neq \emptyset$, where 
$$
I = \setspec{ i }{ \text{$1 \le i \le n$ and $\G(B) /  \overline{\bbM}( B, \pgoth_i )$ is not torsion} } .
$$
Since $\bG(B) =  \bigcap_{i=1}^n\overline{\bbM}( B, \pgoth_i )$,
Lemma \ref{jhgf09wk2mnddm9ijh2wedfgj34b} gives $W = \bigcap_{i=1}^n \torc\big( \overline{\bbM}( B, \pgoth_i ) , \G(B) \big)$.
We have
$\torc\big( \overline{\bbM}( B, \pgoth_i ) , \G(B) \big) = T_{\pgoth_i}$ for each $i \in I$
and $\torc\big( \overline{\bbM}( B, \pgoth_i ) , \G(B) \big) = \G(B)$ for each $i \in \{1,\dots,n\} \setminus I$, so 
\begin{equation} \label {u6f5c263F7v8b92ne8fbvcx23iawekeo09}
\textstyle W = \bigcap_{i = 1}^n \torc\big( \overline{\bbM}( B, \pgoth_i ) , \G(B) \big) = \bigcap_{i \in I} T_{\pgoth_i} .
\end{equation}
For each $i \in I$, consider a subset $\Delta_{\pgoth_i}$ of $\lnd(B)$ such that $\bigcap_{D \in \Delta_{\pgoth_i}} \ker(D) = B^{ ( T_{\pgoth_i} ) }$
($\Delta_{\pgoth_i}$ exists by (a-iii)).
Let $\Delta = \bigcup_{i \in I} \Delta_{\pgoth_i}$. Then
$$
\textstyle
\bigcap_{D \in \Delta} \ker(D)
= \bigcap_{i \in I} \bigcap_{D \in \Delta_{\pgoth_i}} \ker(D)
= \bigcap_{i \in I} B^{ ( T_{\pgoth_i} ) }
= B^{ (W) } 
$$
where the last equality follows from \eqref{u6f5c263F7v8b92ne8fbvcx23iawekeo09}. This proves (b).

(c) Let $R$ denote $B$ endowed with the natural $\G(B)/W$-grading (see paragraph \ref{8A76f5swaz23edafgwuvperol0ut83m395ff}).
Since  $\G(R) = \G(B)/W$ is finitely generated and torsion-free,
and $\rank( \G(B)/W ) = \rank( \G(B)/\bG(B) ) = r \in \Nat \setminus \{0\}$, we have $\G(R) \isom \Integ^r$.
Lemma \ref{i8765redfvbnki8765rfghytrew123456789iuhv}
implies that there exists a field $K$ such that $R_0 \subseteq K \subseteq \Frac(R)$ and $\Frac(R) = K^{(r)}$.
Since $R_0 = B^{(W)} \supseteq \ML(B)$ by (b), we have $\ML(B) \subseteq K$. Since $\Frac(B)=\Frac(R)$, (c) follows.

(d) We already proved that $B$ is non-rigid.
If moreover $G$ is torsion-free then Lemma \ref{876543wsdfgxhnm290cAo} implies that $\hlnd(B) \neq \{0\}$,
so Thm \ref{8237e9d1983hdjyev93}(c)  gives $\Xscr(B) = \bbT(B)$.
\end{proof}


\noindent{\bf Data availability: }
Not applicable.

\noindent{\bf Conflicts of interest: }
The author has no conflicts of interest to declare that are relevant to this article.

\end{document}